\documentclass[a4paper,usenames,dvipsnames]{amsart} 

\usepackage[english]{babel}
\usepackage[utf8x]{inputenc}
\usepackage[T1]{fontenc}

\usepackage[a4paper, top=3cm, bottom=2cm, left=3cm, right=3cm, marginparwidth=1.75cm]{geometry}

\usepackage{amsmath, amsthm}
\usepackage{amsfonts}
\usepackage{amssymb}
\usepackage{xcolor}
\usepackage[shortlabels]{enumitem}
\usepackage{graphicx}
\usepackage{bm}
\usepackage[colorinlistoftodos, textsize=tiny]{todonotes}
\usepackage[colorlinks=true, allcolors=blue]{hyperref}
\usepackage{caption}
\usepackage{xfrac}
\usepackage{mathtools}
\usepackage{verbatim}
\usepackage{bm}
\usepackage{array}
\usepackage{comment}
\usepackage{tikz}
	\usetikzlibrary{shapes.multipart}
	\usetikzlibrary{calc}
\usepackage{tikz-qtree}
\usepackage{forest}
\usepackage{hyperref}
\usepackage[foot]{amsaddr} 
\usepackage{mdframed}

\newtheorem{theorem}{Theorem}
\theoremstyle{plain}
\newtheorem{proposition}[theorem]{Proposition}
\newtheorem{conclusion}[theorem]{Conclusion}
\newtheorem{remark}[theorem]{Remark}
\newtheorem{lemma}[theorem]{Lemma}

\newtheorem{corollary}[theorem]{Corollary}
\numberwithin{theorem}{section}

\newmdtheoremenv{notation2}[theorem]{Convention}

\theoremstyle{definition}
\newtheorem{example}[theorem]{Example}

\newtheorem{notation3}[theorem]{Convention}
\newenvironment{notation4}
  {\begin{mdframed}\begin{notation3}}
  {\end{notation3}\end{mdframed}}
  
\numberwithin{equation}{section}

\theoremstyle{remark}
\newtheorem*{remark*}{Remark}

\newcommand{\Q}{\mathbb{Q}}
\newcommand{\Z}{\mathbb{Z}}
\newcommand{\N}{\mathbb{N}}
\newcommand{\R}{\mathbb{R}}
\newcommand{\C}{\mathbb{C}}

\newcommand{\OK}[1][K]{\ensuremath{\mathcal{O}_{#1}}} 
\newcommand{\OKPlus}[1][K]{\ensuremath{\mathcal{O}_{#1}^+}}
\newcommand{\UK}[1][K]{\ensuremath{\mathcal{U}_{#1}}} 
\newcommand{\UKPlus}[1][K]{\ensuremath{\mathcal{U}_{#1}^+}} \newcommand{\OKdva}[1][K]{\ensuremath{\mathcal{O}_{#1}}^2} % lattice O_K^2
\newcommand{\OKctv}[1][K]{\ensuremath{\mathcal{O}_{#1}}^{\square}} % squares O_K^2
\newcommand{\UKctv}[1][K]{\ensuremath{\mathcal{U}_{#1}}^{\square}} % squares U_K^2

\newcommand{\OF}{\OK[F]} 

\newcommand{\UFctv}{\UKctv[F]}
\newcommand{\UFPlus}{\UKPlus[F]}
\newcommand{\OFPlus}{\OKPlus[F]}

\newcommand{\BQ}[2]{\ensuremath{\Q\big(\sqrt{#1}, \sqrt{#2}\big)}}

\newcommand{\Tr}[2]{\mathrm{Tr}_{#1}\left(#2\right)} % trace
\newcommand{\norm}[2]{\mathcal{N}_{#1}\left(#2\right)}
\newcommand{\M}{\mathcal{M}} 
\renewcommand{\S}{\mathcal{S}}%
\newcommand{\T}{\mathcal{T}}%
\newcommand{\U}{\mathcal{U}}
\newcommand{\csqrt}[1]{\left \lceil\!{\sqrt{#1}}\mspace{1 mu}\right \rceil  } % 
\newcommand{\cm}{\csqrt{m}}
\newcommand{\cs}{\csqrt{s}}
\newcommand{\ct}{\csqrt{t}}
\newcommand{\cu}{\csqrt{u}}
\newcommand{\csqrto}[1]{\csqrt{#1}^{\mathrm{odd}}} % 
\newcommand{\cmo}{\csqrto{m}}
\newcommand{\cso}{\csqrto{s}}
\newcommand{\cuo}{\csqrto{u}}

\newcommand{\abs}[1]{\left|#1\right|}
\newcommand{\ve}{\varepsilon}
\newcommand{\uqf}[1]{\langle #1 \rangle} % 
\newcommand{\bqf}[2]{\uqf{#1}\bot\uqf{#2}} % 
\newcommand{\tqf}[3]{\uqf{#1}\bot\uqf{#2}\bot\uqf{#3}} % 
\newcommand{\p}{\mathfrak{p}} % 
\newcommand{\vct}[1]{\bm{#1}} % 
\newcommand{\e}{\vct{e}} %
\newcommand{\f}{\vct{f}} %
\newcommand{\indx}[1]{\mathcal{I}_{\vct{#1}}} % 

\DeclareMathOperator{\im}{Im}
\DeclareMathOperator{\Ker}{Ker}

\renewcommand{\arraystretch}{1.5}

\newcommand{\splitatcommas}[1]{%
  \begingroup
  \begingroup\lccode`~=`, \lowercase{\endgroup
    \edef~{\mathchar\the\mathcode`, \penalty0 \noexpand\hspace{0pt plus 1em}}%
  }\mathcode`,="8000 #1%
  \endgroup
}% 

\newcommand{\uv}[1]{``#1''} % 
\newcommand{\conds}[1]{{\textcolor{violet}{#1}}} %podminky -- NECHAT

\title[Nonuniversality of ternary forms over biquadratic fields]{There are no universal ternary quadratic forms over biquadratic fields}

\author[J. Kr\'asensk\'y]{Jakub Kr\'asensk\'y$^1$}
\author[M. Tinkov\'{a}]{Magdal\'{e}na Tinkov\'{a}$^1$}
\address{$^1$Charles University, Faculty of Mathematics and Physics, Department of Algebra,
Sokolovsk\'{a} 83, 18600 Praha 8, Czech Republic}
\author[K. Zemkov\'{a}]{Krist\'{y}na Zemkov\'{a}$^2$}
\address{$^2$Fakult\"at f\"ur Mathematik, Technische Universit\"at Dortmund, D-44221 Dortmund,
Germany}
\email{krasensky@seznam.cz, tinkova.magdalena@gmail.com, zemk.kr@gmail.com}

\date{\today}

\thanks{J. K., M. T. and K. Z.  were supported by the Charles University, project GA UK No.\ 1298218. J. K. and M. T. were supported by Czech Science Foundation (GA\v{C}R), grant 17-04703Y, by Charles University Research Centre program UNCE/SCI/022, and by the project SVV-2017-260456. K. Z. was supported by DFG project HO 4784/2-1.
}

\begin{document}

\begin{abstract}
We study totally positive definite quadratic forms over the ring of integers $\mathcal{O}_K$ of a totally real biquadratic field $K=\mathbb{Q}(\sqrt{m}, \sqrt{s})$. We restrict our attention to classical forms (i.e., those with all non-diagonal coefficients in $2\mathcal{O}_K$) and prove that no such forms in three variables are universal (i.e., represent all totally positive elements of $\mathcal{O}_K$). Moreover, we show the same result for totally real number fields containing at least one nonsquare totally positive unit and satisfying some other mild conditions. These results provide further evidence towards Kitaoka's conjecture that there are only finitely many number fields over which such forms exist.

One of our main tools are additively indecomposable elements of $\mathcal{O}_K$; we prove several new results about their properties.
\end{abstract}

\maketitle

\section{Introduction}
Several generations of number theorists have been interested in quadratic forms. One of the most beautiful topics in this area is the study of universal quadratic forms with integral coefficients. It started with Lagrange's famous theorem, which says that every nonnegative rational integer can be expressed as a sum of four squares; in modern terminology, we talk about the universality of the form $x^2+y^2+z^2+w^2$ over the ring $\Z$. Later, all universal forms over $\Z$ in four variables were characterized by Ramanujan \cite{Ra}. In particular, no ternary (i.e., in three variables) quadratic form with integer coefficients can represent all positive elements of $\Z$. 

Considering, instead of $\Z$, the ring of algebraic integers $\OK$ for a totally real algebraic extension $K$ of $\Q$, many directions of research appear. In \cite{Ma}, Maass showed that the sum of three squares is universal over $\mathcal{O}_{\Q(\sqrt{5})}$, which was followed by the result of Siegel \cite{Si} saying that the sum of any number of squares is universal only over $\Z$ and $\mathcal{O}_{\Q(\sqrt{5})}$. Moreover, Hsia, Kitaoka and Kneser \cite{HKK} proved that in any given number field, there always exists a universal quadratic form. This naturally leads to the question how many variables this form must have.

Blomer and Kala \cite{BK, Ka} have shown that for any given $N$, there can be always found a quadratic number field in which every universal quadratic form has at least $N$ variables. The same result can be obtained for multiquadratic fields (Kala and Svoboda \cite{KS}) and cubic fields (Yatsyna \cite{Ya}). On the other hand, one can ask, for a given field, what is the least possible number of variables of a universal quadratic form; \v{C}ech, Lachman, Svoboda and two present authors \cite{CLSTZ} started to examine this question for biquadratic fields. Kitaoka conjectured that there are only finitely many fields which admit universal ternary quadratic forms. This idea was supported by the result of Chan, Kim and Raghavan \cite{CKR}, who proved that $\Q(\sqrt{2}),\Q(\sqrt{3})$ and $\Q(\sqrt{5})$ are the only quadratic fields with a universal classical totally positive definite ternary quadratic form. Inspired by this statement, we prove the following theorem. 

\begin{theorem}\label{Thm:Main}
For a totally real biquadratic field $K$, there is no universal classical totally positive definite ternary quadratic form over $\OK$.  
\end{theorem} 

The proof is divided into Sections~\ref{Sec:NonspecialCases} and~\ref{Sec:Special cases}. This theorem is, to the extent of our knowledge, the first result about universality of ternary forms in fields of even degree different from two. For all totally real number fields of odd degree, a proof of nonexistence of universal ternary forms can be found in \cite{EK}. In fact, one part of our proof can be applied to prove nonuniversality of ternary forms in a broad class of number fields of arbitrary degree:

\begin{theorem}\label{Thm:General}
Let $K$ be a totally real field, $\sqrt2\notin K$, which contains a nonsquare totally positive unit $\ve$ such that $2\ve$ is not a square in $K$. Then there is no universal classical totally positive definite ternary quadratic form over $\OK$. 
\end{theorem}

The proof is contained in Subsection \ref{Subsec:CaseIInonsquare}. For more results on quadratic forms, see for example \cite{BK2,De, Ki1,Ki2,Ro,Sa}.

In totally real number fields, the question of universality of quadratic forms is closely related to the study of indecomposable integers. These are exactly those elements of $\OKPlus$ (the set of totally positive algebraic integers of $K$) which cannot be written as a sum of two elements of $\OKPlus$. In a certain sense, it is difficult to represent them by quadratic forms, which is demonstrated in the works of Blomer, Kala or Yatsyna. Despite this fact, there is not much known about them. In quadratic fields, they were characterized by Perron \cite{Pe}, Dress and Scharlau \cite{DS}, and their norms were studied by several authors \cite{JK, Ka2, TV}. %JK (Jang--Kim, Magda (TV), Víťa)
Some general statements can be found in the work of Brunotte \cite{Bru}. Considering biquadratic fields, we can draw on the results of \cite{CLSTZ}, which we extend in this paper. In particular, we prove the following theorem. 

\begin{theorem} \label{Thm:Inde}
Let $1<m<s<t$ be square-free integers such that $\BQ{m}{s}$ is a biquadratic field containing $\sqrt{t}$. For all $u\in\{m,s,t\}$, set
\[
    \U=\left\{
                \begin{array}{ll}
                  \cu+\sqrt{u} \qquad\text{if } u\equiv 2,3\!\!\!\pmod{4},\\
                  \frac{\cuo+\sqrt{u}}{2}\qquad\text{if } u\equiv 1\!\!\!\pmod{4},
                \end{array}
                \right.              
\]
where $\cuo$ is the smallest odd integer greater than $\sqrt{u}$. Then
\begin{enumerate}
\item $\M$ and $\S$ are indecomposable in $\BQ{m}{s}$, \label{Thm:MaS} 
\item $2\M-1$, if totally positive, is indecomposable in $\BQ{m}{s}$, \label{Thm:2M-1}
\item $\T$ can decompose in $\BQ{m}{s}$. \label{Thm:T}  
\end{enumerate}  
\end{theorem}

We prove the theorem in Subsection~\ref{Subsec:NewIndeElms}. In particular, the statement of part \eqref{Thm:MaS} is covered by Propositions~\ref{Prop:IndecM} and~\ref{Prop:IndecS}, and part \eqref{Thm:2M-1} by Proposition~\ref{Proposition:2M-1indecomposable}. Part \eqref{Thm:T} follows from Example~\ref{Exm:Tdecomp}; note that this example provides an indecomposable algebraic integer from a quadratic field which decomposes in a biquadratic field, thus solving an open question from \cite{CLSTZ}.

The indecomposability of $\M$ stated in Theorem~\ref{Thm:Inde}\eqref{Thm:MaS} is one of the keystones in the proof of Theorem~\ref{Thm:Main}. Although the statement of Theorem \ref{Thm:Inde}\eqref{Thm:MaS} might seem expectable at first glance, the  indecomposability of $\M$ and $\S$ in \emph{every} biquadratic field is not immediate, as indicates part \eqref{Thm:T} of the theorem. The proof is rather technical; one of the difficulties arises from the freedom for possible decompositions, which is much greater than in the quadratic case: Instead of two coefficients and two integral bases, one has to consider four coefficients and five different types of bases. Similar obstacles occur in proofs of most of our lemmas, making them more difficult than their quadratic counterparts. This requires a careful approach and case distinction -- it can be treacherous even in the quadratic case, as illustrated with the overlooked exceptional field in the paper \cite{CKR} (see Appendix~\ref{App:sqrt10} where this omission is handled) -- and also several new ideas.  Moreover, any condition of the type \uv{$\sqrt{n}\notin K$} (e.g., $\sqrt{2}\notin K$) produces an infinite family of exceptional biquadratic fields instead of just excluding the field $\Q(\sqrt{n})$.

\bigskip

The paper is organized as follows: In Section~\ref{Sec:Prelim} we review some basic facts about biquadratic number fields, quadratic forms and indecomposable integers in totally real number fields. Since the proof of our main results consists of several parts, Section~\ref{Sec:IdeaOfProof} provides a brief outline of the proof. In particular, we explain the reason for the subsequent case distinction.

Section~\ref{Sec:BeautifulAndUgly} contains some preparatory statements, mostly about algebraic integers in biquadratic number fields. Above all, we are concerned with indecomposability;  we prove Theorem~\ref{Thm:Inde} in Subsection~\ref{Subsec:NewIndeElms}. Furthermore, we show that some elements of $\OK$ derived from the elements appearing in this theorem are not squares in $K$, and we also study decompositions of small rational integers in $K$.  At the end of the section, we look more closely at representability of some elements by unary quadratic forms and at splitting of quadratic forms and corresponding lattices.

Section~\ref{Sec:NonspecialCases} provides the main part of our proof -- the method introduced here can be applied to most biquadratic fields. The remaining cases are solved 
in Section~\ref{Sec:Special cases}; this time, we use several different methods to deal with arising difficulties, including the knowledge of some indecomposable integers in $K$ and the method of escalation. Some computations used in proofs in Subsections~\ref{Subsec:2}, \ref{Subsec:5} and~\ref{Subsec:Specific} were performed by programs written in Mathematica.

The paper is concluded by three appendixes: Appendix~\ref{App:sqrt10} recovers the problematic quadratic case of $\Q(\sqrt{10})$ from \cite{CKR}. In Appendix~\ref{Subsec:Units}, we provide some insight into the behavior of totally positive units from a quadratic subfield when considering them as elements of the biquadratic field. Finally, to make the orientation in the different branches of the proof of Theorem~\ref{Thm:Main} easier, Appendix~\ref{App} contains a sketch of its tree structure in the main cases. 
%----------------------------------------------------------------------------------------------

\section{Preliminaries} \label{Sec:Prelim}

\subsection{Algebraic integers}
Let $K$ be a totally real number field, i.e., a number field where all embeddings $\sigma$ of $K$ into $\C$ (including the identity) actually map $K$ into $\R$. We say that $\alpha\in K$ is \emph{totally positive}, denoted by $\alpha\succ 0$, if $\sigma(\alpha)>0$ for each embedding $\sigma$. Let $\OK$ be the ring of algebraic integers in $K$; the subset of totally positive elements of $\OK$ will be denoted by $\OKPlus$. Obviously it is closed under addition and multiplication and contains the set of all squares $\OKctv$. The notation $\alpha\succcurlyeq 0$ means that $\alpha$ is either totally positive or $0$, and we use $\alpha\succ \beta$ (resp.\ $\alpha\succcurlyeq \beta$) to denote $\alpha-\beta\succ 0$ (resp.\ $\alpha-\beta\succcurlyeq 0$).  We use symbols $\Tr{K/\Q}{\alpha}$ and $\norm{K/\Q}{\alpha}$ to denote the \emph{trace} and the \emph{norm} of $\alpha$, i.e., 
\[
\Tr{K/\Q}{\alpha}=\sum_{\sigma}\sigma(\alpha) \qquad \text{and} \qquad \norm{K/\Q}{\alpha}=\prod_{\sigma}\sigma(\alpha) 
\] 
where the sum and the product run over all the embeddings $\sigma$ of $K$ into $\C$. Obviously, if $\alpha,\beta\in\OK$ satisfy $\alpha\preccurlyeq\beta$, then $\Tr{K/\Q}{\alpha}\leq\Tr{K/\Q}{\beta}$; if, moreover, both of $\alpha$ and $\beta$ are totally positive, then also $\norm{K/\Q}{\alpha}\leq\norm{K/\Q}{\beta}$.
Furthermore, norm has the following nice property: If $K$ is of degree $N$ over $\Q$ and  $\alpha,\beta\in\OKPlus$, then 
\begin{equation} \label{Eq:NormSum}
\sqrt[N]{\norm{K/\Q}{\alpha+\beta}}\geq\sqrt[N]{\norm{K/\Q}{\alpha}}+\sqrt[N]{\norm{K/\Q}{\beta}};
\end{equation}
this is essentially the generalized Hölder's inequality. Equality happens if and only if $\frac\alpha\beta \in \Q$.

We denote by $\UK$ the set of units of $\OK$, i.e., the set of algebraic integers of norm $\pm1$; furthermore, we write $\UKPlus$ (resp.\ $\UKctv$)  for the subset of totally positive units (resp.\ for the subset of squares of units).

\subsection{Biquadratic fields}

Let $p,q>1$ be two different square-free integers, put $K=\BQ{p}{q}$ and $r=\frac{pq}{\gcd(p,q)^2}$; such a field $K$ is called a (totally real) \emph{biquadratic} field. Throughout the paper, $K$ denotes only such a field.  It has degree $4$ and one possible $\Q$-vector space basis is $(1,\sqrt{p},\sqrt{q},\sqrt{r})$; we can find three quadratic subfields in $K$, namely $\Q(\sqrt{p})$, $\Q(\sqrt{q})$ and $\Q(\sqrt{r})$. There are four embeddings of $K$ into $\C$: If $\alpha=x+y\sqrt{p}+z\sqrt{q}+w\sqrt{r}\in K$, then 
\[\begin{aligned}
\sigma_1(\alpha)&=x+y\sqrt{p}+z\sqrt{q}+w\sqrt{r}, \\
\sigma_2(\alpha)&=x-y\sqrt{p}+z\sqrt{q}-w\sqrt{r}, \\
\sigma_3(\alpha)&=x+y\sqrt{p}-z\sqrt{q}-w\sqrt{r}, \\
\sigma_4(\alpha)&=x-y\sqrt{p}-z\sqrt{q}+w\sqrt{r}.
\end{aligned}\] 
Note that biquadratic fields are Galois extensions of $\Q$, i.e., all $\sigma_i$'s are automorphisms of $K$.

Depending  on $p,q\pmod{4}$, after possibly interchanging the role of $p$, $q$ and $r$, every case can be converted into one of the five cases listed below. The importance of distinguishing these five types of fields lies in the fact that it determines the \emph{integral basis}, i.e., basis of $\OK$ regarding it as a $\Z$-module (see \cite[Section~8]{Ja} and \cite[Theorem~2]{Wi}): 

{\renewcommand{\arraystretch}{1.5} 
\[\begin{array}{cclcl}
    (\text{B}1) & & p\equiv 2\!\!\!\pmod{4},\ q\equiv 3\!\!\!\pmod{4}, & & \left(1, \sqrt{p}, \sqrt{q}, \frac{\sqrt{p}+\sqrt{r}}2\right)\\
    (\text{B}2) & & p\equiv 2\!\!\!\pmod{4},\ q\equiv 1\!\!\!\pmod{4}, & & \left(1, \sqrt{p}, \frac{1+\sqrt{q}}2, \frac{\sqrt{p}+\sqrt{r}}{2} \right)\\
    (\text{B}3) & & p\equiv 3\!\!\!\pmod{4},\ q\equiv 1\!\!\!\pmod{4}, & & \left(1, \sqrt{p}, \frac{1+\sqrt{q}}2, \frac{\sqrt{p}+\sqrt{r}}{2}\right)\\
    (\text{B}4) & & p\equiv 1\!\!\!\pmod{4},\ q\equiv 1\!\!\!\pmod{4}, \text{ and } & & \\
        & (a) &
        \gcd(p,q)\equiv1\!\!\!\pmod{4}, \text{ or } & & \left(1, \frac{1+\sqrt{p}}2, \frac{1+\sqrt{q}}2, \frac{1+\sqrt{p}+\sqrt{q}+\sqrt{r}}{4}\right)\\
        & (b) & \gcd(p,q)\equiv3\!\!\!\pmod{4} & & \left(1, \frac{1+\sqrt{p}}2, \frac{1+\sqrt{q}}2, \frac{1-\sqrt{p}+\sqrt{q}+\sqrt{r}}{4}\right) \\
\end{array}\]} 

Note that in all these cases, we have $p\equiv r\pmod{4}$, so in cases (B1), (B2) and (B3), $p$ and $r$ are interchangeable. In case (B4), all of $p, q, r$ are interchangeable. The field discriminants are (B1) $64pqr$, (B2) $16pqr$, (B3) $16pqr$, (B4) $pqr$; these are actually just products of the discriminants of the quadratic subfields.

 For convenience of the reader, we include a more explicit form of algebraic integers depending on the integral bases.
 
\begin{remark}\label{Rem:BasisElements}
An element $\alpha\in K$ belongs to $\OK$ if and only if there exist $a,b,c,d\in\Z$ such that
 \begin{itemize}[labelindent=\parindent,leftmargin=5.5em]
     \item[\emph{(B1)}] $\alpha=a+\frac{b}{2}\sqrt{p}+c\sqrt{q}+\frac{d}{2}\sqrt{r}$ with $b\equiv d\pmod2$,
     \item[\emph{(B2), (B3)}] $\alpha=\frac{a}{2}+\frac{b}{2}\sqrt{p}+\frac{c}{2}\sqrt{q}+\frac{d}{2}\sqrt{r}$ with $a\equiv c\pmod2$ and $b\equiv d\pmod2$, 
     \item[\emph{(B4}a\emph{)}]  
         \begin{itemize}
            \item[] \hspace{-0.9cm} either $\alpha=\frac{a}{2}+\frac{b}{2}\sqrt{p}+\frac{c}{2}\sqrt{q}+\frac{d}{2}\sqrt{r}$ with $a+b+c+d\equiv0\pmod2$,
            \item[] \hspace{-0.9cm} or $\alpha=\frac{a}{4}+\frac{b}{4}\sqrt{p}+\frac{c}{4}\sqrt{q}+\frac{d}{4}\sqrt{r}$ with $a,b,c,d$ odd and $a+b+c+d\equiv0\pmod4$.
        \end{itemize}
     \item[\emph{(B4}b\emph{)}] 
        \begin{itemize}
            \item[] \hspace{-0.9cm} either $\alpha=\frac{a}{2}+\frac{b}{2}\sqrt{p}+\frac{c}{2}\sqrt{q}+\frac{d}{2}\sqrt{r}$ with $a+b+c+d\equiv0\pmod2$,
            \item[] \hspace{-0.9cm} or $\alpha=\frac{a}{4}+\frac{b}{4}\sqrt{p}+\frac{c}{4}\sqrt{q}+\frac{d}{4}\sqrt{r}$ with $a,b,c,d$ odd and $a+b+c+d\equiv2\pmod4$.
        \end{itemize}
 \end{itemize}
\end{remark} 

\begin{notation4}\label{Conv}
Throughout this article, we use three different triples of letters for the specification of a biquadratic number field, each of them implicitly having a different property. Regardless of the notation, we always assume all the three numbers to be square-free, and any of the three numbers to be equal to the product of the remaining two divided by the second power of their greatest common divisor. Then the biquadratic field is generated by any two of these three elements. We use the following notation:
\begin{itemize}
    \item $p,q,r$: this triple always satisfies one of the possible congruences above (and hence it is closely connected to the basis),
    \item $m,s,t$: always $m<s<t$,
    \item $n_1,n_2,n_3$: this triple does not carry any additional information.
\end{itemize} 
Moreover, to describe an element of $\OK$, we usually use either $a,b,c,d$ for integers or $x,y,z,w$ for rational numbers; a generic element of $\OK$ can be then written, e.g., as $\frac{a}{2}+\frac{b}{2}\sqrt{p}+\frac{c}{2}\sqrt{q}+\frac{d}{2}\sqrt{r}$ or as $x+y\sqrt{n_1}+z\sqrt{n_2}+w\sqrt{n_3}$. 
\end{notation4}
%\vspace{1.5em}
\smallskip

Bearing this notation in mind, we can give the following necessary condition for the totally positive elements of $\OK$. This lemma was derived in \cite[Lemma~3.1]{CLSTZ}. 

\begin{lemma} \label{Lemma:TotPositiveInequalities}
If $x+y\sqrt{n_1}+z\sqrt{n_2}+w\sqrt{n_3} \in \OKPlus$ for some $x,y,z,w\in\Q$, then $x > 0$, $x > |y|\sqrt{n_1}$, $x > |z|\sqrt{n_2}$, $x>|w|\sqrt{n_3}$.
\end{lemma}

Note that it implies that, for any fixed $K$, there are only finitely many totally positive integers with a given trace; this also means that for a fixed $\alpha\in\OKPlus$, the equation $\alpha=\beta+\gamma$ has at most finitely many solutions $\beta,\gamma\in\OKPlus$, and they can straightforwardly be found by a computer program. In particular, it is routine to check whether a given element $\alpha$ of a fixed field $K$ is indecomposable (see Subsection \ref{subsec:indecom}).

In a few proofs, we make use of writing $m, s, t$ as products $m=s_0t_0$, $s=m_0t_0$, $t=m_0s_0$ where $m_0$, $s_0$ and $t_0$ are pairwise coprime square-free numbers, equal to $\gcd(s,t)$, $\gcd(m,t)$ and $\gcd(m,s)$, respectively. Note that the inequality $m<s<t$ translates to $m_0>s_0>t_0$, and in basis (B4) we have $m_0 \equiv s_0 \equiv t_0 \pmod{4}$.

\subsection{Quadratic forms} 
The expression 
\[
Q(x_1,x_2,\ldots,x_n)=\sum_{1\leq i\leq j\leq n} a_{ij}x_ix_j,
\]
where $a_{ij}\in\OK$, is called an $n$-ary quadratic form over $\OK$. We often think of $Q$ as acting on vectors from the lattice $\OK^n$ and correspondingly write $Q(\vct{x})$ instead of $Q(x_1,x_2,\ldots,x_n)$. Given a quadratic form $Q$, we can construct a symmetric bilinear form $B_Q(\vct{x},\vct{y})=\frac12\bigl(Q(\vct{x}+\vct{y})-Q(\vct{x})-Q(\vct{y})\bigr)$ such that $Q(\vct{x})=B_Q(\vct{x},\vct{x})$; this form is called the \emph{polar form} of $Q$. 

An $n$-ary quadratic form is said to be \emph{totally positive definite} if 
$Q(\vct{x})\succ 0$
for all nonzero vectors $\vct{x}\in\OK^n$. The form $Q$ is \emph{classical} if $2$ divides $a_{ij}$ for all $i\neq j$. It is called \emph{diagonal} in the case when $a_{ij}=0$ for all $i\neq j$; in such a case, we write $Q$ simply as $\uqf{a_1}\bot\uqf{a_2}\bot\dots\bot\uqf{a_n}$. More generally, the expression $Q_1\bot Q_2$ means  the quadratic form $Q_1(x_1,\dots, x_n)+Q_2(x_{n+1},\dots,x_{n+m})$.

We say that $Q$ is \emph{universal} if it \emph{represents} all the elements belonging to $\OKPlus$, i.e., for every $\alpha\in\OKPlus$ we can find a vector  $\vct{e}\in\OK^n$ such that $Q(\vct{e})=\alpha$.
Moreover, an $n$-ary quadratic form $Q$ is \emph{represented} by an $m$-ary quadratic form $R$ over $\OK$ if $n\leq m$ and there exist $n$-ary linear forms $\ell_i$, $i=1,\dots,m$, with coefficients in $\OK$, such that 
\[Q(x_1,\dots,x_n)=R(\ell_1(x_1,\dots,x_n), \dots, \ell_m(x_1,\dots,x_n)).\] 
(E.g., the form $Q(x_1,x_2)=2x_1^2+3x_2^2$ is represented by the form $R(y_1,y_2,y_3)=y_1^2+3y_2^2+5y_3^2$ over $\OK$ with $K=\Q(\sqrt2)$, because $Q(x_1,x_2)=R(\sqrt2x_1,x_2,0)$.) 
Note that the case $n=1$ yields the usual representation of an element of $\OK$. In general, it is easy to see that the form $R$ represents all the integers which are represented by the form $Q$ (and possibly some more). Thus, we also say \uv{$R$ is \emph{stronger} than $Q$} instead of \uv{$Q$ is represented by $R$}.

From now on, by a \emph{quadratic form}, or just simply a \emph{form}, we mean a totally positive definite classical quadratic form.

Two $n$-ary quadratic forms $Q_1$ and $Q_2$ are called \emph{equivalent}, denoted by $Q_1\cong Q_2$, if there exists an $n\times n$ matrix $M$ consisting of elements of $\OK$ and with $\det{M}\in\UK$, such that $Q_2(\vct{x})=Q_1(M\vct{x})$.  Note that equivalent quadratic forms represent the same elements; in particular, $Q_1$ is universal if and only if $Q_2$ is universal. 

Given an $n$-ary quadratic form $Q$ and a set of vectors $\vct{v}_1, \dots, \vct{v}_m\in\OK^n$, we call the matrix $\bigl(B_Q(\vct{v}_i,\vct{v}_j)\bigr)_{i,j=1}^{m}$
a \emph{Gram matrix of $\vct{v}_1, \dots, \vct{v}_m$ with respect to $Q$}. When this set of vectors coincides with the canonical basis of $K^n$, the quadratic form  can be expressed as
\[
Q(x_1, x_2, \dots,x_n)=\bigl(\begin{matrix}x_1 & x_2 & \cdots & x_n\end{matrix}\bigr)\cdot\begin{pmatrix}
a_{11} & \frac {a_{12}}2 & \cdots & \frac{a_{1n}}2\\
\frac{a_{12}}{2} & a_{22} & \cdots & \frac{a_{2n}}2\\
\vdots & \vdots & \ddots & \vdots \\
\frac{a_{1n}}2 & \frac{a_{2n}}2 & \cdots & a_{nn}
\end{pmatrix}\cdot\begin{pmatrix}x_1 \\ x_2 \\ \vdots \\ x_n\end{pmatrix};
\]
then we speak simply about the \emph{Gram matrix associated to $Q$} or \emph{the matrix of $Q$}.

A form is totally positive definite, resp.\ diagonal, if and only if its Gram matrix has the same property. To examine the total positive definiteness of $Q$, one can use Sylvester's criterion: $Q$ is totally positive definite if and only if the leading principal minors of its Gram matrix are totally positive definite. Moreover, a form is classical if and only if its Gram matrix contains only integral entries.

\subsection{Indecomposable integers} \label{subsec:indecom}
Let $\alpha$ be a totally positive integer in $K$. We say that $\alpha$ is \emph{indecomposable} in $K$ if the equation $\alpha=\beta+\gamma$ cannot be satisfied for any two elements $\beta,\gamma\in\OKPlus$. The structure of indecomposable integers of all the quadratic fields is well-known. Let $n$ be a square-free positive rational integer and set
\[
    \omega_n=\left\{
                \begin{array}{ll}
                  \sqrt{n} \qquad\text{if } n\equiv 2,3\!\!\!\pmod{4},\\
                  \frac{1+\sqrt{n}}{2}\quad\text{if } n\equiv 1\!\!\!\pmod{4}.
                \end{array}
              \right.
\] 
If we take the continued fraction $[u_0,\overline{u_1,u_2,\ldots,u_l}]$ of $-\overline{\omega_n}$ (i.e., of the conjugate element to $-\omega_n$ in the quadratic field $\Q(\sqrt{n})$),
we can consider only parts of this continued fraction and get rational numbers of the form
\[
\frac{p_i}{q_i}=[u_0,u_1,\ldots,u_i].
\]
We can also obtain these numbers from the recurrence relations
\begin{align*}
p_{i+1}=&\;p_{i-1}+u_{i+1}p_i,\\
q_{i+1}=&\;q_{i-1}+u_{i+1}q_i,
\end{align*}
where $p_{-1}=1$, $q_{-1}=0$, $p_{0}=u_0$ and $q_{0}=1$. Having this sequence of $p_i$ and $q_i$, we define elements $\alpha_i=p_i+q_i\omega_n$ called \emph{convergents} of $-\overline{\omega_n}$. However, in addition, we can consider integers of the form $\alpha_{i,k}=\alpha_i+k\alpha_{i+1}$ where $0\leq k\leq u_{i+2}$. These elements are called \emph{semiconvergents}  of $-\overline{\omega_n}$ and belong to $\OKPlus$ if and only if $i$ is odd. In such a case, the elements $\alpha_{i,k}$'s are all the indecomposable integers in $\Q(\sqrt{n})$; up to multiplication by units, there are only finitely many of them (see \cite{Pe,DS}). 

In biquadratic fields, we do not have such a characterization of indecomposable integers. 
The only to us known result is \cite[Th.~2.1]{CLSTZ} which claims that under certain conditions, the indecomposables from quadratic subfields remain indecomposable in the biquadratic field.  The meaning of $p$, $q$ and $r$ in this theorem agrees with the one given by Convention~\ref{Conv}.

\begin{theorem}\label{Theorem:Indecomposables}
Let $K=\BQ p q $ be a biquadratic number field. For $k\in\{p,q,r\}$, we set $M_k=\max\{u_i;\; i\text{ odd}\}$, where $[u_0,\overline{u_1,u_2,\dots, u_{l-1}, u_l}]$ is the continued fraction of $-\overline{\omega_k}$.

\begin{enumerate}[(a)]
\item Let $\alpha\in\Q(\sqrt{p})$ be indecomposable. 
	\begin{itemize}
		\item If $\alpha$ is a convergent of $-\overline{\omega_p}$ and $\sqrt{r}>\sqrt{p}$, then $\alpha$ is indecomposable in $K$.
        \item If $\sqrt{r}>M_p\sqrt{p}$, then $\alpha$ is indecomposable in $K$.
    \end{itemize}

\item Let $\beta\in\Q(\sqrt{q})$ be indecomposable.
	\begin{itemize}
    	\item In the cases (B1), (B2), (B3), $\beta$ is indecomposable in $K$.
        \item In the case (B4), if $\sqrt{r}>\sqrt{q}$ and $\beta$ is a convergent of $-\overline{\omega_q}$, then $\beta$ is indecomposable in~$K$. 
        \item In the case (B4), if $\sqrt{r}>M_q\sqrt{q}$, then $\beta$ is indecomposable in $K$.
    \end{itemize}
\item  Let $\gamma\in\Q(\sqrt{r})$ be indecomposable.
	\begin{itemize}
    	\item If $\gamma$ is a convergent of $-\overline{\omega_r}$ and $\sqrt{p}>\sqrt{r}$, then $\gamma$ is indecomposable in $K$.
        \item If $\sqrt{p}>M_r\sqrt{r}$, then $\gamma$ is indecomposable in $K$.
    \end{itemize}
\end{enumerate}
\end{theorem}

Nevertheless, as we shall see in Subsection~\ref{Subsec:NewIndeElms}, this theorem does not cover all the indecomposable integers originating from quadratic subfields and, moreover, we will provide an example of an element indecomposable in a quadratic subfield which decomposes in our biquadratic field $K$. Both was already foreshadowed in Theorem~\ref{Thm:Inde}.

%----------------------------------------------------------------------------------------------
\section{Idea of the proof} \label{Sec:IdeaOfProof} \label{Sec:ProofIdea}

In the main part of this article, we do not use the method of escalation as  developed in \cite{BH} (the only exception will be fields containing $2$ or $5$ in Subsections~\ref{Subsec:2}, \ref{Subsec:5} and~\ref{Subsec:Specific}). Instead of that, we follow the ideas from \cite{CKR}: Assume that $Q$ is a universal classical totally positive definite ternary quadratic form over $\OK$. Obviously $Q$ has to represent the number $1$, and hence $Q$ can be orthogonally split into $Q\cong\uqf{1}\bot Q_0$ (see Corollary~\ref{Cor:Stepeni}\eqref{Cor:Stepeni2}).
This step is one of the key ingredients for the proof, as it can actually be used repeatedly any time when a totally positive unit (not necessarily $1$) happens to be represented by the form $Q$ or any of its subforms.

Now it is easy to see that if we have \uv{enough} totally positive units which do not differ (multiplicatively) from each other by a square, then the nonuniversality of $Q$ follows: Suppose that $\ve$, $\ve'$ are two totally positive units such that $\ve,\ve',\ve\ve'\notin\UKctv$ (in particular, $\ve, \ve'\neq1$ and $\ve\neq\ve'$). The unary form $\uqf1$ clearly represents all squares, and in particular all elements of $\UKctv$, but it does not represent $\ve$. We will see in Lemma~\ref{Lemma:TrivDecomp} that a totally positive unit cannot be written as a sum of two elements of $\OKPlus$; thus, $\ve$ has to be represented by the binary form $Q_0$. Therefore, we can repeat the splitting and write $Q_0\cong\bqf{\ve}{\gamma}$ for some $\gamma\in\OKPlus$, i.e., $Q \cong \tqf{1}{\ve}{\gamma}$. After that, we apply the same arguments for $\ve'$. Note that $\uqf{\ve}$ represents only elements from $\ve\OKctv$, and so it does not represent $\ve'$; we obtain that $\ve'$ must be represented by $\uqf{\gamma}$, which easily gives $Q\cong\tqf{1}{\ve}{\ve'}$. But in such a case, $\ve\ve'$ is another nonsquare totally positive unit, and it is not represented by $Q$ for the same reasons. Therefore, under the given assumptions, no  classical totally positive definite ternary quadratic form can be universal, as we have actually shown that no such form can represent all the elements $1, \ve, \ve', \ve\ve'$ at once. 

The discussion above indicates that the complexity of the problem depends heavily on the size of the factor group $\sfrac{\UKPlus}{\UKctv}$. 
Using Dirichlet's unit theorem, one can easily see that the group $\sfrac{\UK}{\UKctv}$ has $16$ elements. Thus, as the group $\sfrac{\UK}{\UKPlus}$ is obviously nontrivial, we deduce from the equality $\abs{\sfrac{\UK}{\UKctv}}=\abs{\sfrac{\UK}{\UKPlus}}\cdot\abs{\sfrac{\UKPlus}{\UKctv}}$
that the size of the factor group $\sfrac{\UKPlus}{\UKctv}$ can be only $1, 2, 4$ or $8$.  Let us consider the following three cases:
\begin{enumerate}[(I)]
    \item $\left|\sfrac{\UKPlus}{\UKctv}\right|>2$; in this case we can find $\ve, \ve'$ as above,
    \item $\left|\sfrac{\UKPlus}{\UKctv}\right|=2$; here we can find a totally positive unit $\ve$ such that $\{1, \ve\}$ is a system of representatives of the factor group,
    \item $\left|\sfrac{\UKPlus}{\UKctv}\right|=1$; all units are squares, thus the only relevant unit in this case is $1$.
\end{enumerate}
We have solved the case (I) already, there does not exist any universal (classical, totally positive definite) ternary quadratic form. Lemma~\ref{Lemma:RootsOfUnits} hints that this is often the case.

In the case (II),  it is sufficient to consider the quadratic forms $\tqf{1}{\ve}{\gamma}$ for $\gamma\in\OKPlus$ (including the cases $\gamma=1$ and $\gamma=\ve$). The obvious question is whether such a form represents $2$. Putting aside the case $\sqrt2\in K$, we have two possibilities: either $\uqf{\ve}$ does not represent $2$, and hence $\gamma$ has to be either $1$ or $2$ (see Lemma~\ref{Lemma:TrivDecomp}), or $\uqf{\ve}$ does represent $2$, and we have to look for another number which is not yet represented. Note that $\uqf{\ve}$ represents $2$ if and only if $2\ve$ is a square in $\OK$; we will use this condition significantly in both cases. It is also worth pointing out Proposition~\ref{Prop:StrongerForms} which summarizes some arguments used repeatedly throughout the whole proof.

The most challenging part of the proof is, of course, the case (III). We can still be lucky enough to obtain a diagonalizable binary quadratic form $Q_0$, and hence we consider quadratic forms $\tqf{1}{\beta}{\gamma}$ separately; but in general, $Q_0$ has the form 
$\alpha' y^2+2\beta' yz + \gamma' z^2$ with $\beta'\neq0$.  The problems here are caused mainly by the fact that $\beta'$ (and the term $2\beta' yz$ in general) does not have to be totally positive, which prevents us from using estimations.

It should be clear now from the ideas above that in the cases with \uv{not enough totally positive units}, we need to find other elements which are not squares (because squares are  already represented by $\uqf1$), and which are difficult to express as a sum of two elements of $\OKPlus$. In other words, we need to find some nonsquare indecomposable elements; this is the main purpose of the next section. 

%----------------------------------------------------------------------------------------------
\section{Preparation for the proof} \label{Sec:BeautifulAndUgly}
In this section, we introduce several tools that we need for the proof of Theorem~\ref{Thm:Main}. As indicated at the end of the previous section, we want to know whether some given elements of $\OK$ are squares and how to express some numbers as a sum of elements belonging to $\OKPlus$. The latter also includes the problematics of indecomposable elements in $K$; we provide a proof of Theorem~\ref{Thm:Inde}. Furthermore, we take a more detailed look at representations by unary forms. Finally, for the sake of nondiagonalizable forms in the case (III), we  include some lemmas about quadratic lattices.

%----------------------------------
\subsection{Squares} \label{Subsec:Squares}

If $\alpha\in\OK$ has the property that, when considered as an element of the field $K$, it is a square, then $\alpha$ is actually a square in the ring $\OK$, i.e., $\sqrt{\alpha}\in\OK$. However obvious this may seem, it is not automatic; in fact, it is a consequence of $\OK$ being algebraically closed in $K$. In the following, we will slightly abuse the language, and speak about algebraic integers \emph{being} or \emph{becoming a square in the (biquadratic) field $K$}, by which we actually mean that they are squares in the ring $\OK$.

The following lemma generalizes \cite[Lemma 4.1]{CLSTZ}. Note that in the statement, the meaning of $n_3$ is given by Convention~\ref{Conv} as $\frac{n_1n_2}{\gcd(n_1,n_2)^2}$.

\begin{lemma} \label{Lemma:MartinSquares}
Let $F=\Q(\sqrt{n_1})$ and $K=\BQ{n_1}{n_2}$. \begin{enumerate}
    \item Suppose that $\alpha\in\OK[F]$ is not a square in the field $F$ but becomes a square in $K$. Then if $\beta=x+y\sqrt{n_1}+z\sqrt{n_2}+w\sqrt{n_3}\in\OK[K]$ is such that $\alpha=\beta^2$, it must hold that $x=y=0.$ \label{Lemma:MartinSquaresItem1}
    \item More generally, if $\beta=x+y\sqrt{n_1}+z\sqrt{n_2}+w\sqrt{n_3} \in K$ (not necessarily integral) satisfies $\beta^2\in\Q(\sqrt{n_1})$, then either $x=y=0$ or $z=w=0$.
\end{enumerate}
\end{lemma}
\begin{proof}
Clearly it suffices to prove the second statement. We can write $\beta=A+B\sqrt{n_2}$ with $A=x+y\sqrt{n_1}$ and $B=z+\frac{w}{\gcd(n_1,n_2)}\sqrt{n_1}$; in particular, $A,B\in\Q(\sqrt{n_1})$. Then it holds that $\beta^{2} = (A^2+n_2 B^2) + 2AB\sqrt{n_2}$, thus one sees that $\beta^2\in\Q(\sqrt{n_1})$ is equivalent to $2AB=0$. If $A=0$, then $x=y=0$, and if $B=0$, then $z=w=0$. This completes the proof.
\end{proof}

As a corollary of this lemma, we obtain a simple yet powerful criterion for quadratic elements which become a square in the biquadratic field.

\begin{corollary}\label{Corollary:BecomingSquare}
Let $F=\Q(\sqrt{n_1})$ and $\alpha\in\OK[F]$.  Suppose that $\alpha$ is not a square in $F$ but becomes a square in $K=\BQ{n_1}{n_2}$.
Then every odd divisor of $\gcd(n_2,n_3)$ divides $\alpha$ as well.
\end{corollary}

\begin{proof}
From Lemma~\ref{Lemma:MartinSquares}(\ref{Lemma:MartinSquaresItem1}) we know that 
\[\alpha=\biggl(\frac{c\sqrt{n_2}+d\sqrt{n_3}}{2}\biggr)^2=\frac{c^2n_2+d^2n_3}{4}+\frac{cd\gcd(n_2,n_3)}{2}\sqrt{n_1}\] 
where the form of integral bases (B1)--(B4) ensures $c,d\in\Z$, $c\equiv d\pmod2$. If $c, d$ are even, then clearly both expressions $\frac14(c^2n_2+d^2n_3)$ and $\frac12cd\gcd(n_2,n_3)\sqrt{n_1}$ are divisible by $\gcd(n_2,n_3)$; if they are odd, then still they are both divisible by all the odd divisors of $\gcd(n_2,n_3)$.
\end{proof}

%----------------------------------
\subsection{Additive decompositions of rational integers} \label{Subsec:AddDecomp}
In the study of universality of quadratic forms in biquadratic fields, we repeatedly discuss whether our form can represent some (suitably chosen) element. This element is usually either a rational integer or an indecomposable element in $K$. In the first case, we often need to know in what ways we can express our chosen number as a sum of two elements $\alpha,\beta\in\OKPlus\cup\{0\}$ -- i.e., we examine its \emph{additive decompositions}. If these elements are also rational integers, we will say that this decomposition is \emph{trivial}. We
will be particularly interested in the cases when one of $\alpha$ and $\beta$ is a square.

First of all, let us focus on the additive decompositions of numbers $2,3$ and $5$.

\begin{lemma} \label{Lemma:TrivDecomp} \label{Rem:TrivDecompConcrete}
In a biquadratic field $K$:
\begin{enumerate}
    \item All (totally positive) units are indecomposable. 
    \item The number $2$ can always decompose only trivially as $1+1$ or $0+2$.
    \item If $\sqrt5\notin K$, then $3$ can decompose only trivially as $1+2$ or $0+3$.\\
    The only nontrivial decomposition is $\frac{3+ \sqrt5}{2} + \frac{3- \sqrt5}{2}$.
    \item If $\sqrt2, \sqrt3, \sqrt5, \sqrt{13}, \sqrt{17},\sqrt{21}\notin K$, then $5$ decomposes only trivially as $2+3$, $1+4$ or $0+5$.\\
    Up to embeddings, the list of all nontrivial decompositions is as follows: 
    \begin{itemize}
     \item $(2+\sqrt2)+(3-\sqrt2)$,
     \item $(2+\sqrt3)+(3-\sqrt3)$,
     \item $\frac{3+\sqrt5}{2} + \frac{7-\sqrt5}{2}$,
     \item $\frac{5+\sqrt5}{2} + \frac{5-\sqrt5}{2}$,
     \item $\frac{5+\sqrt{13}}{2} + \frac{5-\sqrt{13}}{2}$,
     \item $\frac{5+\sqrt{17}}{2} + \frac{5-\sqrt{17}}{2}$,
     \item $\frac{5+\sqrt{21}}{2} + \frac{5-\sqrt{21}}{2}$,
     \item $\bigl(2 + \frac{\sqrt2+\sqrt6}{2}\bigr)+\bigl(3 - \frac{\sqrt2+\sqrt6}{2}\bigr)$.
 \end{itemize}
\end{enumerate}
\end{lemma}

\begin{remark*}
As we will see from the proof, parts (1) and (2) actually hold for \emph{any} totally real field $K$, not only for the biquadratic ones.
\end{remark*}

\begin{proof}
 Part (1) follows directly from the inequality \eqref{Eq:NormSum}, as the norm of any element of $\OKPlus$ is greater or equal to $1$. The same idea is used in part (2): If $2=\alpha+\beta$ for $\alpha,\beta \in \OKPlus$, then $2 = \sqrt[N]{\norm{K/\Q}{\alpha+\beta}}\geq\sqrt[N]{\norm{K/\Q}{\alpha}}+\sqrt[N]{\norm{K/\Q}{\beta}} \geq 1+1$; we see that $\alpha$ and $\beta$ must be units, and, moreover, equality was attained in \eqref{Eq:NormSum}, thus $\frac{\alpha}{\beta} \in \Q$, which shows $\alpha=\beta=1$.

Part (2) could also have been obtained as an easy corollary of part (3): From the list of all $\alpha \succcurlyeq 0$ such that $3 \succcurlyeq \alpha$, only $0, 1$ and $2$ satisfy $2 \succcurlyeq \alpha$ as well. In a similar fashion, we can obtain (3) from the full statement of (4).

For the proof of (4), suppose $5 = \alpha + \beta$ for $\alpha, \beta \in \OKPlus \cup \{0\}$. If $\alpha\in\Q$, we obtain precisely the trivial decompositions. For $\alpha\notin \Q$, we consider all possible cases:

First, $\alpha, \beta \in \Q\big(\sqrt{n})$ for some quadratic subfield of $K$. Without loss of generality, suppose $\Tr{\Q(\sqrt{n})/\Q}{\alpha}\leq\Tr{\Q(\sqrt{n})/\Q}{\beta}$, hence $\alpha = \frac{a}{2} + \frac{b}{2}\sqrt{n}$ where $0 < \frac{a}{2} \leq \frac52$; by choosing a suitable embedding we can ensure $b>0$. By Lemma~\ref{Lemma:TotPositiveInequalities} (which will be used repeatedly in the whole following proof) we have $a > b\sqrt{n}$. If $n \not\equiv 1 \pmod{4}$, then $\frac{a}{2},\frac{b}{2} \in \Z$, hence the only solution of the inequalities $0<\frac{a}{2} \leq \frac52$, $a > b\sqrt{n} > 0$ is $\frac{a}{2}=2$, $\frac{b}{2}=1$, $n\in\{2,3\}$. If $n\equiv 1 \pmod{4}$, half-integers are also allowed (provided that $a\equiv b \pmod{2}$); the same inequalities are satisfied only by $\frac{a}{2}=\frac32$, $\frac{b}{2}=\frac12$, $n=5$ and by $\frac{a}{2}=\frac52$, $\frac{b}{2}=\frac12$, $n\in\{5, 13, 17, 21\}$.

The second case is when, in the decomposition, exactly two of the three numbers $\sqrt{m}$, $\sqrt{s}$, $\sqrt{t}$ are multiplied by a nonzero coefficient -- let us denote these two square roots by $\sqrt{n_1}$, $\sqrt{n_2}$. By looking at the bases (B1)--(B4), observe that the corresponding coefficients are half-integers since if quarter-integers are involved, all three square roots have nonzero coefficients. By first choosing $\alpha$ such that $\Tr{K/\Q}{\alpha}\leq\Tr{K/\Q}{\beta}$ and then using a suitable embedding, we have $\alpha = \frac{a}{2} - \frac{b_1}{2}\sqrt{n_1}  - \frac{b_2}{2}\sqrt{n_2}$ for $a,b_1,b_2 >0$ and $\frac{a}{2} \leq \frac52$. Thus $\alpha \succcurlyeq 0$ if and only if $a > b_1\sqrt{n_1} + b_2\sqrt{n_2}$. From the form of the possible integral bases (B1)--(B4) we see that not all of $a$, $b_1$, $b_2$ can be odd at the same time. Thus either $a$ is even, $a=2a'$, and $4 \geq 2a' > b_1\sqrt{n_1}+b_2\sqrt{n_2}$, or without loss of generality $b_1$ is even and $5 \geq a > 2\sqrt{n_1}+\sqrt{n_2}$. The only square-free integers $n_1, n_2$ which may satisfy the latter inequality are $2,3$ in any order; but in that case we are in the integral basis (B1) which requires $a$ to be even; that is impossible since $4 \not> 2\sqrt2 + \sqrt3$. So we are left with the former inequality $4 \geq 2a' > b_1\sqrt{n_1}+b_2\sqrt{n_2}$; here necessarily $a'=2$, $b_1=b_2=1$ and $4>\sqrt{n_1}+\sqrt{n_2}$. This is satisfied by $\{n_1, n_2\}$ equal to $\{2,3\}$, $\{2,5\}$, $\{2,6\}$ and $\{3,5\}$; however, we have to compare these possibilities for $\alpha=2 - \frac12(\sqrt{n_1}+\sqrt{n_2})$ with the corresponding integral bases (B1)--(B4) to see whether they belong to $\OK$. In fact, only $2 - \frac12(\sqrt{2}+\sqrt{6})$ is an algebraic integer, so the only  decomposition of this type is $5 = \bigl(2 - \frac{\sqrt2+\sqrt6}{2}\bigr)+\bigl(3 + \frac{\sqrt2+\sqrt6}{2}\bigr)$ (up to embedding).

In the third and last case, all of $\sqrt{m}$, $\sqrt{s}$ and $\sqrt{t}$ are multiplied by nonzero coefficients. In this case, instead of requiring $\Tr{K/\Q}{\alpha}\leq \Tr{K/\Q}{\beta}$, we take $\alpha$ to be the summand which, in a suitable embedding, has the form $\frac{a}{4} - \frac{b}{4}\sqrt{m}-\frac{c}{4}\sqrt{s}-\frac{d}{4}\sqrt{t}$ for $b,c,d>0$ (one of the two summands has this property). Then $\alpha \succ 0$ if and only if 
\begin{equation} \label{eq:ineqq}
a > b\sqrt{m}+c\sqrt{s}+d\sqrt{t};
\end{equation}
to arrive at a contradiction, we distinguish two subcases:

For the first subcase, suppose that  quarter-integers are involved, i.e., all of $a,b,c,d$ are odd. This necessarily means that $m\equiv s \equiv t \equiv 1 \pmod{4}$, which easily implies $t\geq 65$. Applying Lemma~\ref{Lemma:TotPositiveInequalities} to $\beta = (5 - \frac{a}{4}) + \frac{b}{4}\sqrt{m}+\frac{c}{4}\sqrt{s}+\frac{d}{4}\sqrt{t}\succ 0$, we arrive at $(5-\frac{a}{4})> |\frac{d}{4}|\sqrt{t} \geq \frac{\sqrt{t}}{4} \geq \frac{\sqrt{65}}{4} > \frac{8}{4}$, thus $a \leq 11$. Plugging this into (\ref{eq:ineqq}), we see that $\alpha \succ 0$ implies $11 > \sqrt{m}+\sqrt{s}+\sqrt{t}$. However, this never holds -- if $m$ and $s$ are coprime, we get at least $\sqrt{5}+\sqrt{13}+\sqrt{65} \approx 13.904$; if they are not, the smallest value is $\sqrt{21}+\sqrt{33}+\sqrt{77}\approx 19.102$. So there is no decomposition of this type.

For the second subcase, suppose that no quarter-integers are involved, i.e., all $a,b,c,d$ are even; denote $a=2a'$. Similarly to the previous paragraph, from $\beta\succ 0$ we obtain $5-\frac{a'}{2} > \frac{\sqrt{t}}{2}$, so except for the field $\BQ{2}{3}$ this yields $5-\frac{a'}{2} > \frac{\sqrt{10}}{2}>\frac{3}{2}$, i.e.,  $\frac{a'}{2}<\frac72$. In the case of $\BQ{2}{3}$, $\frac{a'}{2}$ is necessarily in $\Z$ and $t=6$, so $\frac{a'}{2}\leq 3$. Thus $\frac{a'}{2} \leq \frac{6}{2}$ in all cases; by plugging this into (\ref{eq:ineqq}), we obtain $6 > \sqrt{m}+\sqrt{s}+\sqrt{t}$. This is satisfied only by $m=2$, $s=3$, $t=6$; however, even this case fails since $\frac12(2-\sqrt2-\sqrt3-\sqrt6) \notin \OK$ and $\frac12(2-\sqrt2-2\sqrt3-\sqrt6)$ is not totally positive. So there are no decompositions of this kind either.
\end{proof}

\begin{remark}\label{Rem:TrivDecompSquares} 
We are particularly interested in decompositions of $1$, $2$, $3$ and $5$ as $\omega^2+\beta$ with $\beta\succcurlyeq0$. By scrutinizing the decompositions in Lemma~\ref{Lemma:TrivDecomp}, we obtain -- aside of the obvious possibilities with $\omega^2\in\{0,1,4\}$, and the slightly less obvious $\omega^2\in\{2,3,5\}$ if the appropriate square root lies in $K$ -- only the following decompositions (up to embedding): 
\begin{enumerate}
    \item $3$ can be written as $\bigl(\frac{\sqrt5+1}{2}\bigr)^2+\bigl(\frac{\sqrt5-1}{2}\bigr)^2$;
    \item $5$ can be written as $\bigl(\frac{\sqrt2+\sqrt6}{2}\bigr)^2+\bigl(3-\sqrt3\bigr)$ 
    or $\bigl(\frac{\sqrt3+\sqrt{7}}{2}\bigr)^2 + \bigl(\frac{\sqrt7-\sqrt{3}}{2}\bigr)^2$. 
\end{enumerate}
\end{remark}

Although the additive decompositions of the form $\omega^2+\beta$ can be fully described also for the numbers $8$ and $10$, we exclude biquadratic fields with small values of $m$ in order to shorten both the statement and the proof.

\begin{lemma} \label{Lemma:10and8}
Let $\sqrt2, \sqrt3, \sqrt5, \sqrt6,\sqrt7 \not\in K$. Then:
 \begin{enumerate}
    \item If $8 = \omega^2 + \beta$, $\beta\succcurlyeq 0$, then  either $\omega^2 \in\{ 0,1,4,9\}$ (these possibilities always exist), or: 
    \begin{itemize}
         \item $\omega^2 = \big(\frac{1\pm \sqrt{13}}{2}\big)^2 = \frac{7\pm\sqrt{13}}{2}$ if $\sqrt{13}\in K$,
         \item $\omega^2 = \big(\frac{1\pm \sqrt{17}}{2}\big)^2 =  \frac{9\pm\sqrt{17}}{2}$ if $\sqrt{17}\in K$,
         \item $\omega^2 = \big(\frac{1\pm \sqrt{21}}{2}\big)^2 = \frac{11\pm\sqrt{21}}{2}$ if $\sqrt{21}\in K$.
     \end{itemize}
     \item If $10 = \omega^2 + \beta$, $\beta\succcurlyeq 0$, then $\omega^2$ can take the same values as listed before, with the additional possibility $\omega^2=9$, and also $\omega^2=10$ if $\sqrt{10}\in K$.
\end{enumerate}
\end{lemma}
\begin{proof}
The first part is a trivial consequence of the second one. Let us therefore consider $10 \succcurlyeq \omega^2$ for 
\[\omega= \frac{a}{4}+\frac{b}{4}\sqrt{m}+\frac{c}{4}\sqrt{s}+\frac{d}{4}\sqrt{t}\] 
with $a,b,c,d\in \Z$. By comparing traces, we get the inequality
\begin{equation} \label{eq:ineq}
    10 \geq \frac{a^2}{16} + \frac{b^2}{16}m + \frac{c^2}{16}s + \frac{d^2}{16}t.
\end{equation}
If $m>10$, this immediately implies $\frac{|b|}{4}, \frac{|c|}{4}, \frac{|d|}{4} < 1$. For $m=10$, the same argument holds, except for $\frac{b}{4}=\pm1$, which yields the decomposition $10 = (\pm\sqrt{10})^2+0$. 
Further, we consider several cases:

First, $b,c,d$ are all zeros. Then $\frac{a}{4}$ has to be an integer, and clearly the full list of its possible values is $\frac{a}{4} = 0, \pm 1, \pm 2, \pm 3$.

Second, exactly one of $b,c,d$ is nonzero. Then $\omega$ as well as $\beta$ belong to a quadratic subfield $\Q(\sqrt{n})$ of $K$ where $n\in\{m,s,t\}$. If $n\not\equiv 1 \pmod{4}$, then the coefficient in front of $\sqrt{n}$ is an integer, which cannot happen since it is smaller then $1$ (putting aside the already discussed case of $10 = (\pm\sqrt{10})^2+0$). So $n\equiv 1 \pmod{4}$, $n\geq 13$, and we have to consider only $10 \succcurlyeq \bigl(\frac{a'\pm\sqrt{n}}{2}\bigr)^2$ for $a'$ odd, without loss of generality positive. This is equivalent to $40 \geq a'^2 + 2a'\sqrt{n} + n$. By plugging in $a'=3$ and $n=13$ we find out that $a'\geq 3$ is never possible. Thus $a'=1$; the resulting inequality $39 \geq n + 2\sqrt{n}$ is satisfied for $13$, $17$ and $21$ but not for $29$, which gives us the remaining three decompositions from the statement.

The last possibility is when at least two of $b,c,d$ are nonzero; we prove that this is not possible. Suppose first that there are no quarter-integers involved, i.e., all of $a$, $b$, $c$, $d$ are even. Then if all of $b,c,d$ were nonzero, (\ref{eq:ineq}) would yield $40 \geq m+s+t$ which can easily be checked never to hold for $m\geq 10$.

Thus if no quarter-integers are involved, exactly two of $\frac{b}{4}$, $\frac{c}{4}$, $\frac{d}{4}$ are nonzero (and thus equal to $\pm \frac12$). So we have to find out when $10 \succcurlyeq \bigl(\frac{a'}{2} + \frac{\pm\sqrt{n_1}\pm\sqrt{n_2}}{2}\bigr)^2$ holds for some $n_1,n_2\in\{m,s,t\},n_1\neq n_2$. By Lemma~\ref{Lemma:TotPositiveInequalities} this gives $10-\frac{a'^2}{4}-\frac{n_1+n_2}{4} \geq \left|\pm 2\frac{\sqrt{n_1}\sqrt{n_2}}{4}\right|$, which implies $40 - n_1 - n_2 - 2\sqrt{n_1 n_2} \geq 0$. That never holds for $n_1 \neq n_2 \geq 10$, giving the desired contradiction.

The only remaining case is when quarter-integers are involved in the decomposition. This can only happen if $m\equiv s\equiv t\equiv 1 \pmod{4}$, and all $a,b,c,d$ are odd -- especially, they cannot be zero. So the basic condition given by (\ref{eq:ineq}) translates to $10 \geq \frac{1+m+s+t}{16}$, i.e., $159 \geq m+s+t$. This almost never holds: Clearly $m,s$ cannot be coprime. Thus $m$ is a composed square-free number satisfying $159 \geq m + (m+4) + (m+8)$; this gives $m=21$ and $m=33$ as the only options. Then the inequality $159 - 21 \geq s + (s+4)$ only allows the composed square-free numbers $s=33$, $s=57$ and $s=65$. After computing the corresponding $t$'s we get $\BQ{21}{33}$ as the only case where indeed $159 \geq m+s+t$.

So we only have to prove that $10 \succcurlyeq \bigl(\frac{a+ b\sqrt{21}+ c\sqrt{33}+ d\sqrt{77}}{4}\bigr)^2$ cannot happen for $a$ odd and $\frac{|b|}{4}, \frac{|c|}{4}, \frac{|d|}{4} \in \bigl\{\frac14, \frac34\bigr\}$. This requires just straightforward calculations which we omit.
\end{proof}
\begin{remark*}
As we will see in Subsection~\ref{Subsec:SqrtsWithUnits}, the previous lemma is not required when $\sqrt{21}\in K$. In particular, there is no harm in avoiding the computation for the field  $\BQ{21}{33}$.
\end{remark*}

%----------------------------------
\subsection{New indecomposable elements}\label{Subsec:NewIndeElms}
Theorem~\ref{Theorem:Indecomposables} gives us several conditions under which indecomposable integers from quadratic subfields do not decompose in the biquadratic field. The following examples show that if these conditions are not satisfied, indecomposable elements can become decomposable in the larger field.

\begin{example} 
Let us consider $K=\BQ{2}{5}$ and $\alpha=7+2\sqrt{10}$, which is indecomposable in $\Q(\sqrt{10})$. Nevertheless, $\alpha$ can be written as
\[
\alpha=\underbrace{\frac{7}{2}+\sqrt{2}+\frac{1}{2}\sqrt{5}+\sqrt{10}}_{\beta}+\underbrace{\frac{7}{2}-\sqrt{2}-\frac{1}{2}\sqrt{5}+\sqrt{10}}_{\gamma}
\]
where $\beta$ and $\gamma$ are both totally positive integers of $\BQ{2}{5}$. In a similar way we can see that $\alpha'=10+3\sqrt{10}$, indecomposable in $\Q(\sqrt{10})$, can be expressed as
\[
\alpha'=\underbrace{5+\frac{3}{2}\sqrt{2}+\sqrt{5}+\frac{3}{2}\sqrt{10}}_{\beta'}+\underbrace{5-\frac{3}{2}\sqrt{2}-\sqrt{5}+\frac{3}{2}\sqrt{10}}_{\gamma'}.
\]
for $\beta',\gamma'\in\OKPlus$. Thus $\alpha$ and $\alpha'$ decompose in this biquadratic field.  
\end{example} 

In the following example, we will show the case when a certain element, later called $\T$, decomposes in a biquadratic field. Note that this provides the proof of Theorem~\ref{Thm:Inde}\eqref{Thm:T}.

\begin{example} \label{Exm:Tdecomp}
Let $K=\BQ{2}{21}$ and consider $\alpha=7+\sqrt{42}$, which is indecomposable in $\Q(\sqrt{42})$. Note that this element is of the form $\alpha=\csqrt{42}+\sqrt{42}$, and $42\equiv 2\pmod{4}$.  
However, as before, the element $\alpha$ can be written as 
\[
\alpha=\underbrace{\frac{7}{2}+\frac{3}{2}\sqrt{2}+\frac{1}{2}\sqrt{21}+\frac{1}{2}\sqrt{42}}_{\beta}+\underbrace{\frac{7}{2}-\frac{3}{2}\sqrt{2}-\frac{1}{2}\sqrt{21}+\frac{1}{2}\sqrt{42}}_{\gamma}
\]
where $\beta,\gamma\in\OKPlus$.   
\end{example}

On the other hand, Theorem~\ref{Theorem:Indecomposables} does not cover all indecomposable integers of quadratic subfields which do not decompose in our biquadratic field. In the following part, we will show some examples of such elements, which we later use in proving nonuniversality of ternary quadratic forms. 

Let $u\in\N$ be a square-free positive integer greater than $1$. In what follows, $\csqrt{u}$ stands for the smallest element of $\N$ greater than $\sqrt{u}$. Furthermore, we will denote by $\csqrto{u}$ the smallest odd positive integer greater than $\sqrt{u}$. Denote
\[
\M_1=\cm+\sqrt{m}.
\]
Moreover, for $m\equiv 1\pmod{4}$ put
\[
\M_{1/2}=\frac{\cmo+\sqrt{m}}{2}.
\]
Note that for $m\equiv 1\pmod{4}$ we have both $\M_1$ and $\M_{1/2}$; in other cases, there is only $\M_1$.
The elements $\S_1$, $\S_{1/2}$, $\T_1$ and $\T_{1/2}$ are defined in a similar way. Most importantly, we denote  
\[
    \M=\left\{
                \begin{array}{ll}
                  \M_1=\cm+\sqrt{m} \qquad\text{if } m\equiv 2,3\!\!\!\pmod{4},\\
                  \M_{1/2}=\frac{\cmo+\sqrt{m}}{2}\quad\text{if } m\equiv 1\!\!\!\pmod{4}.
                \end{array}
              \right.
\]
Likewise, we define $\S$ and $\T$. 

Now we will show that $\M$ is indecomposable in $\Q(\sqrt{m})$. All the indecomposable integers of this subfield can be obtained from the continued fraction $[u_0,\overline{u_1,u_2,\dots, u_{l-1}, u_l}]$ of $\sqrt{m}$ or $\frac{\sqrt{m}-1}{2}$ as convergents and semiconvergents. In the case of $m\equiv 2,3\pmod{4}$, the first two convergents of $\sqrt{m}$ are $\alpha_{-1}=1$ and $\alpha_{0}=\lfloor \sqrt{m}\rfloor+\sqrt{m}$. Semiconvergents $\alpha_{-1,k}$ are consequently equal to
\[
\alpha_{-1,k}=\alpha_{-1}+k\alpha_{0}
\]
where $0\leq k\leq u_{1}$ and $1\leq u_{1}$. Note that these elements are totally positive since the index $-1$ is odd. If $k=1$, we have
\[
\alpha_{-1,1}=\alpha_{-1}+\alpha_{0}=\lfloor \sqrt{m}\rfloor+1+\sqrt{m}=\cm+\sqrt{m}=\M_1=\M,
\]
thus $\M$, as one of the totally positive semiconvergents of $\sqrt{m}$, is indecomposable in $\Q(\sqrt{m})$.

In the case when $m\equiv 1\pmod{4}$, we consider the first two convergents
\[
\alpha_{-1}=1 \text{\quad and \quad} \alpha_{0}=\left\lfloor\frac{\sqrt{m}-1}{2}\right\rfloor+\frac{1+\sqrt{m}}{2}.
\]

Let us first suppose that $\lfloor\sqrt{m}\rfloor$ is even, i.e., $\cm$ is odd. Put $\sqrt{m}=2l+\lambda$ for some $l\in\Z$ and $0<\lambda<1$. Then
\[
\left\lfloor\frac{\sqrt{m}-1}{2}\right\rfloor=\left\lfloor\frac{2l+\lambda-1}{2}\right\rfloor=\left\lfloor l+\frac{\lambda-1}{2}\right\rfloor=l-1=\frac{\lfloor\sqrt{m}\rfloor}{2}-1.
\]
Therefore,
\begin{align*}
\alpha_{-1,1}=&\;\alpha_{-1}+\alpha_{0}=1+\frac{\lfloor\sqrt{m}\rfloor}{2}-1+\frac{1+\sqrt{m}}{2}=\frac{\lfloor\sqrt{m}\rfloor+1+\sqrt{m}}{2}\\
=&\;\frac{\cmo+\sqrt{m}}{2}=\M_{1/2}=\M.
\end{align*}
If $\lfloor\sqrt{m}\rfloor$ is odd and $\cm$ is even, i.e, $\sqrt{m}=2l+1+\lambda$ for some $l\in\Z$ and $0<\lambda<1$, we can say that
\[
\left\lfloor\frac{\sqrt{m}-1}{2}\right\rfloor=\left\lfloor\frac{2l+1+\lambda-1}{2}\right\rfloor=\left\lfloor l+\frac{\lambda}{2}\right\rfloor=l=\frac{\lfloor\sqrt{m}\rfloor-1}{2}
\]
As previously, we can conclude that
\[
\alpha_{-1,1}=1+\frac{\lfloor\sqrt{m}\rfloor-1}{2}+\frac{1+\sqrt{m}}{2}=\frac{\lfloor\sqrt{m}\rfloor+2+\sqrt{m}}{2}=\frac{\cmo+\sqrt{m}}{2}=\M_{1/2}=\M.
\]
In both these cases, $\M$ is a semiconvergent of $\frac{\sqrt{m}-1}{2}$, thus it is indecomposable in $\Q(\sqrt{m})$. Moreover, the previous part does not depend on the fact whether we choose $m,s$ or $t$. Thus we also know that $\S$ and $\T$ are indecomposable in the corresponding quadratic subfields.

Note that, in the same way, we can prove the indecomposability of $2\M-1$ in $\Q(\sqrt{m})$ if this element is totally positive. It is a simple matter to see that $u_{1}\geq 2$ in this case and $2\M-1$ is the semiconvergent $\alpha_{-1,2}$. 

Our next concern is to discuss the indecomposability of $\M$ in a general biquadratic field. Having  indecomposability in $\Q(\sqrt{m})$, Theorem~\ref{Theorem:Indecomposables} gives us a partial result of this task for some special choices of $m$. First of all, let us focus on the integral basis (B1). If $m\equiv 3\pmod{4}$ and consequently $\M=\M_1$, Theorem~\ref{Theorem:Indecomposables} says that this element is indecomposable in the considered type of biquadratic fields. The same conclusion can be drawn for the integral bases (B2) and (B3) if $m\equiv 1\pmod{4}$, in which case $\M=\M_{1/2}$. 

Therefore, it is natural to ask whether this element is indecomposable in the remaining cases of integral bases and $m\;\text{mod } 4$. Since $\M$ is a convergent of $\sqrt{m}$ or $\frac{\sqrt{m}-1}{2}$ only if $u_1=1$, Theorem~\ref{Theorem:Indecomposables} resolves this problem only in some special cases. However, as we will see in Proposition~\ref{Prop:IndecM}, the element $\M$ actually cannot decompose in any biquadratic field $\Q(\sqrt{m},\sqrt{s})$. For its proof, we first need to know how close the values of $m$, $s$ and $t$ can be. While there are infinitely many fields where $s-m=1$, the other differences cannot be arbitrarily small.

\begin{lemma} \label{Lemma:mstInequalities}
The following holds in every biquadratic field (given by the three square roots $\sqrt{m}$, $\sqrt{s}$, $\sqrt{t}$):
 \begin{enumerate}
  \item $\sqrt{t}-\sqrt{s} > \frac12$; if the integral basis is of type (B4), then $\sqrt{t}-\sqrt{s}> 2$, \label{Lemma:mstInequalities-1}
  \item if $\sqrt{t}-\sqrt{s} \leq 1$, then $m$ is even and $2\sqrt{m} > \sqrt{s}$, \label{Lemma:mstInequalities-2}
  \item $\sqrt{t}-\sqrt{m} > 1$; if the integral basis is of type (B4), then $\sqrt{t}-\sqrt{m} > 4$. \label{Lemma:mstInequalities-3}
 \end{enumerate}
\end{lemma}

\begin{proof}
All the three parts use the notation $m=s_0t_0$, $s=m_0t_0$, $t=m_0s_0$; recall that $m<s<t$ is then equivalent to $m_0>s_0>t_0$, and that in basis (B4) we have $m_0 \equiv s_0 \equiv t_0 \pmod{4}$.

\smallskip

(1) We can compute
\begin{equation}\label{Eq:t-s-ineq}
\sqrt{t}-\sqrt{s} = \sqrt{m_0}\bigl(\sqrt{s_0}-\sqrt{t_0}\bigr) = \frac{\sqrt{m_0}}{\sqrt{s_0}+\sqrt{t_0}}(s_0-t_0) >  \frac12 (s_0-t_0).
\end{equation}
The proof is concluded by observing that $s_0-t_0 \geq 1$ and in the case (B4) $s_0-t_0 \geq 4$.

\smallskip

(2) The inequality \eqref{Eq:t-s-ineq} implies that if $s_0-t_0 \neq 1$, then $\sqrt{t}-\sqrt{s} > 1$, which contradicts our assumption. Therefore, $s_0 = t_0+1$, and it follows that $m = (t_0+1)t_0$ is even. To obtain the required inequality, we write
\[
1 \geq \sqrt{t}-\sqrt{s} = \frac{\sqrt{m_0}}{\sqrt{s_0}+\sqrt{t_0}} \cdot 1 > \frac{\sqrt{m_0}}{2\sqrt{s_0}};
\]
multiplying both sides by $2\sqrt{s_0t_0}$ yields the desired inequality $2\sqrt{s_0t_0} > \sqrt{m_0t_0}$.

\smallskip

(3) We start by proving $\sqrt{s_0}(\sqrt{s_0+1}-\sqrt{s_0-1}) > 1$: Since $x \mapsto \sqrt{x}$ is a strictly concave function, it follows that $\sqrt{s_0} > \frac{\sqrt{s_0+1}+\sqrt{s_0-1}}{2}$, which after multiplication by $(\sqrt{s_0+1}-\sqrt{s_0-1})$ yields
\[
\sqrt{s_0}\bigl(\sqrt{s_0+1}-\sqrt{s_0-1}\bigr) >  \frac{(s_0+1)-(s_0-1)}{2} = 1,
\]
i.e., the desired inequality. Using this, we get
\[
\sqrt{t}-\sqrt{m} = \sqrt{s_0}\bigl(\sqrt{m_0}-\sqrt{t_0}\bigr) \geq \sqrt{s_0}\bigl(\sqrt{s_0+1}-\sqrt{s_0-1}\bigr) > 1.
\]
To prove the stronger inequality for (B4), the auxiliary inequality is $\sqrt{s_0}(\sqrt{s_0+4}-\sqrt{s_0-4}) > 4$, otherwise the proof does not change.
\end{proof}

On a side note, all the proven inequalities are the strongest possible, as can be seen by considering a sequence of fields with $m_0=t_0+2$ and $s_0=t_0+1$ (since $m_0, s_0, t_0$ are square-free, and hence not divisible by $4$, $t_0 \equiv 1 \pmod{4}$ follows)
or, in the case of (B4), $m_0=t_0+8$ and $s_0=t_0+4$. 

Now we proceed to the proof of the indecomposability of $\M$ in biquadratic fields. Note that the following proposition, aside from playing a key role in the proof of Theorem~\ref{Thm:Main}, also forms a part of Theorem~\ref{Thm:Inde}.

\begin{proposition}\label{Prop:IndecM} 
The element $\M$ is indecomposable in $\Q(\sqrt{m},\sqrt{s})$.
\end{proposition}

\begin{proof}
Since $\M$ is indecomposable  in the quadratic subfield $\Q(\sqrt{m})$, we do not have to prove its indecomposability in any of the cases where it is guaranteed by Theorem~\ref{Theorem:Indecomposables}. In particular, it is indecomposable if $m \equiv 1 \pmod{4}$ and the integral basis is not (B4).

Assume first that $\M = \M_{1/2}$. If it is decomposable, then the integral basis is (B4). Suppose that $\M$ decomposes as
\[
\M_{1/2}=\frac{\cmo+\sqrt{m}}{2}=\underbrace{\frac{a}{4}+\frac{b}{4}\sqrt{m}+\frac{c}{4}\sqrt{s}+\frac{d}{4}\sqrt{t}}_{\gamma}+\underbrace{\frac{a'}{4}+\frac{b'}{4}\sqrt{m}-\frac{c}{4}\sqrt{s}-\frac{d}{4}\sqrt{t}}_{\gamma'}
\]
where $\gamma,\gamma'\in\OKPlus$ (in particular, $a,a',b,b',c,d\in\Z$). By Lemma~\ref{Lemma:TotPositiveInequalities} we have $\frac{a}{4} > \left| \frac{d}{4} \right|\sqrt{t}$ and $\frac{a'}{4} > \left| \frac{-d}{4}\right|\sqrt{t}$; putting these two inequalities together, we obtain
\[
\Bigl| \frac{d}{4} \Bigr|\sqrt{t} < \min\Bigl\{\frac{a}{4}, \frac{a'}{4}\Bigr\} \leq \frac12\frac{a+a'}{4} = \frac{\cmo}{4},
\]
yielding $|d|\sqrt{t} < \cmo < \sqrt{m} + 2$. By part \eqref{Lemma:mstInequalities-3} of Lemma~\ref{Lemma:mstInequalities}, $d=0$.

Once we know that $d=0$, we exploit the total positivity of $\gamma$ and $\gamma'$ by choosing one specific embedding to obtain the inequalities
\[
\frac{a}{4}-\frac{b}{4}\sqrt{m}-\frac{\abs{c}}{4}\sqrt{s}>0 \hspace{5mm} \text{and} \hspace{5mm} \frac{a'}{4}-\frac{b'}{4}\sqrt{m}-\frac{\abs{c}}{4}\sqrt{s}>0.
\]
 Putting them together, we get
 \[
\abs{\frac{c}{4}}\sqrt{s} < \min\Bigl\{\frac{a-b\sqrt{m}}{4}, \frac{a'-b'\sqrt{m}}{4} \Bigr\} \leq \frac12 \frac{a+a'-(b+b')\sqrt{m}}{4} = \frac12 \frac{\cmo-\sqrt{m}}{2} < \frac12,
\]
which clearly cannot hold for $c \neq 0$.  Thus, we have $c=0$ and $d=0$, and hence $\gamma,\gamma'$ are elements of $\Q(\sqrt{m}) \cap \OKPlus = \OKPlus[\Q(\sqrt{m})]$; however, we already know that $\M$ is indecomposable in $\Q(\sqrt{m})$. That proves that $\M_{1/2}$ is indecomposable in any biquadratic field.

\bigskip

Now we will handle the case when $\M=\M_1$. The proof is very similar. Start by supposing that
\[
\M_1=\cm+\sqrt{m}=\underbrace{\frac{a}{2}+\frac{b}{2}\sqrt{m}+\frac{c}{2}\sqrt{s}+\frac{d}{2}\sqrt{t}}_{\gamma}+\underbrace{\frac{a'}{2}+\frac{b'}{2}\sqrt{m}-\frac{c}{2}\sqrt{s}-\frac{d}{2}\sqrt{t}}_{\gamma'}
\]
where again $\gamma,\gamma'\in\OKPlus$. We use Lemma~\ref{Lemma:TotPositiveInequalities} in the same manner as above to obtain
\[
\Bigl|\frac{d}{2}\Bigr|\sqrt{t} < \min\Bigl\{\frac{a}{2}, \frac{a'}{2}\Bigr\} \leq \frac12\frac{a+a'}{2} = \frac{\cm}{2},
\]
and hence $\abs{d}\sqrt{t}<\sqrt{m}+1$; similarly as before, invoking part \eqref{Lemma:mstInequalities-3} of Lemma~\ref{Lemma:mstInequalities} yields $d=0$.

Again, we use suitable embeddings of $\gamma$ and $\gamma'$ and put the resulting inequalities together:
\[
\Bigl|\frac{c}{2}\Bigr|\sqrt{s} < \min\Bigl\{\frac{a-b\sqrt{m}}{2}, \frac{a'-b'\sqrt{m}}{2} \Bigr\} \leq \frac12 \frac{a+a'-(b+b')\sqrt{m}}{2} = \frac12 \bigl(\cm-\sqrt{m}\bigr) < \frac12.
\]
This clearly implies $c=0$. But then necessarily $\gamma,\gamma'\in\OKPlus[\Q(\sqrt{m})]$, which contradicts the indecomposability of $\M$ in the quadratic field $\Q(\sqrt{m})$.
\end{proof}

In the following proposition, we prove the indecomposability of $\S$; this result, together with the previous proposition, concludes part \eqref{Thm:MaS} of Theorem~\ref{Thm:Inde}. The main techniques of the proof are very similar as in the case of $\M$; the only problem is that the inequality $\sqrt{t}-\sqrt{s} > 1$ is not necessarily satisfied: Not only are there fields where $\sqrt{t}-\sqrt{s}$ is very close to $0{.}5$, but in some of them $\ct = \cs$ actually happens. An example is the field $\BQ{1806}{2814}$, where $t_0=42$, $s_0=43$ and $m_0=67$. However, in every such field, part \eqref{Lemma:mstInequalities-2} of Lemma~\ref{Lemma:mstInequalities} applies, which we will exploit significantly.

\begin{proposition} \label{Prop:IndecS}
The element $\S$ is indecomposable in $\Q(\sqrt{m},\sqrt{s})$.
\end{proposition}
\begin{proof}
As in the case with $\M$, the element $\S$ is indecomposable in $\Q(\sqrt{s})$, so by Theorem~\ref{Theorem:Indecomposables} it is indecomposable if $s \equiv 1 \pmod{4}$ (i.e., $\S=\S_{1/2}$) and the integral basis is not (B4), and also if $s \not\equiv 1 \pmod{4}$ (i.e., $\S=\S_1$) and $s=q$.

Writing 
\[
\S_{1/2}=\frac{\cso+\sqrt{s}}{2}=\underbrace{\frac{a}{4}+\frac{b}{4}\sqrt{m}+\frac{c}{4}\sqrt{s}+\frac{d}{4}\sqrt{t}}_{\gamma}+\underbrace{\frac{a'}{4}-\frac{b}{4}\sqrt{m}+\frac{c'}{4}\sqrt{s}-\frac{d}{4}\sqrt{t}}_{\gamma'},
\]
the case of $\S = \S_{1/2}$ in the integral basis (B4) can be treated in the same manner as $\M=\M_{1/2}$ in the proof of Proposition~\ref{Prop:IndecM}, with the minor difference that part \eqref{Lemma:mstInequalities-1} of Lemma~\ref{Lemma:mstInequalities} is used instead of part \eqref{Lemma:mstInequalities-3}.

\bigskip

Let us turn our attention to the cases when $\S=\S_1$. We write
\begin{equation}\label{Eq:S1decomp}
\S_1 = \cs + \sqrt{s} = \underbrace{\frac{a}{2}+\frac{b}{2}\sqrt{m}+\frac{c}{2}\sqrt{s}+\frac{d}{2}\sqrt{t}}_{\gamma}+\underbrace{\frac{a'}{2}-\frac{b}{2}\sqrt{m}+\frac{c'}{2}\sqrt{s}-\frac{d}{2}\sqrt{t}}_{\gamma'},
\end{equation}
and the same approach as before yields
\begin{equation} \label{Eq:S1-dineq}
\Bigl| \frac{d}{2} \Bigr|\sqrt{t} < \min\Bigl\{\frac{a}{2}, \frac{a'}{2}\Bigr\} \leq \frac12\frac{a+a'}{2} = \frac{\cs}{2},
\end{equation}
implying $|d|\sqrt{t} < \cs < \sqrt{s} + 1$. It is clear that $|d| \geq 2$ is impossible; however, $|d| = 1$ is not excluded by this inequality. 

Suppose $|d|=1$. If $a\neq a'$, then the inequality \eqref{Eq:S1-dineq} can be improved to 
\[
\frac12 \sqrt{t} = \Bigl| \frac{d}{2} \Bigr|\sqrt{t} < \min\Bigl\{\frac{a}{2}, \frac{a'}{2}\Bigr\} \leq \frac{\cs-1}{2},
\]
but $\sqrt{t} < \cs -1 < \sqrt{s}$ clearly cannot happen. Thus $a=a' = \frac{1}{2}\cs$. Consequently, the inequalities $\frac{1}{2}\cs > \left|\frac{c}{2}\right|\sqrt{s}$ and $\frac{1}{2}\cs > \bigl|\frac{c'}{2}\bigr|\sqrt{s}$ together with the condition $c+c'= 2$ ensure that $c=c'=1$. All in all, the assumption $d \neq 0$ leads us to a decomposition of the form
\[
\S_1 = \cs + \sqrt{s} = \underbrace{\frac{\cs}{2}+\frac{b}{2}\sqrt{m}+\frac12\sqrt{s}+\frac12\sqrt{t}}_{\gamma} + \underbrace{\frac{\cs}{2}-\frac{b}{2}\sqrt{m}+\frac12\sqrt{s}-\frac12\sqrt{t}}_{\gamma'}.
\]
On a side note, $\gamma$ and $\gamma'$ differ only by an embedding, so total positivity of one implies total positivity of the other; therefore, we shall focus only on $\gamma$. Obviously $b>0$; otherwise in one embedding all three signs would be negative, but the inequality $\frac{\cs-\sqrt{s}-\sqrt{t}}{2}>0$ clearly cannot hold. On the other hand, if we suppose $b\geq 3$, then by Lemma~\ref{Lemma:TotPositiveInequalities} we have
\[
\frac{\cs}{2} > \Bigl| \frac{b}{2}\Bigr| \sqrt{m} \geq \frac{3}{2} \sqrt{m}.
\]
Recall that \eqref{Eq:S1-dineq} implies $\sqrt{t}<\sqrt{s}+1$, hence part \eqref{Lemma:mstInequalities-2} of Lemma~\ref{Lemma:mstInequalities} guarantees that $m$ is even and $2\sqrt{m} > \sqrt{s}$. But this would yield $2\cs > 3\sqrt{s}$; however, this inequality never holds, meaning that either $b=1$ or $b=2$.

If $b=2$, then the coefficient in front of $\sqrt{m}$ is an integer, whereas in front of $\sqrt{s}$ and $\sqrt{t}$ are half-integers. Comparing this with bases (B1)--(B3), this must mean that $m=q$. Nonetheless, we also know that $m$ is even, which is a contradiction.

If $b=1$, then $\gamma=\frac12\left(\cs + \sqrt{m} + \sqrt{s} + \sqrt{t}\right)$ must be totally positive. However, that would mean $(\cs - \sqrt{s}) - (\sqrt{t}-\sqrt{m}) >0$, which is impossible since clearly $\cs - \sqrt{s}<1$, whereas $\sqrt{t}-\sqrt{m}>1$ by Lemma~\ref{Lemma:mstInequalities}\eqref{Lemma:mstInequalities-3}.

So, although it was much more work than before, we have showed that $d=0$. After that, we return to the usual strategy:
\[
\S_1 = \cs + \sqrt{s} = \underbrace{\frac{a}2+\frac{b}2\sqrt{m}+\frac{c}2\sqrt{s}}_{\gamma} + 
\underbrace{\frac{a'}2-\frac{b}2\sqrt{m}+\frac{c'}2\sqrt{s}}_{\gamma'}.
\]
Again, on each of $\gamma$ and $\gamma'$ we apply one specific embedding, and putting the obtained inequalities together, we get
\[
\Bigl|\frac{b}{2}\Bigr|\sqrt{m} < \min\Bigl\{\frac{a-c\sqrt{s}}{2}, \frac{a'-c'\sqrt{s}}{2} \Bigr\} \leq \frac12 \frac{a+a'-(c+c')\sqrt{s}}{2} = \frac12 \bigl(\cs-\sqrt{s}\bigr) < \frac12,
\]
which clearly implies $b=0$. Since $\S$ is indecomposable in $\Q(\sqrt{s})$, we have proven the indecomposability of $\S=\S_1$ in any biquadratic field.
\end{proof}

Let us foreshadow that in the proof of Theorem~\ref{Thm:Main}, we use the indecomposability of $\S$ only in a few cases where $m$ is a prime; in such a situation (and more generally if $m$ and $s$ are coprime) the proof of Proposition~\ref{Prop:IndecS} becomes easier thanks to $t = ms$.

The following statement is interesting on its own, and forms part \eqref{Thm:2M-1} of Theorem~\ref{Thm:Inde}. But for the proof of Theorem~\ref{Thm:Main} we need only the case when $m=85$ (see Subsection~\ref{Subsec:m=85}), where it can also be checked directly. 

\begin{proposition}
\label{Proposition:2M-1indecomposable}
If $2\M-1$ is totally positive, then it is indecomposable in $\Q(\sqrt{m},\sqrt{s})$.
\end{proposition}
\begin{proof}
The proof is analogous to the ones of Propositions~\ref{Prop:IndecM} and~\ref{Prop:IndecS}. The only tricky part is to show that if $\M=\M_1$, the decomposition
\[
2\M-1 = \biggl(\Bigl(\cm -\frac12\Bigr) + \sqrt{m} + \sqrt{s} + \frac{\sqrt{t}}{2}\biggr) + \biggl(\Bigl(\cm -\frac12\Bigr) + \sqrt{m} - \sqrt{s} - \frac{\sqrt{t}}{2}\biggr)
\]
is not valid since $\bigl(\cm -\frac12\bigr) + \sqrt{m} + \sqrt{s} + \frac{\sqrt{t}}{2}$ is not a totally positive integer. To prove that, the following claim is used: If $t \equiv 1 \pmod{4}$ and $\sqrt{s}-\sqrt{m}\leq \frac12$, then $\sqrt{t}>4\sqrt{m}$. This claim is proved similarly to Lemma \ref{Lemma:mstInequalities}(\ref{Lemma:mstInequalities-2}).
\end{proof}

\begin{remark} 
If we are interested in indecomposability of $\T$, the only part of Theorem~\ref{Theorem:Indecomposables} which is useful is the one claiming that in bases (B1), (B2) and (B3), all indecomposable integers from $\Q(\sqrt{q})$ remain indecomposable. Thus if $t=q$ and the integral basis is not (B4) -- these conditions can be reformulated as $m\equiv s \not\equiv 1 \pmod{4}$ -- then $\T$ is indecomposable. On the other hand, Example~\ref{Exm:Tdecomp} shows that sometimes $\T$ does decompose. To characterize the cases when that happens could be a direction of a further examination.
\end{remark}

%----------------------------------
\subsection{Some useful properties of \texorpdfstring{$\M$}{M} and \texorpdfstring{$\S$}{S}} \label{Subsec:UsefulProp}

In the study of ternary forms over $\OK$, it is also useful to know whether $n\M$ is a square. As we will see in Corollary~\ref{Cor:NonsquareM}, for $n\leq5$, this occurs only for few cases of $m$ and for some specific biquadratic fields. We will also analyze the same question for $\S$ and $2\S$. In order to use Lemma~\ref{Lemma:MartinSquares}(\ref{Lemma:MartinSquaresItem1}), we divide the question into two separate problems: First, we prove that the considered element is not a square in the appropriate quadratic subfield, and then we examine it in the biquadratic field. This approach yields the following:

\begin{lemma} \label{Lemma:NonsquareM}  
Let $n\leq5$ be a positive integer. Then
\begin{enumerate}
    \item  $n\M_1$ is a square in $\Q(\sqrt{m})$ if and only if 
    \begin{itemize}
        \item $n=2$ 
                and $m\in\{3,5\}$.
    \end{itemize}
    $n\M_{1/2}$ is a square in $\Q(\sqrt{m})$  if and only if 
    \begin{itemize}
        \item $n=1$ or $n=4$ 
        and $m=5$;
        \item $n=3$  and $m\in\{21, 33\}$,
        \item $n=5$ and $m\in\{5, 65, 85\}$.
    \end{itemize}
   
    \item Let $K=\BQ{m}{s}$. Then \\
    $n\M_1$ is a square in $K$ but not in $\Q(\sqrt{m})$ if and only if
        \begin{itemize}
            \item $n=3$, $m=5$ and $K=\BQ56$,
            \item $n=5$ and either $m=3$ and $K=\BQ{3}{10}$, or $m=21$ and $K=\BQ{21}{30}$; 
        \end{itemize}
    $n\M_{1/2}$  can only be a square in $K$ if it is a square in $\Q(\sqrt{m})$. 
\end{enumerate}
\end{lemma}

\begin{remark*}
 Recall that $\M_1$ is defined for all square-free values of $m$, but $\M_{1/2}$ is defined only for those square-free $m$ which satisfy $m\equiv1\pmod4$; thus, the above lemma does not cover $n\M_{1/2}$ for $m\not\equiv1\pmod4$ even though it might be an algebraic integer. Moreover, note that in (1), we do not use the fact that $m<s<t$. 
\end{remark*}

\begin{proof}
(1)  
First, consider the odd multiples of $\M_1$; let $a,b\in\Z$  be such that $a\equiv b \pmod2$ and
\[
\biggl(\frac{a+b\sqrt{m}}{2}\biggr)^2=(2k+1)\cm+(2k+1)\sqrt{m}
\]
for some $k\in\Z$, $k\geq0$, and square-free $m\in\Z$, $m>1$. Comparing the coefficients in front of $\sqrt{m}$ gives the equality $2ab=4(2k+1)$, which cannot be satisfied with $a\equiv b \pmod2$. Thus, $(2k+1)\M_1$ is not a square in $\Q(\sqrt{m})$ for any $k$ and any $m$.

Observe that $4\M_1$ is a square if and only if $\M_1$ is a square (more generaly, for $k,l\in\Z$, it holds that $k^2l\M$ is a square if and only if $l\M$ is a square). Thus, in the case of $n\M_1$, it only remains to deal with $n=2$. First note that no halves can appear in the hypothetical square root of $2\M_1$. Hence, let $a,b\in\Z$ be such that
\begin{equation} \label{Eq:2M1}
    \left(a+b\sqrt{m}\right)^2=2\left(\cm+\sqrt{m}\right).
\end{equation}
Comparing the coefficients in front of $\sqrt{m}$, we see that necessarily $a=b=\pm1$.  Then the rest of the equality yields  $1+m=2\cm$
, which is satisfied only for $m\leq5$. Evaluating \eqref{Eq:2M1} at $a=b=\pm1$ and $m\in\{2,3,5\}$ gives that $2\M_1$ is a square if and only if $m\in\{3,5\}$.

\bigskip
For the multiples of $\M_{1/2}$, consider $n\in\Z$, $n>0$, and let $a,b\in\Z$ be such that $a\equiv b\pmod2$  and 
\begin{equation}\label{Eq:nM12}
\left(\frac{a+b\sqrt{m}}{2}\right)^2=n\frac{\cmo+\sqrt{m}}{2}; \end{equation}
note that we can assume without loss of generality $a,b>0$. Taking the conjugates gives
\[\left(a-b\sqrt{m}\right)^2=2n\bigl(\cmo-\sqrt{m}\bigr)<4n,\]
i.e., $\left|a-b\sqrt{m}\right|<2\sqrt{n}$, which produces the upper bound
\begin{equation} \label{Eq:nM12odhad}
    m<\left(\frac{2\sqrt{n}+a}{b}\right)^2. 
\end{equation}
Comparing the coefficients in front of $\sqrt{m}$ in \eqref{Eq:nM12} yields $n=ab$, which gives us for a fixed value of $n$ only a few possibilities for $a,b$. It is just a matter of a simple computation to check which $m$'s satisfy both \eqref{Eq:nM12odhad} and \eqref{Eq:nM12}. 

\bigskip
(2) 
Let us start with a simple observation: For a given $m_0\in\N$, there are only finitely many biquadratic fields $K$ such that $\gcd(s,t)=m_0$. This becomes obvious if we write $s=m_0t_0$, $t=m_0s_0$, $m=s_0t_0$, since the inequality $m<s<t$ translates into $m_0>s_0>t_0$. Thus every field with $\gcd(s,t)=m_0$ is determined by the choice of $s_0$ and $t_0$, which are square-free numbers ($t_0=1$ is possible) smaller than $m_0$ such that $m_0$, $s_0$, $t_0$ are pairwise coprime.

Suppose now that $n\M_1 = n\cm + n\sqrt{m}$ is not a square in $\Q(\sqrt{m})$ but becomes a square in $K$. Invoking Lemma~\ref{Lemma:MartinSquares}(\ref{Lemma:MartinSquaresItem1}), we know that this means $\bigl(\frac{c\sqrt{s}+d\sqrt{t}}{2}\bigr)^2=n\M_1$ for some $c,d\in\Z$. Comparing the coefficients in front of $\sqrt{m}$ and writing $m_0$ for $\gcd(s,t)$, we obtain $\frac{2cdm_0}{4}=n$, i.e., $cdm_0=2n$. Thus $m_0$ is some divisor of $2n$, which gives us only a few concrete fields $\BQ{m}{s}$ which have to be checked. Through realizing that $c\equiv d\equiv n \pmod2$, one can reduce the number of possibilities for $c,d$, and consequently also for $m_0$, significantly. The few remaining cases can be checked directly, concluding this part of the proof.

For $n\M_{1/2}$, we proceed analogously: Again, we write $m_0$ for $\gcd(s,t)$; then the requirement $\bigl(\frac{c\sqrt{s}+d\sqrt{t}}{2}\bigr)^2=n\M_{1/2}$ leads to $cdm_0=n$, thus a simple computation shows that $n\M_{1/2}$ never becomes a square for $n=1,2,4$. Furthermore, for $n=3$ the only possible field is $\BQ{2}{3}$ with $m_0=3$; and for $n=5$, we have $m_0=5$, meaning that both $s_0$ and $t_0$ are chosen from $\{1,2,3\}$.  However, in all the resulting fields we have $m\not\equiv 1 \pmod{4}$, so there is no $\M_{1/2}$ and considering $n\M_{1/2}$ is utterly irrelevant.
\end{proof}

In most situations, we are interested whether $n\M$  (equal either to $n\M_1$ or to $n\M_{1/2}$ according to the value of $m$) is a square without distinguishing between quadratic and biquadratic fields.

\begin{corollary} \label{Cor:NonsquareM}
In a biquadratic field $K$, the following holds: 
\begin{enumerate}
    \item $\M$ and $4\M$ are squares if and only if $m=5$, \label{Cor:NonsquareM_14}
    \item $2\M$ is a square if and only if $m=3$, \label{Cor:NonsquareM_2}
    \item $3\M$ is a square if and only if $m\in\{21,33\}$, \label{Cor:NonsquareM_3}
    \item $5\M$ is a square if and only if $m\in\{5,65,85\}$ or $K=\BQ{3}{10}$. \label{Cor:NonsquareM_5}
\end{enumerate}
\end{corollary}

The following lemma shows that $\S$ is almost never a square. Note that it does not say anything about $\S_1$ if $s \equiv 1 \pmod{4}$, since then $\S=\S_{1/2}$.

\begin{lemma}
\label{Lemma:NonsquareS}
The element $\S\in\OK$ is a square only in the fields $\BQ{2}{3}$, $\BQ25$ and $\BQ35$; the element $2\S$ is a square only in $\BQ23$ and $\BQ25$. More specifically:
 \begin{enumerate}
     \item If $\S$ is nonsquare in $\Q(\sqrt{s})$, then it is nonsquare in the biquadratic field $\BQ{m}{s}$ with the exception of $\BQ{2}{3}$ where $\S=2+\sqrt{3} = \bigl(\frac{\sqrt{2}+\sqrt{6}}{2}\bigr)^2$.
     \item If $2\S$ is nonsquare in $\Q(\sqrt{s})$, then it is nonsquare in the biquadratic field $\BQ{m}{s}$ with the exception of $\BQ{2}{5}$ where $2\S=3+\sqrt{5} = \bigl(\frac{\sqrt{2}+\sqrt{10}}{2}\bigr)^2$.
 \end{enumerate}
\end{lemma}
\begin{proof}
If $\S$ (resp.\ $2\S$) is a square already in $\Q(\sqrt{s})$, then by the first part of Lemma~\ref{Lemma:NonsquareM} and the subsequent remark we get $s=5$ (resp.\ $s=3$). Thus it remains to prove the statements (1) and~(2).

First we prove that if $\S$ or $2\S$ becomes a square in $\BQ{m}{s}$, then $m=2$: Observe that neither $\S$ nor $2\S$ has odd integer divisors, so if either of them is a square in $\BQ{m}{s}$, Corollary~\ref{Corollary:BecomingSquare} gives $\gcd(m,t)=1$ or $2$. The first case is clearly impossible since then $s=mt$ contradicts $m<s<t$. If $\gcd(m,t)=2$, then the required inequality $s=\frac{mt}{4}<t$ together with $2 \mid m$ yields $m=2$.

So it remains to examine the field $\BQ{2}{s}$. We shall make a general observation: If a nonsquare $\alpha\in\OK[\Q(\sqrt{s})]$ becomes a square in $\BQ{2}{s}$, then $2\alpha$ is a square in $\Q(\sqrt{s})$. Indeed, from Lemma~\ref{Lemma:MartinSquares}(\ref{Lemma:MartinSquaresItem1}) we get $\alpha = \bigl( \frac{a\sqrt{2}+b\sqrt{2s}}{2}\bigr)^2$, which can be rewritten as $2\alpha = (a+b\sqrt{s})^2$. This proves the observation.

In $\Q(\sqrt{s})$, the number $2\S$ is a square only for $s=3$ and $4\S$ only for $s=5$ (by Lemma~\ref{Lemma:NonsquareM} and the subsequent Remark). Thus, our observation yields that $\S$ can become a square only in $\BQ{2}{3}$ and $2\S$ only in $\BQ{2}{5}$; indeed, it is the case both for $\S$ and $2\S$. 
\end{proof}

Our next concern is to study additive decompositions of $2\M$. In particular, we are interested in the question how $2\M$ can be expressed as a sum of a square of an algebraic integer and of a totally positive element. In some cases, $2\M-1$ happens to be totally positive; then it is possible for $2\M$ to be written as $(2\M-1)+1$. Therefore, in such a situation, we need to know whether $2\M-1$ is a square.

\begin{lemma} \label{Lemma:2M-1isNonsquare}
The element $2\M-1$ is a square if and only if $m=2$.
\end{lemma} 

\begin{proof}
First of all, let us note that when $\M=\M_{1/2}$, then either $2\M-1$ is not totally positive and thus not a square, or $2\M-1=\M_1$, which is not a square in any case of $m\equiv 1\pmod{4}$, see Lemma~\ref{Lemma:NonsquareM}.

Let us now turn to $\M=\M_1$. First, suppose that $2\M_1-1$ is a square in the subfield $\Q(\sqrt{m})$, i.e., $2\M_1-1=(a+b\sqrt{m})^2$ for some $a,b\in\Z$. Then we have
\[
2\M_1-1=2\cm-1+2\sqrt{m}=a^2+b^2m+2ab\sqrt{m}.
\]
It follows that $ab=1$; the resulting equality $2\cm - 1 = m+1$ holds for $m=2$, whereas for $m\geq 3$ the right-hand side is too large.

Finally we prove that if $2\M-1$ was not a square in $\Q(\sqrt{m})$, it does not become one in $K$ either: Since $2\M-1$ has no nontrivial odd divisor, using Corollary~\ref{Corollary:BecomingSquare} we see that $2\M-1$ can become a square only in a field with $\gcd(s,t)=1$ or $2$; however, there are no such fields.
\end{proof}

\begin{lemma}\label{Lemma:SquaresUnder2M}
Let $2\M=\omega^2+\gamma$ for some $\omega\in\OK$ and $\gamma\in\OKPlus$. Then either $\omega^2=0$ or $\omega^2=1$ with the exception of $m\in\{2,3,5\}$. In particular, if $m\notin\{2,3,5\}$, then the element $2\M$ cannot be written as a sum of two squares in $\OK$.
\end{lemma}

\begin{remark*}
Note that 
$2\M_{1/2}=1+\gamma$ is possible if and only if $\cm$ is even, and $2\M_1=1+\gamma$ is possible if and only if $\cm-\sqrt{m}\geq\frac12$. 
\end{remark*}

\begin{proof}
In the first part of the proof, our strategy is to study all decompositions of the form $2\M = \gamma' + \gamma$ with $\gamma',\gamma \in \OKPlus \cup \{0\}$, ignoring the condition $\gamma' = \omega^2$.

Supposing first that $\gamma', \gamma\in\Q(\sqrt{m})$, it is easy to check that the only decompositions of $2\M$ are $0+2\M$, $\M+\M$ and possibly $1+(2\M-1)$. From them, $0$ and $1$ are squares, whereas $\M$ is a square only for $m=5$, $2\M$ only for $m=3$ (both according to Lemma~\ref{Lemma:NonsquareM}) and $2\M-1$ only for $m=2$ (Lemma~\ref{Lemma:2M-1isNonsquare}). 

The next step is to show that $2\M$ does not admit any decomposition where exactly one of the coefficients in front of $\sqrt{s}$ and $\sqrt{t}$ is nonzero. That follows quite easily since it is possible to choose an embedding where all the nonzero coefficients in front of square roots have negative sign.

Only now we start exploiting the fact that $\gamma'$ is a square: It remains to prove that $\omega^2 \preccurlyeq  2\M$ is never satisfied by an $\gamma'=\omega^2$ such that all coefficients in $\omega^2$ 
are nonzero.

Let us first focus on $\M_1$. It follows quite easily that if all four coefficients in $\omega^2$ are nonzero, then at least three coefficients in $\omega = \frac12(a+b\sqrt{m}+c\sqrt{s}+d\sqrt{t})$ are nonzero. By comparing traces of $2\M$ and $\omega^2$, we get $\frac14(a^2+b^2m+c^2s+d^2t) \leq 2\cm$, which implies
\begin{equation}\label{eq:inequa}
1+m+s\leq 8\cm.
\end{equation}
From this, we see that $1 + m + (m+1) \leq 8\cm$, which holds only for $m \leq 19$; moreover, for any fixed value of $m$, \eqref{eq:inequa} yields an upper bound for $s$ (e.g., for $m=19$ we obtain $20 + s \leq 40$, which has no relevant solutions since $s=20$ is not square-free). In this way we obtain only a few fields which have to be handled by direct computations. 

If $\M=\M_{1/2}$ and the basis is not (B4), the same technique provides a stronger result: The inequality obtained by comparing traces is in this case $1+m+s \leq 4\cmo$; however, the only $m\equiv 1 \pmod{4}$ satisfying $2+2m \leq 4\cmo$ is $m=5$.

If the basis is (B4) but, nevertheless, there are no quarter-integer coefficients in $\omega$, then the inequalities from the previous paragraph still apply. If quarter-integers are involved, we obtain a weaker inequality $\frac{1}{16}(a^2+b^2m+c^2s+d^2t) \leq \cmo$; however, in that case all of $a,b,c,d$ must be odd (in particular nonzero). Thus we get $1+m+s+t \leq 16\cmo$; this yields $1 + m + (m+4) + (m+8) \leq 16\cmo$, which is actually satisfied only for $m\leq 33$; together with the condition $m \equiv 1 \pmod{4}$ we only have to consider $m = 13$, $m=17$ and $m=21$ with $\cmo=5$ and $m=29$ and $m=33$ with $\cmo=7$ (ignoring $m=5$). However, in the former three cases, the inequality $1+m+s+t \leq 16 \cdot 5$ turns out to never hold for a field where $s\equiv t \equiv 1 \pmod{4}$; in the latter two, $1+m+s+t \leq 16 \cdot 7$ has no solution where $m=29$ or $m=33$.
\end{proof}

%----------------------------------
\subsection{Unary forms} \label{Subsec:Represent}
When looking at representability of an element by a diagonal form, it is useful to have some information on the possible representations of this element by the unary subforms. In particular, if the considered element is a rational prime $p$, we need to understand expressions of the form $p=\gamma\alpha^2$ for some $\alpha\in\OK$ and $\gamma\in\OKPlus$. But the following proposition about biquadratic divisors of rational primes can be interesting on its own.

\begin{proposition} \label{Prop:pDivSquare}
Let $p$ be a rational prime number divisible in a biquadratic field $K$ by $\alpha^2$ for an element $\alpha\in\OK\setminus\UK$. Then $p=\mu\beta^2$ for some $\mu\in\UKPlus$ and  $\beta\in\OK$.
\end{proposition}

\begin{proof}
\newcommand{\q}{\mathfrak{q}}
Recall that for a biquadratic number field $K$, the extension $K/\Q$ is Galois; hence, invoking the so-called $efg$-theorem, we can write
\begin{equation} \label{Eq:PrimeRam}
    p\OK=\prod_{i=1}^g{\p_i^e}
\end{equation}
for some prime ideals $\p_1, \dots, \p_g\subseteq\OK$ with $f=\left[\sfrac{\OK}{\p_i}:\sfrac{\Z}{p\Z} \right]$ and $efg=4$. Let $(\alpha^2)=\q_1^2\cdots\q_t^2$ for some prime ideals $\q_1, \dots, \q_t \subseteq\OK$ (not necessary pairwise different). By the assumption, $(\alpha^2)\mid p\OK$; therefore, $e\geq2$. Using that $efg=4$, we have only three possibilities for the values of $e, f, g$.

First, assume that $e=g=2$ and $f=1$. Then either $(\alpha^2)=\p_1^2\p_2^2=p\OK$, or $(\alpha^2)=\p_1^2$ (without loss of generality). In the first case, $p$ and $\alpha^2$ differ by a unit $\mu\in\UK$, and since both of $p$ and $\alpha^2$ are totally positive, the unit $\mu$ has to be totally positive as well. In the latter case it follows that the prime ideal $\p_1$ is principal; then $\p_2=\sigma(\p_1)$ for some $\sigma\in\mathrm{Gal}(K/\Q)$, and thus $\p_2=(\sigma(\alpha))$. Hence $p\OK=(\alpha\sigma(\alpha))^2$, and $p=\mu\alpha^2\sigma(\alpha)^2$ for some $\mu\in\UKPlus$ follows.

Second, consider the case $e=f=2$ and $g=1$. That means that $p\OK=\p^2$ for some prime ideal $\p$; thus necessarily $(\alpha^2)=\p^2$, and it follows that $p=\mu\alpha^2$ for some $\mu\in\UKPlus$.

Finally, suppose that $e=4$ and $f=g=1$, i.e., $p\OK=\p^4$. In this case either $(\alpha^2)=\p^2$, or $(\alpha^2)=\p^4$. The ideal $p\OK$ is generated by $\alpha^4$ or by $\alpha^2$, respectively, and thus $p=\mu\alpha^4$, or $p=\mu\alpha^2$, for some $\mu\in\UKPlus$. 
\end{proof}

The previous proof also provides the information that the prime $p$ is equal either to $\mu\alpha^2$, or $\mu\alpha^4$, or $\mu\alpha^2\sigma(\alpha)^2$ (where $\mu=1$ is allowed). In particular, we can deduce the following corollary.

\begin{corollary} \label{Cor:pDivSquare}
A rational prime $p$ is divisible by a nonunit square in $K$ if and only if $p=\mu\beta^2$ for $\mu\in\UKPlus$ and $\beta\in\OK$. Moreover, this can happen only if $p$ ramifies in $\OK$.
\end{corollary}

Recall that if a form $R$ is stronger than a form $Q$ (i.e., if $Q$ is represented by $R$), then the set of elements represented by $R$ includes all the elements which are represented by $Q$; in particular, if $Q$ is universal, then $R$ must be universal, too. When dealing with universality of diagonal forms, we are particularly interested in unary forms: For example, it is obvious that the form $\uqf1$ is stronger than the form $\uqf4$ (and thus if $Q_0\bot\uqf4$ is universal for some form $Q_0$, then $Q_0\bot\uqf1$ must be universal as well). The following proposition provides some results in this direction.

\begin{proposition}\label{Prop:StrongerForms}
Let $\gamma\in\OKPlus$.
\begin{enumerate}
    \item If $\uqf{\gamma}$ represents two elements $\alpha,\beta\in\OKPlus$, then their product $\alpha\beta$ is a square in $K$. 
    \item \label{Prop:StrongerFormsUnits} If $\uqf{\gamma}$ represents a unit $\mu$, then $\gamma$ itself has to be a totally positive unit. In particular, if $\mu=1$, then $\uqf{\gamma}\cong\uqf{1}$.
    \item \label{Prop:StrongerFormsA} If $\uqf{\gamma}$ represents $\mu\alpha^2$ with $\mu\in\UKPlus$ and $\alpha\in\OK$, then  $\uqf{\mu}$ is stronger than  $\uqf{\gamma}$.
    \item \label{Prop:StrongerFormsB} If $\uqf{\gamma}$ represents $\alpha\in\OKPlus$, then  $\uqf{\gamma}$ is stronger than $\uqf{\alpha}$.
    \item \label{Prop:StrongerFormsPrime} If  $\uqf{\gamma}$ represents a rational prime number $p$, then either $\uqf{\gamma}\cong\uqf{p}$, or there exists $\mu\in\UKPlus$ (possibly $\mu=1$) such that  $\uqf{\mu}$ is stronger than  $\uqf{\gamma}$.
\end{enumerate}
\end{proposition}

\begin{proof}
The parts (1) and (4) are trivial.

(2) Assume $\gamma z^2=\mu$ for some $z\in\OK$; the conclusion follows readily by comparing the norms. As for the second part, if $\gamma z^2=1$, then clearly $\gamma\in\UKctv$, and thus $\uqf{\gamma}\cong\uqf{1}$.

(3) Suppose $\gamma z^2=\mu\alpha^2$ for some $z\in\OK$. For an arbitrary $t\in\OK$, we want to prove that $\gamma t^2$ is represented by $\uqf{\mu}$: The equality $\gamma z^2=\mu\alpha^2$ can be rewritten as $\gamma t^2=\mu\left(\frac{\alpha t}{z}\right)^2$, where $\frac{\alpha t}{z}\in\OK$, because obviously $\gamma t^2 \in\OK$ and the unit $\mu$ is invertible.

(5) Let $z\in\OK$ be such that $\gamma z^2=p$. If $z$ is a unit, then obviously $\uqf{\gamma}\cong\uqf{p}$. Suppose $z\notin\UK$; combining Corollary~\ref{Cor:pDivSquare} with part \eqref{Prop:StrongerFormsA} of this proposition, we conclude the proof.
\end{proof}

%----------------------------

\subsection{Direct decompositions of a quadratic lattice}

Until now we were mostly preparing tools to handle diagonal forms; nevertheless, not all forms are diagonalizable. To deal with this problem, we have to work with the underlying lattice and its bases. At several places in this subsection, we exploit the fact that $\OK$ is a Dedekind domain by using the structure theorem for finitely generated modules over a Dedekind domain, see e.g. \cite[Th.~3.31]{MilneANT}. For our purposes, the following simplified version suffices:

\begin{theorem} \label{thm:structure}
Let $M$ be a finitely generated torsion-free module over a Dedekind domain $R$. Then $M$ is isomorphic to a direct sum of finitely many projective modules of rank one (which can be identified with fractional ideals):
\[
M \simeq I_1 \oplus \cdots \oplus I_{\ell}.
\]
The number $\ell$ is uniquely given. Further, if we interpret the projective modules as fractional ideals, then $M \simeq J_1 \oplus \cdots \oplus J_{\ell}$ if and only if $I_1\cdots I_{\ell} J_1^{-1}\cdots J_{\ell}^{-1}$ is a principal fractional ideal.
\end{theorem}

Before applying this theorem, we will use some more elementary tools. Recall that vectors $\vct{x}_1, \ldots, \vct{x}_n \in \OK^n$ form a lattice basis of $\OK^n$  if $\OK^n = \OK\vct{x}_1 + \cdots + \OK\vct{x}_n$. Observe that in that case, the decomposition is a direct one.

If, for a given vector $\vct{e}\in\OK^n$, its coordinates generate a nontrivial ideal $I$ in $\OK$, then the determinant of any matrix with entries from $\OK$ containing $\vct{e}$ as one of the columns lies in $I$ and thus is not invertible. This shows that such a vector cannot be part of a lattice basis of $\OK^n$. The following lemma claims that the converse holds as well. Note that all statements in this subsection hold for \emph{any} totally real field $K$.

\begin{lemma}\label{Lemma:ComplOfBasis}
A vector $\vct{e}$ is contained in some lattice basis of $\OK^n$ if and only if its coordinates generate the trivial ideal $\OK$. 
\end{lemma}
\begin{proof}
We have already seen a proof of one direction above the statement of the lemma. As for the other implication, let us first prove that the submodule $\OK\vct{e}$ has a complement in $\OK^n$, i.e., there is a submodule $M$ such that $\OK^n = \OK \vct{e} \oplus M$. If coordinates of $\vct{e}$ generate $\OK$, then there is some $\vct{d}\in\OK^n$ such that $\vct{d}^{\mathrm{T}}\vct{e}=1$. Using one such $\vct{d}$, we define a map $\pi: \OK^n \to \OK\vct{e}$ by putting $\pi(\vct{x}) = (\vct{d}^{\mathrm{T}}\vct{x})\vct{e}$. It is clearly a homomorphism of modules and it acts as identity on $\OK\vct{e}$, whence its image is $\OK\vct{e}$ and $\pi^2=\pi$. Therefore, we can write $\OK^n = \im{\pi} \oplus \Ker{\pi} = \OK\vct{e} \oplus \Ker{\pi}$. It remains to show that $\Ker{\pi}\simeq\OK^{n-1}$; this follows from Theorem~\ref{thm:structure}.
\end{proof}

An invertible change of variables (i.e., a substitution) of an $n$-ary form corresponds to choosing a different basis of the underlying lattice $\OK^n$. Bearing this in mind, we can use the above stated lemma to deduce the following corollary. Especially its second part is almost folklore, but it is difficult to find a reference.

\begin{corollary} \label{Cor:Stepeni}
Let $Q$ be an $n$-ary quadratic form.
 \begin{enumerate}
    \item \label{Cor:Stepeni1} 
    If $Q$ represents $\alpha\in\OK$ in such a way that $\alpha=Q(\vct{e})$ where coordinates of $\vct{e}$ generate the trivial ideal $\OK$, then there exists an $n$-ary quadratic form $Q'$ equivalent to $Q$ such that $Q'(y_1,y_2,\ldots,y_n)=\sum_{1\leq i\leq j\leq n} b_{ij}y_iy_j$ with $b_{11}=\alpha$.
     \item \label{Cor:Stepeni2} If $Q$ represents a unit $\mu\in\UKPlus$ (for example $1$), then $Q \cong \uqf{\mu} \bot Q_0$ for some $(n-1)$-ary subform $Q_0$. 
 \end{enumerate}
\end{corollary}
\begin{proof}
To prove (1), use Lemma~\ref{Lemma:ComplOfBasis} and consider a basis $(\vct{e}, \vct{x_2}, \ldots, \vct{x_n})$ of the lattice $\OK^n$. Then $Q(y_1\vct{e}+\ldots+y_n\vct{x_n})$ has the desired form, since the coefficient in front of $y_1^2$ is $Q(\vct{e})=\alpha$.

The statement (2) can be deduced as follows: First observe that if $Q(\vct{e})=\mu$, then the coordinates of $\vct{e}$ generate the ideal $(\mu)=\OK$, so the first part of the statement can be applied. Thus we transform the form so that the first diagonal coefficient is $\mu$.

To erase the non-diagonal coefficients for the first variable, we need another basis transformation (which corresponds to one step of Gram--Schmidt orthogonalization): We replace the basis $(\vct{e},\vct{x_2},\ldots,\vct{x_n})$ by $(\vct{e},\vct{z_2},\ldots,\vct{z_n})$ where $\vct{z_i} := \vct{x_i} -\mu^{-1}\frac{b_{1i}}{2}\vct{e}$; indeed, for such a basis we compute \[
B_Q(\vct{e},\vct{z_i})= B_Q(\vct{e},\vct{x_i}) -\mu^{-1}\frac{b_{1i}}{2}B_Q(\vct{e},\vct{e}) = \frac{b_{1i}}{2} -\mu^{-1}\frac{b_{1i}}{2}\mu = 0. \qedhere
\]
\end{proof}

Let us focus on the $2$-dimensional case. For a vector $\vct{v}\in K^2$ we denote by $\indx{v}$ the \emph{index} of $\vct{v}$ (sometimes also called the \emph{coefficient}), i.e., 
\[ \indx{v}=\left\{\eta\in K \mid \eta\vct{v}\in\OKdva  \right\};\]
note that it is a fractional $\OK$-ideal, and $\indx{v}\vct{v}=K\vct{v}\cap\OKdva$. 

\begin{lemma} 
Let $\vct{v}\in K^2$, $\vct{v}=\begin{psmallmatrix}v_1\\ v_2\end{psmallmatrix}$. Then $\indx{v}=(v_1, v_2)^{-1}$.
\end{lemma}

\begin{proof}
Consider a fractional $\OK$-ideal $J$; we have the following chain of equivalences: 
\begin{multline*}
    J\supseteq (v_1,v_2)
    \Longleftrightarrow
    J\supseteq(v_1) \ \&\  J\supseteq(v_2)
    \Longleftrightarrow
    J^{-1}(v_1)\subseteq\OK \ \&\  J^{-1}(v_2)\subseteq\OK 
    \\
    \Longleftrightarrow
    J^{-1}\begin{psmallmatrix}v_1\\v_2\end{psmallmatrix}\subseteq\OKdva
    \Longleftrightarrow
    J^{-1}\subseteq\indx{v}
\end{multline*}
The choice $J=(v_1,v_2)$ obviously satisfies $J\supseteq(v_1,v_2)$; thus, $(v_1,v_2)^{-1}\subseteq \indx{v}$. On the other hand, setting $J=\indx{v}^{-1}$ and using the other direction, we obtain $(v_1,v_2)^{-1}\supseteq \indx{v}$.
\end{proof}

\begin{lemma}\label{Lemma:LatticeDecomp}
Let $\e$ be any nonzero vector from $\OKdva$. Then there exists some $\f\in K^2$ such that $\OKdva=\indx{\e}\e\oplus \indx{\f}\f$, Moreover, $\f$ can be chosen in such a way that $\indx{\f}=\indx{\e}^{-1}$. That is, if $\e=\begin{psmallmatrix}e_1\\ e_2\end{psmallmatrix}$, then $\indx{\f}=(e_1,e_2)$.
\end{lemma}

\begin{proof}
Consider the $\OK$-module $\sfrac{\OKdva}{\indx{\e}\e}$. As a factor of a finitely generated module, it is finitely generated; it is simple to check that it is torsion-free as well. Over a Dedekind domain, a module with these two properties is projective (this is a corollary of Theorem~\ref{thm:structure}). Using a well-known property of projective modules, there is a submodule $N$ of $\OKdva$ such that $\OKdva = \indx{\e}\e \oplus N$.

The next step is to observe that $N$ must be a module of the form $K\vct{g} \cap \OKdva$ for some $\vct{g}\in\OKdva$; this follows from simple linear algebraic considerations ($\vct{g}$ can be chosen as any nonzero element of $N$).

So far we know $N=K\vct{g} \cap \OKdva=\indx{\vct{g}}\vct{g}$ for some $\vct{g}$ which even belongs to $\OK^2$; however, we have not yet ensured that $\indx{\vct{g}}=\indx{\e}^{-1}$. We have $\OKdva = \indx{\e}\e \oplus \indx{\vct{g}}\vct{g}$. Since $\OKdva$ can also be written as $\OK \vct{x}_1 \oplus \OK \vct{x}_2$ where $\vct{x}_1,\vct{x}_2$ is the standard basis, we have two decompositions of a module into projective modules of rank 1; in such a situation, Theorem~\ref{thm:structure} tells us that the fractional ideals $\OK\OK$ and $\indx{\e}\indx{\vct{g}}$ are the same up to multiplication by a principal fractional ideal, say $\OK\OK = (\delta)\indx{\e}\indx{\vct{g}}$. Then we can define a new vector $\f = \frac{1}{\delta}\vct{g}$ to obtain $\indx{\e}\indx{\f}=\OK$.
\end{proof}

%----------------------------------------------------------------------------------------------
\section{\texorpdfstring{\protect\hypertarget{Sec:NonspecialCases}{Proof in the main cases}}{Proof in the main cases}} \label{Sec:NonspecialCases}

In this section, we follow the proof strategy outlined in Section~\ref{Sec:ProofIdea}. Recall that we assume $Q \cong \uqf1\bot Q_0$ to be a universal classical ternary quadratic form and that we distinguish the following three cases:
\begin{enumerate}[(I)]
    \item $\left|\sfrac{\UKPlus}{\UKctv}\right|>2$; we have solved this case immediately in Section~\ref{Sec:ProofIdea}. \label{CaseI}
    \item $\left|\sfrac{\UKPlus}{\UKctv}\right|=2$ with a system of representatives $\{1, \ve\}$, \label{CaseII}
    \item $\left|\sfrac{\UKPlus}{\UKctv}\right|=1$, where the only relevant unit is $1$. \label{CaseIII}
\end{enumerate}

Cases \ref{CaseII} and \ref{CaseIII} are further divided into subsections according to the whether $2\ve$ is a square, resp.\ according to the diagonalizability of $Q_0$; see Appendix~\ref{App} for a schematic structure of the proof. Moreover, every subsection requires its own  additional conditions on $K$; these are always listed at the very beginning of the subsection and from the proof it will be evident why they are needed.  All the omitted cases will be covered in Section~\ref{Sec:Special cases}.

\subsection{\texorpdfstring{\protect\hypertarget{Subsec:CaseI}{Case \ref{CaseI}}}{Case (I)}}\label{Subsec:CaseI}

This case has been solved in Section~\ref{Sec:ProofIdea}, but there is actually more to be said. In particular, this part covers two different types of $\sfrac{\UKPlus}{\UKctv}$, namely $\left|\sfrac{\UKPlus}{\UKctv}\right|=4$ and $\left|\sfrac{\UKPlus}{\UKctv}\right|=8$. In the former case, as we have outlined in Section~\ref{Sec:IdeaOfProof}, four units of $\sfrac{\UKPlus}{\UKctv}$ suffice to prove that every universal form over $\OK$ must have at least four variables, which implies that no ternary form is universal. In the latter, having eight such units, we can claim an even stronger statement -- in these fields, any universal form needs at least eight variables. In the following example, we present some biquadratic fields with this property.

\begin{example}
Considering possible systems of fundamental units in $K$, see \cite{MU}, a sufficient condition for $\left|\sfrac{\UKPlus}{\UKctv}\right|=8$ to happen is that the fundamental units $\varepsilon_m, \varepsilon_s$ and $\varepsilon_t$ of quadratic subfields are totally positive and nonsquare in $K$ and none of the products $\varepsilon_m\varepsilon_s$, $\varepsilon_m\varepsilon_t$, $\varepsilon_s\varepsilon_t$ and $\varepsilon_m\varepsilon_s\varepsilon_t$ is a square in $K$. For example, these conditions are satisfied in the fields $\BQ{3}{910}$, $\BQ{21}{55}$ and $\BQ{21}{110}$.
\end{example}

\subsection{\texorpdfstring{\protect\hypertarget{Subsec:CaseIInonsquare}{Case \ref{CaseII} with $2\ve$ not being a square }}{Case (II) with 2e not being a square}}\label{Subsec:CaseIInonsquare}

Throughout this subsection suppose $\sqrt2\notin K$. In the following, we shall repeatedly exploit the first part of Proposition~\ref{Prop:StrongerForms}: The product of two numbers represented by the same unary form is a square.

We assume the ternary quadratic form $Q$ to be universal, hence it must represent both $1$ and $\ve$. Using the second part of Corollary~\ref{Cor:Stepeni} twice, we end up with a diagonal form $Q\cong\tqf{1}{\ve}{ \gamma}$ (with $\gamma\in\OKPlus$ not yet specified). This form is supposed to be universal; hence, it represents $2$. Noting that $\uqf{\ve}$ represents neither $1$ nor $2$ and using Lemma~\ref{Lemma:TrivDecomp}, we see that $2$ has to be represented by $\bqf{1}{ \gamma}$. Using the same lemma again, we find out that $\uqf{\gamma}$ represents either $1$ or $2$, since $\uqf{1}$ does not represent $2$ (because \conds{$\sqrt2\notin K$}).

Similarly, we can consider the representation of $2\ve$. Observe that $2\ve$ can be decomposed only as $0+2\ve$ or $\ve+\ve$; otherwise multiplication by $\ve^{-1}\in\OKPlus$ would give a nontrivial decomposition of $2$, contradicting Lemma~\ref{Lemma:TrivDecomp}. Since, just as before, $\uqf{1}$ represents neither $\ve$ nor $2\ve$, the form $\bqf{\ve}{\gamma}$ represents $2\ve$. Further, $\uqf{\ve}$ does not represent $2\ve$ because if \conds{$2$ is not a square}, then $2\ve^2$ is not a square either. Therefore, $\uqf{\gamma}$ has to represent $\ve$ or $2\ve$.

We have deduced that $\uqf{\gamma}$ represents either $1$ or $2$ and also either $\ve$ or $2\ve$; however, this is impossible, since none of the four products $1\cdot \ve$, $1 \cdot 2\ve$, $2 \cdot \ve$ and $2 \cdot 2\ve$ is a square. Thus, $Q$ cannot be universal.

\begin{conclusion}\label{Concl:CaseIInotsquare}
Assuming that $\ve$ is a nonsquare totally positive unit and that $2\ve$ is not a square, and further requiring $\sqrt{2}\notin K$, we have seen that there does not exist any ternary quadratic form which simultaneously represents $1, \ve, 2, 2\ve$.
\end{conclusion}

Note that the proof used neither the condition $\left|\sfrac{\UKPlus}{\UKctv}\right|=2$ nor the fact that $K$ is a biquadratic field: We only needed to know that in $K$, all units are indecomposable and $2$ decomposes only trivially as $1+1$ or $0+2$ (both of this holds in any totally real field, see the unnumbered remark below Lemma~\ref{Lemma:TrivDecomp}), that $\sqrt2 \notin K$ and that $K$ contains a totally positive unit $\ve$ such that $2\ve$ is not a square. Thus, we have actually proved the more general Theorem~\ref{Thm:General}.

\subsection{\texorpdfstring{\protect\hypertarget{Subsec:CaseIIsquare}{Case \ref{CaseII} with $2\ve$ being a square}}{Case (II) with 2e being a square}}\label{Subsec:CaseIIsquare}

Observe that in this case   $m\neq2$ (i.e., $\sqrt2\notin K$): If both $2$ and $2\ve$ were squares, then so would be $\ve$. 
Furthermore, suppose $K\neq \BQ35$.

Just as in the previous subsection, the only potentially universal form is $\tqf{1}{\ve}{\gamma}$ for some $\gamma\in\OKPlus$. Nevertheless, in this case the subform $\bqf{1}{\ve}$ represents all the numbers $2$, $3$,  $4$. One natural way to go would be to consider representations of $5$ and $5\ve$. However, it turns out to be easier to use the elements $\M$ and $\ve\M$, or, in the cases with $m\in\{3, 5\}$, the elements $\S$ and $\ve\S$. 

\hypertarget{CaseIIsquare-not35}{First assume $m\neq3,5$.} Note that both  $\M$ and $\ve\M$ are indecomposable: For $\M$ this is stated in Proposition~\ref{Prop:IndecM}, and any nontrivial decomposition of $\ve\M$ would yield a decomposition of $\M$. Furthermore, if $\ve\M$ was a square, then so would $2\M$; thus, by Corollary~\ref{Cor:NonsquareM}, neither $\M$ nor $\ve\M$ is a square \conds{since $m\neq 3,5$}. This can be rephrased as \uv{$\M$ is represented by neither of the unary forms $\uqf{1}$, $\uqf{\ve}$}. It is easy to see that the same holds for $\ve\M$ as well.

Suppose now that $\M$ is represented by $\tqf{1}{\ve}{\gamma}$ for some $\gamma\in\OKPlus$. Due to its indecomposability, it has to be represented by one of the unary subforms; from the previous paragraph, we know that it must be $\uqf{\gamma}$. The same argument holds for $\ve\M$. But if the unary form $\uqf{\gamma}$ represents both $\M$ and $\ve\M$, it means that their product, $\ve\M^2$, is a square. Thus $\ve$ is a square, which is a contradiction.

\begin{remark} 
The equality $\ve=\M$ might hold for some integers $m$, in which case $\M$ would be already represented by the binary form $\bqf1\ve$. However, we do not have to handle this separately, since it would mean that $\ve\M$ is a square, which we have already discussed and rejected.
\end{remark}

\hypertarget{CaseIIsquare-35}{Now let $m\in\{3, 5\}$;} for $K \neq \BQ35$, neither $\S$ nor $2\S$ is a square by Lemma~\ref{Lemma:NonsquareS}, and indecomposability of $\S$ follows from Proposition~\ref{Prop:IndecS}. Otherwise, we can use exactly the same arguments as above, with $\M$ replaced by $\S$.

Let us summarize the above discussion.

\begin{conclusion} \label{Concl:CaseIIsquare}
Suppose $\ve$ is a nonsquare totally positive unit such that $2\ve$ is a square. Except possibly for the field $K=\BQ35$, there does not exist any universal ternary quadratic form.

In particular, if $m\neq3,5$, then no ternary quadratic form is able to simultaneously represent $1$, $\ve$, $\M$ and $\ve\M$, and if $m\in\{3, 5\}$ but $s\neq5$,  then no ternary quadratic form is able to simultaneously represent $1$, $\ve$, $\S$ and $\ve\S$.
\end{conclusion}

If we put the previous two conclusions together, we obtain the following:

\begin{conclusion}[Case (II)]\label{Concl:epsilon}
All in all: If the field $K$ contains a nonsquare totally positive unit $\ve$ and if $m\neq2$ and $K\neq\BQ{3}{5}$, there is no universal ternary form over $K$.  
\end{conclusion}

\subsection{\texorpdfstring{\protect\hypertarget{Subsec:CaseIIIdiag}{Case \ref{CaseIII} with $Q_0$ diagonalizable}}{Case (III) with Q0 diagonalizable}}\label{Subsec:CaseIIIdiag}

For this subsection, let $\sqrt2, \sqrt3, \sqrt5\notin K$.

In this case, the form $\uqf{1}$ already represents all totally positive units; nonetheless, we assume the binary quadratic subform $Q_0$ to be diagonalizable. Thus, we have a universal ternary quadratic form $Q\cong\tqf{1}{\beta}{\gamma}$ with unspecified $\beta, \gamma \in \OKPlus$. Note that if both $\beta$ and $\gamma$ are rational integers, then the quadratic form $Q$ cannot be universal by the result of Siegel \cite{Si} \conds{invoking $\sqrt5\notin K$}. (Alternatively, to present a self-contained proof, we can observe that $\uqf{n}$ cannot represent $\M$ since $\M$ is  not a square and $\frac{\M}{n}$ for $n\geq2$ is never in $\OK$.)

The form $Q$ has to represent $2$; using Lemma~\ref{Lemma:TrivDecomp}, we can see that $\bqf{\beta}{\gamma}$ has to represent either $1$ or $2$.

\hypertarget{CaseIIIdiag-1}{First, assume that $\bqf{\beta}{\gamma}$ represents $1$}; we can put $\beta=1$ by Proposition~\ref{Prop:StrongerForms}\eqref{Prop:StrongerFormsUnits}. Since the form $\tqf{1}{1}{\gamma}$ has to represent $3$, by invoking Lemma~\ref{Lemma:TrivDecomp} again we get that $\uqf{\gamma}$ represents $1, 2$ or $3$. Using parts \eqref{Prop:StrongerFormsUnits} and \eqref{Prop:StrongerFormsPrime} of Proposition~\ref{Prop:StrongerForms}, we get that $\gamma\in\{1,2,3\}$, as $1$ is the only totally positive unit up to multiplication by a square; but a quadratic form with rational integer coefficients cannot be universal by the above mentioned argument.

\hypertarget{CaseIIIdiag-2}{Second, let $\bqf{\beta}{\gamma}$ represent $2$}; by Lemma~\ref{Lemma:TrivDecomp}, we have two possibilities for the additive decomposition of $2$: either $2=1+1$ or $2=2+0$. The first leads to the nonuniversal quadratic form $\tqf111$; the latter together with part \eqref{Prop:StrongerFormsPrime} of Proposition~\ref{Prop:StrongerForms} yields the quadratic form $\tqf12\gamma$. This form has to represent $5$; note that all the possible summands can be obtained by considering the inequality $5\succcurlyeq\omega^2$ for $\omega\in\OK$ (because if $5\succcurlyeq2\omega^2$, then $5\succcurlyeq\omega^2$) -- thus, using Remark~\ref{Rem:TrivDecompSquares}  \conds{together with the assumption $\sqrt2, \sqrt3, \sqrt5\notin K$},  we get that $\uqf{\gamma}$ has to represent at least one of the numbers $1, 2, 3, 4, 5$. Invoking Proposition~\ref{Prop:StrongerForms}, each of these possibilities leads to a quadratic form with integral coefficients: More specifically, we use part \eqref{Prop:StrongerFormsUnits} of this proposition to handle the case when $\uqf\gamma$ represents $1$, part \eqref{Prop:StrongerFormsPrime} for the primes $2,3,5$, and finally part \eqref{Prop:StrongerFormsA} for $4$. Thus, none of the possibilities for $\uqf\gamma$ leads to a universal form.

\begin{conclusion}\label{Concl:CaseIIIdiag}
If $\UKPlus=\UKctv$ holds and $\sqrt2, \sqrt3, \sqrt5\notin K$, then there is no diagonalizable universal ternary form over the biquadratic field $K$. 

In particular, we use the representations of $1,2,3$ and $5$ to transform each of the potential diagonalizable universal quadratic forms to a form with integral coefficients. Such a form cannot be universal by the result of Siegel \cite{Si}, and also because it does not represent $\M$.
\end{conclusion}

\subsection{Case \ref{CaseIII} with \texorpdfstring{$Q_0$}{Q0} nondiagonalizable} \label{Subsec:CaseIIInondiag}

This subsection requires the conditions $\splitatcommas{\sqrt2,\sqrt3,\sqrt5,\sqrt6,\sqrt7,\sqrt{10},\sqrt{13},\sqrt{17},\sqrt{21}\notin K}$ and $m\neq33, 65, 85$.

In this case we assume that the binary subform $Q_0$ is not diagonalizable. In particular, it does not represent $1$ (because if it did, we could use part \eqref{Cor:Stepeni2} of Corollary~\ref{Cor:Stepeni} to obtain a diagonal form). Since $Q\cong\uqf{1}\bot Q_0$ represents $2$ and there are only trivial additive decompositions of $2$ by Lemma~\ref{Lemma:TrivDecomp}, \conds{$\sqrt2 \notin K$} implies that $2$ has to be represented by $Q_0$. Denote $\vct{e}=\begin{psmallmatrix}e_1\\ e_2\end{psmallmatrix}$ a vector from $\OKdva$ such that $Q_0(\vct{e})=2$. To put $2$ as one of the diagonal coefficients of the form $Q_0$, we would need to extend the vector $\vct{e}$ to a basis of $\OKdva$. Referring to Lemma~\ref{Lemma:ComplOfBasis}, we can see that this is possible only if the ideal $(e_1, e_2)$ is equal to $\OK$, but that may not always be the case: In general,  $(e_1,e_2)$ is an ideal the square of which divides $(2)$, because $2=Q_0(\vct{e})$ means that $2$ is a combination of $e_1^2$, $e_1e_2$ and $e_2^2$.

Thanks to Lemma~\ref{Lemma:LatticeDecomp}, we can at least find a vector $\vct{f}=\bigl(\begin{smallmatrix}f_1\\ f_2\end{smallmatrix}\bigr)$ from $K^2$ which forms together with $\vct{e}$ a \uv{generalized basis} of $\OKdva$ in the sense that $\OKdva=\indx{e}\e \oplus \indx{f}\f$; recall that $\vct{f}$ is usually not an element of $\OKdva$ and it is chosen so that $\indx{e}\indx{f}=\OK$. We have $\indx{f}=(e_1,e_2)$ and $\indx{e}=(f_1,f_2)$; note that it follows from the discussion above that $\indx{f}^2$ divides $(2)$.

Denote $Q_0(\vct{f})=\gamma$ and $B_{Q_0}(\vct{e},\vct{f})=\beta$, 
where $B_{Q_0}$ is the polar form of the quadratic form $Q_0$.  Thus, taking a general vector $y\e+z\f\in\OKdva$ with $y\in\indx{e}$ and $z\in\indx{f}$, we have 
\begin{equation}\label{Eq:Q0}
Q_0(y\vct{e}+z\vct{f})=2y^2+2\beta yz+\gamma z^2.
\end{equation}
 Multiplying by $2$ and completing the square, we get
\begin{equation}\label{Eq:2Q0}
2Q_0(y\vct{e}+z\vct{f})=(2y+\beta z)^2 +(2\gamma-\beta^2)z^2.
\end{equation}
The total positive definiteness of $Q_0$ implies $2\gamma-\beta^2\succ0$. Note that it is not clear  whether $\beta$ and $\gamma$ belong to $\OK$. As it turns out, for $\gamma$ it is not always the case; see Appendix~\ref{App:sqrt10} for an example. This was overlooked in \cite{CKR}; however, for both their and our proof, the following lemma is sufficient.

\begin{lemma}\label{Lemma:Integrity} 
For any $y\in\indx{e}$ and $z\in\indx{f}$, we have:
\begin{enumerate}
    \item $2\gamma\in\OK$, \label{Lemma:Int1}
    \item $2\gamma-\beta^2\in\OKPlus$, \label{Lemma:Int2}
    \item $\beta\in\OK$, \label{Lemma:Int3}
    \item $2y\in\indx{f}\subseteq\OK$,\label{Lemma:Int4}
    \item $2y+\beta z\in\indx{f}\subseteq\OK$.\label{Lemma:Int5}
\end{enumerate}
\end{lemma}

\begin{proof}
Obviously, \eqref{Lemma:Int3} follows readily from \eqref{Lemma:Int1} and \eqref{Lemma:Int2} using $\beta\in K$; \eqref{Lemma:Int5} is an immediate consequence of \eqref{Lemma:Int3}, \eqref{Lemma:Int4}, and the assumption on $z$. Thus we only need to prove \eqref{Lemma:Int1}, \eqref{Lemma:Int2} and \eqref{Lemma:Int4}.

Since the form $Q_0$ represents only elements of $\OK$ (when plugging in vectors from $\OKdva$, which the vector $e_i\vct{f}$ certainly fulfills), we have that $Q_0(e_i\vct{f})=\gamma e_i^2\in\OK$ for both $i=1,2$. Then the product $\gamma^2e_1^2e_2^2\in\OK$, and thus $\gamma e_1e_2\in\OK$ as well. Therefore, $(\gamma)(e_1,e_2)^2\subseteq\OK$. As $(2)\subseteq(e_1,e_2)^2$, we obtain $(2\gamma)\subseteq(\gamma)(e_1,e_2)^2\subseteq\OK$, thus concluding \eqref{Lemma:Int1}.

As for \eqref{Lemma:Int2}, note that $2\gamma - \beta^2$ is the determinant of the matrix $\bigl(\begin{smallmatrix}2&\beta\\ \beta&\gamma\end{smallmatrix}\bigr)$. This is the Gram matrix of vectors $\vct{e}$, $\vct{f}$, which can be obtained from the original matrix of our form, 
\[
M = \Biggl(\begin{matrix}Q_0(\vct{x_1})&B_{Q_0}(\vct{x_1},\vct{x_2})\\ B_{Q_0}(\vct{x_1},\vct{x_2})& Q_0(\vct{x_2})\end{matrix}\Biggr)
\]
(where $\vct{x_1}, \vct{x_2}$ are the vectors of the standard basis), as $\bigl(\begin{smallmatrix}2&\beta\\ \beta&\gamma\end{smallmatrix}\bigr)=E^\mathrm{T}ME$ for $E = \bigl(\begin{smallmatrix}e_1& f_1\\ e_2 & f_2\end{smallmatrix}\bigr)$. Taking determinants yields $2\gamma - \beta^2 = (e_1f_2-e_2f_1)^2\cdot\det M$. Since $Q_0$ is an integral form, clearly $\det M \in \OKPlus$. Thanks to
$e_1f_2-e_2f_1 \in (e_1,e_2) (f_1,f_2)$ and $(e_1,e_2) (f_1,f_2) = \indx{f} (f_1,f_2)  = \OK$, we have $(e_1f_2-e_2f_1)^2 \in \OKctv$ and \eqref{Lemma:Int2} is proven.

To prove \eqref{Lemma:Int4}, combine $(2y)\subseteq(2)\indx{f}^{-1}$ and $(2)\subseteq\indx{f}^2$ to obtain $(2y)\subseteq\indx{f}$.
\end{proof}

Let us now take a look at how $5$ can be  represented by $Q$.  Taking into account Remark~\ref{Rem:TrivDecompSquares} along with \conds{the assumptions $m\neq2,3,5$}, it turns out that $Q_0$ has to represent  $4$ or $5$; thus, $2Q_0$ represents $8$ or $10$. Invoking Lemma~\ref{Lemma:Integrity}, both $2y+\beta z$ and $2\gamma-\beta^2$ are elements of $\OK$; hence, we are allowed to use Lemma~\ref{Lemma:10and8}.  \conds{Since we suppose $\splitatcommas{\sqrt2,\sqrt3,\sqrt5,\sqrt6,\sqrt7,\sqrt{10},\sqrt{13},\sqrt{17},\sqrt{21}\notin K}$}, 
we obtain the following possible representations:

\begin{equation}\label{Tab:8and10}
\begin{array}{c|cc|c}
2Q_0(y\vct{e}+z\vct{f}) & (2y+\beta z)^2 &  (2\gamma-\beta^2)z^2 & \text{Case}\\ \hline
10 & 0 & 10 & \text{(i)} \\
   & 1 &  9 & \text{(ii)} \\
   & 4 &  6 & \text{(iii)} \\
   & 9 &  1 & \text{(iv)} \\ \hline
 8 & 0 &  8 & \text{(v)} \\
   & 1 &  7 & \text{(vi)} \\
   & 4 &  4 & \text{(vii)}
\end{array}
\end{equation}

Furthermore, the form $Q$ has to represent $\M$. Since $\M$ is nonsquare by Corollary~\ref{Cor:NonsquareM} \conds{whenever $m\neq5$}, and indecomposable by Proposition~\ref{Prop:IndecM}, we conclude that $\M$ has to be represented by $Q_0$, and thus the form $2Q_0$ represents $2\M$.

Recall that the $\OK$-ideal $\indx{f}^2$ must divide $(2)$. Now we distinguish between two cases: Either $\indx{f}=\OK$, or $\indx{f}$ is a proper ideal such that $\indx{f}^2$ divides $(2)$. 

\bigskip
\noindent \hypertarget{CaseIIInondiag-a}{(a)} 
First, assume that $\indx{f}=\OK$; then necessarily  $\indx{e}=\OK$ as well.

\hypertarget{CaseIIInondiag-a-1}{In cases (ii), (iv) and (vii)}, we have that $(2\gamma-\beta^2)z^2$ is a square of a rational integer, and thus $2\gamma-\beta^2$ is a square in $K$. Looking at $\eqref{Eq:2Q0}$, we see that $2Q_0$ can represent only those elements of $2\OK$ which can be written as a sum of two squares. This yields a contradiction because $2Q_0$ represents $2\M$, which cannot be written as a sum of two squares by Lemma~\ref{Lemma:SquaresUnder2M} \conds{as we have $m\neq2,3,5$}.

\hypertarget{CaseIIInondiag-a-2}{The case (vi)} can be rephrased as \uv{the form $\uqf{2\gamma-\beta^2}$ represents $7$}; part \eqref{Prop:StrongerFormsPrime} of Proposition~\ref{Prop:StrongerForms} yields then that either $\uqf{2\gamma-\beta^2}\cong\uqf{7}$ or the unary form  $\uqf{2\gamma-\beta^2}$ is represented by $\uqf{1}$ (as there is no nonsquare totally positive unit). Hence, one of the binary forms $\bqf17$ or $\bqf11$ is stronger than $2Q_0$. But neither of these forms represents $2\M$: The latter is excluded in Lemma~\ref{Lemma:SquaresUnder2M} explicitly, the former is ruled out by using the same lemma and noting that neither $\frac{2\M}{7}$ nor $\frac{2\M-1}{7}$ is an element of $\OK$. Therefore, $2Q_0$ does not represent $2\M$ either. 

\hypertarget{CaseIIInondiag-a-3}{To deal with the cases (i), (iii) and (v)}, let us look at the representation of $2\M$ by the form $2Q_0$; suppose
\[2\M=\left(2\Tilde{y}+\beta\Tilde{z}\right)^2+\left(2\gamma-\beta^2\right)\Tilde{z}^2 \]
for some $\Tilde{y}, \Tilde{z}\in\OK$. Invoking Lemma~\ref{Lemma:SquaresUnder2M} \conds{along with the assumptions $m\neq2,3,5$}, we have two possibilities for the value of $2\Tilde{y}+\beta\Tilde{z}$, namely $0$ and $\pm1$. First, suppose that $2\Tilde{y}+\beta\Tilde{z}=0$; then necessarily $\left(2\gamma-\beta^2\right)\Tilde{z}^2=2\M$. Moreover, we have assumed that $(2\gamma-\beta^2)z^2=10, 6, 8$ (in cases (i), (iii), (v), respectively); therefore, we obtain that $20\M$, $12\M$, $16\M$, respectively, is a square in $K$. After reducing the integral squares, $5\M$, $3\M$, $\M$, respectively, is a square in $K$. In all cases, it is a contradiction to Corollary~\ref{Cor:NonsquareM} \conds{since the conditions for this subsection imply that $m\notin\{5,21,33,65,85\}$ and $K\neq\BQ{3}{10}$}. Thus, we must have $2\Tilde{y}+\beta\Tilde{z}=\pm1$; in such a case, the ideal $(2, \beta)$ is equal to the whole ring $\OK$. Simultaneously, we have $2y+\beta z=0$ or $\pm2$, which can be rewritten as $\beta z=2(k-y)$ for $k=0$ or $\pm1$. If $\mathfrak{q}$ is a prime ideal dividing the principal ideal $(2)$, then $\mathfrak{q}\nmid(\beta)$, because otherwise $\mathfrak{q}$ would divide the ideal $(2, \beta)=\OK$, which is impossible. It follows that $(2)\mid(z)$; in other words, $z\in(2)$, and there exists $z'\in\OK$ such that $z=2z'$. Applying this to the equality $(2\gamma-\beta^2)z^2=10, 6, 8$ implies $(2\gamma-\beta^2)z'^2=\frac{10}{4}, \frac{6}{4}, \frac{8}{4}$ in cases (i), (iii), (v), respectively. Obviously, this is absurd in cases (i) and (iii), as $(2\gamma-\beta^2)z'^2 \in \OK$ but $\frac{10}{4}, \frac{6}{4}\notin\OK$. To deal with the case (v), we rewrite the equality $(2\gamma-\beta^2)z'^2=2$ as $2(\gamma z'^2-1)=\beta^2z'^2$ and use the same argument again to show $(2) \mid (z'^2)$, i.e.,  $\frac{z'^2}{2} \in \OK$. Then the equality $(2\gamma-\beta^2)\frac{z'^2}{2}=1$ means that $\frac{z'^2}{2}$ is a unit and thus a square, implying that $2$ is a square as well, which is absurd.

\vspace{5mm}
\noindent \hypertarget{CaseIIInondiag-b}{(b)}
Now assume that $\indx{f}\neq \OK$. Recall that  $\indx{e}=\indx{f}^{-1}$ and $\indx{f}^2=(e_1,e_2)^2\mid(2)$.
Note that in this case, $2$ has to be ramified in $K$, and thus this cannot happen for fields with integral basis (B4), i.e., those where $p,q,r\equiv1\pmod4$: Indeed, the discriminant of such fields, which equals to $pqr$, is not divisible by $2$, and hence $2$ is not ramified in $\OK$.

Recall that we have
\begin{equation*}
2Q_0(y\vct{e}+z\vct{f})=(2y+\beta z)^2 +(2\gamma-\beta^2)z^2
\end{equation*}
with $y\in\indx{e}$ and $z\in\indx{f}=(e_1,e_2)\subsetneq\OK$; in particular, $z\notin\UK$.

\hypertarget{CaseIIInondiag-b-1}{In the cases (ii), (iv) and (vii)}, we have again that $2\gamma-\beta^2$ is a square in $K$, and thus this case can be solved by using the same argument as in part (a), i.e., that $2\M$ cannot be written as a sum of two squares \conds{for $m\neq2, 3, 5$}. Alternatively, the impossibility of the case (iv) can also be seen immediately, as $(2\gamma-\beta^2)z^2=1$ obviously implies $z\in\UK$.

\hypertarget{CaseIIInondiag-b-2}{In the case (vi)}, we have an odd rational integer in the right column, i.e.,  $(2\gamma-\beta^2)z^2=7$. Since $z\in\indx{f}$, we obtain that the ideal $\bigl((2\gamma-\beta^2)z^2\bigr)=(7)$ is contained in the ideal $\indx{f}^2$, and hence $(2, 7)\subseteq\indx{f}^2$ as well. But that is absurd, as obviously $(2, 7)=\OK$. (Note that the same method could have been used to deal with the case (ii) as well.)

\hypertarget{CaseIIInondiag-b-3}{In the remaining cases (i), (iii) and (v)}, we consider the representation of $2\M$ by $2Q_0$, i.e., we find $\Tilde{y}\in\indx{e}$ and $\Tilde{z}\in\indx{f}$ such that 
\[2\M=(2\Tilde{y}+\beta\Tilde{z})^2+(2\gamma-\beta^2)\Tilde{z}^2;\] 
\conds{as $m\neq2,3,5$}, Lemma~\ref{Lemma:SquaresUnder2M}  yields that $2\Tilde{y}+\beta\Tilde{z}$ is equal either to $0$ or $\pm1$. The latter is impossible, since $2\Tilde{y}+\beta\Tilde{z}\in\indx{f}$, and $\indx{f}$ does not contain any unit. The former possibility implies $(2\gamma-\beta^2)\Tilde{z}^2=2\M$; recall that we have $(2\gamma-\beta^2)z_n^2=n$ for some $z_n\in\indx{f}$ and $n=10, 6, 8$ in cases (i), (iii), (v), respectively. Multiplying  together the equalities $(2\gamma-\beta^2)\Tilde{z}^2=2\M$ and $(2\gamma-\beta^2)z_n^2=n$, we get that $2n\M$ is a square in $K$ for $n=10, 6, 8$, respectively, i.e., that $5\M$, $3\M$, $\M$, respectively, is a square in $K$. But that is a contradiction to Corollary~\ref{Cor:NonsquareM},  \conds{since the assumptions for this subsection imply $\splitatcommas{m\neq5,21, 33, 65,85}$ and $K\neq\BQ{3}{10}$}.

\begin{conclusion}\label{Concl:CaseIIInondiag}
Let $K$ be a biquadratic field such that $\UKPlus=\UKctv$, and assume $\splitatcommas{\sqrt2,\sqrt3,\sqrt5,\sqrt6,\sqrt7,\sqrt{10},\sqrt{13},\sqrt{17},\sqrt{21}\notin K}$ and $\splitatcommas{m\neq33, 65, 85}$. Then no nondiagonalizable ternary form over $\OK$ is universal. In particular, no such form can represent all of $1$, $2$, $5$ and $\M$. 
\end{conclusion}

Note that at this point, we have proven Theorem~\ref{Thm:Main} for all biquadratic fields $K$ such that $\splitatcommas{\sqrt2,\sqrt3,\sqrt5,\sqrt6,\sqrt7,\sqrt{10},\sqrt{13},\sqrt{17},\sqrt{21}\notin K}$ and $\splitatcommas{m\neq33, 65, 85}$.

%----------------------------------------------------------------------------------------------
\section{Proof in the special cases} \label{Sec:Special cases}
In this section, we focus on the \uv{special cases} of Theorem~\ref{Thm:Main} which are not included in the main proof in Section~\ref{Sec:NonspecialCases}. In particular, we are now interested in the fields which contain $\sqrt2$, $\sqrt3$, $\sqrt5$, $\sqrt6$, $\sqrt7$, $\sqrt{10}$, $\sqrt{13}$, $\sqrt{17}$, $\sqrt{21}$ or where $m\in\{33, 65, 85\}$. Note that in many of these cases, a problem arises only in one branch of the tree structure of the proof (for a basic overview, see Conclusions~\ref{Concl:epsilon}, \ref{Concl:CaseIIIdiag} and \ref{Concl:CaseIIInondiag}).  Hence, our approach is to detect the problems and solve them \uv{locally}. The exception are the cases $m=2$ and $m=5$; they need to be solved by a different method and we postpone them to Subsections~\ref{Subsec:2} and~\ref{Subsec:5}. Moreover, seven most problematic fields  will be treated in Subsection~\ref{Subsec:Specific}.

\subsection{\texorpdfstring{$\sqrt3\in K$}{sqrt3 in K}} \label{Subsec:m=3}
If the field contains $\sqrt3$, then the main proof fails at several places; however, we can use the fact that the fundamental unit $\ve_3=2+\sqrt{3}$ is totally positive and it is a square only in the field $\BQ23$ (which has to be handled separately). Moreover, $2\ve_3=(1+\sqrt3)^2$ is a square, thus we can use Conclusion~\ref{Concl:CaseIIsquare}, provided $\sqrt5 \notin K$. All in all, only two cases remain to be treated separately: $\BQ23$ and $\BQ35$; this will be done in Subsection~\ref{Subsec:Specific}.

\subsection{\texorpdfstring{$\sqrt6\in K$}{sqrt6 in K}, \texorpdfstring{$\sqrt7\in K$}{sqrt7 in K}, \texorpdfstring{$\sqrt{21}\in K$}{sqrt21 in K} or \texorpdfstring{$\sqrt{33}\in K$}{sqrt33 in K}} \label{Subsec:SqrtsWithUnits}
The cases where $m=2$ are handled in Subsection~\ref{Subsec:2}. Otherwise, by Conclusion~\ref{Concl:epsilon} it suffices to show that if $K$ contains any of the aforementioned numbers $\sqrt6$, $\sqrt7$, $\sqrt{21}$ and $\sqrt{33}$, it contains a nonsquare totally positive unit.

Indeed, $\ve_6=5+2\sqrt{6}$,  $\ve_7=8+3\sqrt{7}$, $\ve_{21}=\frac12 (5+\sqrt{21})$  
and $\ve_{33}=23+4\sqrt{33}$ are all totally positive, and as they are equal to $(\sqrt2+\sqrt3)^2$,  $\bigl(\frac{3\sqrt{2}+\sqrt{14}}{2}\bigr)^2$, $\bigl(\frac{\sqrt{3}+\sqrt{7}}{2}\bigr)^2$ and $(2\sqrt3 + \sqrt{11})^2$ respectively, they are not squares with the exception of the fields $\BQ{2}{6}$, $\BQ{2}{7}$, $\BQ37$ and $\BQ{3}{11}$, respectively. But the first two of these exceptional fields belong to the case $m=2$, which is handled in Subsection~\ref{Subsec:2}, and the latter two contain the nonsquare totally positive unit $\ve_3=2+\sqrt3$. 

\subsection{\texorpdfstring{$\sqrt{10}\in K$}{sqrt10 in K}}\label{Subsec:m=10} This case is more challenging than the ones in the previous subsections since $\ve_{10}$ is not totally positive, so we are not sure to which branch of the main proof the field $K$ belongs. However, the only problem with $\sqrt{10}\in K$ arises from the fact that Lemma~\ref{Lemma:10and8} allows more decompositions of $8$ and $10$ than usually. 
 Thus the only place where the main proof of nonuniversality breaks down is one spot in Subsection~\ref{Subsec:CaseIIInondiag} where this particular lemma is invoked. It means that here we only have to consider fields $K$, where all the totally positive units are squares, and forms $Q = \uqf{1} \bot Q_0$, where $2Q_0(y\e+z\f)=(2y+\beta z)^2+(2\gamma-\beta^2)z^2$ for suitable vectors $\e$ and $\f$. For the sake of brevity, denote $\Delta=2\gamma-\beta^2$. Let us stress that although $y$ and $\gamma$ do not necessarily belong to $\OK$,  both $z$ and $2\gamma-\beta^2$ always do (see Lemma~\ref{Lemma:Integrity}).

Note that all fields with $m<10$ are handled elsewhere, namely in Subsections~\ref{Subsec:m=3},~\ref{Subsec:SqrtsWithUnits},~\ref{Subsec:2} and~\ref{Subsec:5}, and the respective proofs do not require $\sqrt{10}\notin K$, since they do not depend on Lemma~\ref{Lemma:10and8}; hence, here we may assume $m=10$.

As in the main proof, we know that since $Q$ represents the nonsquare indecomposable element $\M=4+\sqrt{10}$, $Q_0$ has to represent $\M$, and thus $2Q_0$ represents $2\M$.  Invoking Lemma~\ref{Lemma:SquaresUnder2M}, we know that there is a $z\in\OK$ such that $\Delta z^2$ is either $2\M=8+2\sqrt{10}$ or $2\M-1=7+2\sqrt{10}$. Moreover, the element $\overline{\M}=4-\sqrt{10}$ has the same algebraic properties as $\M$, and hence there is also a $z'\in\OK$ such that $\Delta z'^2$ is either $2\overline{\M}=8-2\sqrt{10}$ or $2\overline{\M}-1=7-2\sqrt{10}$.

Let us go through the possibilities: If the unary form $\uqf{\Delta}$ represents both $2\M$ and $2\overline{\M}$, then $\M\overline{\M} = 4^2-(\sqrt{10})^2=6$ must be a square -- a contradiction. Similarly, if it represents $2\M$ and $2\overline{\M}-1$, then $2\M\bigl(2\overline{\M}-1\bigr) = 16-2\sqrt{10}$ should be a square, but it is not (the field $\Q\Bigl(\sqrt{16-2\sqrt{10}}\Bigr)$ is not biquadratic). The case with $2\overline{\M}$ and $2\M-1$ is analogous.

The only remaining possibility is when $\Delta z^2 = 2\M-1 = 7+2\sqrt{10}$ and $\Delta z'^2 = 2\overline{\M}-1 = 7-2\sqrt{10}$, in which case $(2\M-1)\bigl(2\overline{\M}-1\bigr)=9$ is indeed a square. Applying the norm on the equalities $\Delta z^2 = 2\M-1 = 7+2\sqrt{10}$ and  $\Delta(z^2+z'^2) = (2\M-1)+\bigl(2\overline{\M}-1\bigr)= 14$, it follows that $\norm{K/\Q}{\Delta}$ divides both $\norm{K/\Q}{2\M-1}=3^4$ and $\norm{K/\Q}{14} = 2^4\cdot 7^4$; therefore, $\norm{K/\Q}{\Delta}=\pm 1$. However, this would mean that $\Delta=2\gamma-\beta^2$ is a totally positive unit, hence a square. But then $2\M-1$ must be a square as well; according to Lemma~\ref{Lemma:2M-1isNonsquare}, this never happens.

\subsection{\texorpdfstring{$\sqrt{13}\in K$}{sqrt13 in K}}
Just as in the previous section, the only problem caused by $\sqrt{13}$ arises from another possible decomposition in Lemma~\ref{Lemma:10and8}, so we only have to fix the case of nondiagonalizable forms. In particular, we may again assume that the only totally positive units in the field are the squares of units.

Furthermore, the proofs of the cases with $m<10$ as well as $m=10$ do not require the above mentioned lemma. Thus the only potentially problematical field containing $\sqrt{13}$ with $m\neq 13$ could be $\BQ{11}{13}$; however, in this field we have a nonsquare totally positive unit $\ve_{11}=10+3\sqrt{11}$, and so Conclusion~\ref{Concl:epsilon} applies. Therefore, we can suppose $m=13$.

Denoting as before $\Delta=2\gamma-\beta^2$, we can find $z\in\OK$ such that $\Delta z^2$ is either $2\M  = 5+\sqrt{13}$ or $2\M-1=4+\sqrt{13}$ and $z'\in\OK$ such that $\Delta z'^2$ is either $2\overline{\M} = 5-\sqrt{13}$ or $2\overline{\M}-1=4-\sqrt{13}$. Again we go through all possibilities: If $\Delta z^2=2\M$ and $\Delta z'^2=2\overline{\M}$, then $\M\overline{\M} = \frac14\bigl(5^2-(\sqrt{13})^2\bigr) = 3$ must be a square, which it is not. If $\Delta z^2=2\M-1$ and $\Delta z'^2=2\overline{\M}-1$, then $(2\M-1)\bigl(2\overline{\M}-1\bigr) = \bigl(4^2-(\sqrt{13})^2\bigr) = 3$ gives a contradiction in the same way. If $\Delta z^2=2\M$ and $\Delta z'^2=2\overline{\M}-1$, then $2\M(2\overline{\M}-1) = 7 - \sqrt{13} = \bigl(\frac{\sqrt{2}-\sqrt{26}}{2}\bigr)^2$ is a square only in $\BQ{2}{13}$ where $m=2$. The case with $2\M-1$ and $2\overline{\M}$ is analogous and also collapses only in the field $\BQ{2}{13}$.

\subsection{\texorpdfstring{$\sqrt{17}\in K$}{sqrt17 in K}} \label{Subsec:17}
We use the same technique as in the previous two subsections ($\sqrt{10}\in K$ and $\sqrt{13}\in K$), as the problem with $\sqrt{17}\in K$ is again only in the use of Lemma~\ref{Lemma:10and8}, and hence it arises only when dealing with the nondiagonalizable form $Q_0$.

Similarly as above, we may put aside fields with small values of $m$: All the cases with $m\leq10$ have been solved without any application of this lemma, and hence are valid also when $\sqrt{17}$ is contained in the field. The same holds for the field $\BQ{13}{17}$. Thus it remains to consider the fields $\BQ{11}{17}$, $\BQ{14}{17}$ and $\BQ{15}{17}$; but the corresponding units $\ve_{11}=10+3\sqrt{11}$, $\ve_{14}=15+4\sqrt{14}$ and $\ve_{15}=4+\sqrt{15}$ are here totally positive and nonsquare, enabling us to use Conclusion~\ref{Concl:epsilon}. Thus, there is no harm in assuming $m=17$.

Just as before, we denote $\Delta = 2\gamma-\beta^2$ and use the nonsquareness and indecomposability of $\M=\frac12(5+\sqrt{17})$. Moreover, note that in this case $2\M-1$ is not totally positive, thus we get only one possibility from Lemma~\ref{Lemma:SquaresUnder2M}: There is a $z\in \OK$ such that $\Delta z^2 = 2\M=5+\sqrt{17}$, and also a $z'\in \OK$ such that $\Delta z'^2 = 2\overline{\M}=5-\sqrt{17}$. Putting the two equalities together, we find that $\M\overline{\M} = \frac14\bigl(5^2-(\sqrt{17})^2\bigr) = 2$ must be a square, which is a contradiction. 

\subsection{\texorpdfstring{$m=65$}{m=65}} 
The problem with $m=65$ is that in this case, for $\M=\M_{1/2}=\frac12(9+\sqrt{65})$ we have  $5\M_{1/2}=\left(\frac12\left(5+\sqrt{65}\right)\right)^2$ (see also Corollary~\ref{Cor:NonsquareM}). The only part of the main proof which fails here is if there is no nonsquare totally positive unit in $K$ and we have a nondiagonalizable form, i.e, in Subsection~\ref{Subsec:CaseIIInondiag} (see also Conclusion~\ref{Concl:CaseIIInondiag}). Thus we have to use a modified approach in this branch of the proof.

The solution is simple since we do not have to change almost anything. We consider the representations of $2$ and $5$ in exactly the same way as before, and in the place where we used $\M=\M_{1/2}$, we use the number $A=\frac12\left(25+3\sqrt{65}\right)$; the rest of the proof remains word by word the same. But we have to check the required properties of $A$ by hand:

\begin{lemma} \label{Lemma:65}
Let $K=\BQ{65}{s}$ where $t>s>65$; for the element $A=\frac12\left(25+3\sqrt{65}\right)$ the following holds:
\begin{enumerate}
    \item $A$ is totally positive,
    \item $A$ is indecomposable, 
    \item $2A\succcurlyeq \omega^2$ implies $\omega=0$, 
    \item $A$, $3A$ and $5A$ are not squares.
\end{enumerate}
\end{lemma}
\begin{proof} 
The claim (1) is trivial and (4) can be seen by checking that the three fields generated by $\sqrt{A}$, $\sqrt{3A}$ and $\sqrt{5A}$ are all of degree four but not biquadratic.

(2) In $\Q(\sqrt{65})$, $A$ is a semiconvergent of $\frac{\sqrt{65}-1}{2}$, thus it is indecomposable in $\Q(\sqrt{65})$. Since $65\equiv1\pmod4$, then except for basis (B4) we have $m=q$, thus $A$ remains indecomposable in $K$ according to Theorem~\ref{Theorem:Indecomposables}. In case of basis (B4), we have to check the inequality $\sqrt{t}>7\sqrt{65}$ to be able to use Theorem~\ref{Theorem:Indecomposables} ($7$ is the biggest coefficient in the continued fraction of $\frac{\sqrt{65}-1}{2}$). If $m$ and $s$ are coprime, then this is clearly always satisfied thanks to $t=65s$. If, however, they are not, we have $s=5m_0$ and $t=13m_0$, which means that the mentioned theorem can be used only if $\sqrt{13m_0}>7\sqrt{65}$, i.e., if $m_0>7^2\cdot5$. The indecomposability of $A$ in the remaining fields $\BQ{65}{5m_0}$ with $m_0\equiv 1 \pmod4$, $13<m_0\leq 7^2\cdot 5$ can be checked directly, e.g., by a computer program. 

(3) Denote $\omega=\frac{a}{4}+\frac{b}{4}\sqrt{65}+\frac{c}{4}\sqrt{s}+\frac{d}{4}\sqrt{t}$ for integers $a,b,c,d$, and suppose $\omega\neq0$; note that either all of $a,b,c,d$ are even, or all of them are odd. By comparing the traces of $2A$ and $\omega^2$ we obtain the inequality $25\cdot16 \geq a^2+65b^2+sc^2+td^2$. It is immediate that $|b|,|c|,|d|\leq \sqrt{\frac{25\cdot 16}{65}}<3$, which leaves the possibilities $b,c,d \in \{0, \pm1, \pm2\}$. 
It is easy to check that all $b$, $c$, $d$ cannot be zeros at once. If one of these numbers equals $\pm2$, the remaining two must be zeros, and the rest of the computation is easy. If, on the other hand, one of them is $\pm1$, then so must be the remaining two (this follows from the form of the integral basis (B4)) and $s\equiv 1\pmod4$. But then we have $25\cdot 16 \geq 1+65+s+t$; clearly, if $65$ and $s$ are coprime, then $t=65s$, which is too large. If they are not, then we have $s=5m_0$ and $t=13m_0$, which gives $334\geq 18m_0$; together with the condition $m_0\equiv 1\pmod4$ and $m_0>13$ the only theoretically possible case is $m_0=17$, i.e., the field $\BQ{65}{85}$. But one can easily check that none of the possibilities in $25+3\sqrt{65} \succcurlyeq \bigl(\frac14(1\pm\sqrt{65}\pm\sqrt{85}\pm\sqrt{221})\bigr)^2$ holds. Thus, the only $\omega\in\OK$ with the required property is $\omega=0$.
\end{proof}

\subsection{\texorpdfstring{$m=85$}{m=85}} \label{Subsec:m=85}
The problem with $m=85$ is just the same as with $m=65$, and it appears only in Subsection~\ref{Subsec:CaseIIInondiag}. The solution is very similar to that for $m=65$ (even slightly easier): In the place of the main proof where we used $\M=\M_{1/2}$, we use $\M_1 = 2\M_{1/2}-1 = 10+\sqrt{85}$ instead.

The properties of $\M_1$ for $m=85$ which we need are:
\begin{enumerate}
    \item $\M_1$ is totally positive (this is immediate),
    \item $\M_1$ is indecomposable (this follows from Proposition~\ref{Proposition:2M-1indecomposable}),
    \item $2\M_1\succcurlyeq\omega^2$ implies $\omega^2\in\{0,1\}$ (which can be checked just as in Lemma~\ref{Lemma:65}(3); note that Lemma~\ref{Lemma:SquaresUnder2M} cannot be used since here $\M \neq \M_1$),
    \item $\M_1$, $3\M_1$ and $5\M_1$ are all nonsquare in $K$ (this comes directly from Lemma~\ref{Lemma:NonsquareM}).
\end{enumerate}

Once we have these properties, we can replace $\M$ by $\M_1$ and then use the same procedure as in the main proof.

\subsection{\texorpdfstring{$\sqrt2\in K$}{sqrt2 in K}} \label{Subsec:2} 
In any field $K$ containing $\sqrt{2}$, i.e., in fields where $m=2$, we exhibit four explicit elements $\lambda_1, \lambda_2, \lambda_3, \lambda_4 \in \OKPlus$ and show that no ternary form over $\OK$ can represent these four elements at the same time. (In the next subsection we do the same for all fields containing $\sqrt5$.) The proof uses the key idea of the method of escalation developed in \cite{BH}. Exceptional cases are handled in Subsection~\ref{Subsec:Specific}.

Contrary to the rest of the paper, this subsection, as well as the next two ones, uses computer programs significantly. As we will see, in an explicit field, the problem can be reduced to a finite computation. Thus, we first find a relatively small integer $N$, such that all fields $\BQ2s$ with $s\geq N$ can be solved uniformly via computing determinants of a few hundreds or thousands of matrices (we used a program in Mathematica) -- this reduces the infinite family to a finite one. All the remaining cases can be then solved after making some individual adjustments.

To show that four elements $\lambda_1, \ldots, \lambda_4 \in\OKPlus$ cannot be represented by the same ternary form $Q$, it suffices to prove the following:

\smallskip

\textit{There exists no $4\times 4$ symmetric, totally positive semidefinite matrix with entries in $\OK$ whose diagonal elements are $\lambda_1, \ldots, \lambda_4$ and which is singular}. 

\smallskip

The reason is as follows: Suppose that all four elements are represented, i.e., $Q(\vct{\ell}_i)=\lambda_i$ for $\vct{\ell}_i\in\OK^3$, $i=1, \ldots, 4$, and consider the Gram matrix $G$ of $\vct{\ell}_1,\ldots, \vct{\ell}_4$ (defined as $G=(g_{ij})$, $g_{ij}=B_Q(\vct{\ell}_i,\vct{\ell}_j)$). As a Gram matrix of a totally positive definite form, $G$ is symmetric and totally positive semidefinite, and from the definition it is clear that the $\lambda_i$'s are its diagonal entries. Since regularity of $G$ would imply linear independence of $\vct{\ell}_1,\ldots,\vct{\ell}_4$, it follows that $G$ must be singular.

Therefore, our strategy is to check all possible totally positive semidefinite matrices $G$ with the given diagonal and to prove that they are necessarily regular. If the matrix
\[
G=\begin{pmatrix}
\lambda_1 & \rho_{12} & \rho_{13} & \rho_{14}\\
\rho_{12} & \lambda_2 & \rho_{23} & \rho_{24}\\
\rho_{13} & \rho_{23} & \lambda_3 & \rho_{34}\\
\rho_{14} & \rho_{24} & \rho_{34} & \lambda_{4}\\
\end{pmatrix}
\]
is totally positive semidefinite, Sylvester's criterion applied to $2\times 2$ subdeterminants yields the necessary condition $\rho_{ij}^2 \preccurlyeq \lambda_i\lambda_j$ for all non-diagonal entries $\rho_{ij}$. It turns out that on each of the six positions $(i,j)$, $i<j$, there are only finitely many numbers $\rho_{ij}$ which satisfy this condition (explicit lists for our choice of diagonal coefficients are provided in Lemma~\ref{Lemma:2coeff}). Thus altogether we obtain only finitely many candidates for matrix $G$. By checking their determinant, we prove them all to be regular.

\bigskip

The method described so far can be used for \emph{any} totally real field $K$ (and will be used in the next two subsections as well). Now let us start with $\sqrt{2}\in K$, i.e., $m=2$. It also means that $s$ is odd and $t=2s$. In this case, we put
\[
\lambda_1=1,\quad \lambda_2=\M=2+\sqrt{2}, \quad \lambda_3=3 \quad \text{and} \quad \lambda_4=\mathcal{S}.
\] 
Our next aim is to determine all the possible coefficients $\rho_{ij}$ which satisfy $\lambda_i\lambda_j\succcurlyeq \rho_{ij}^2$ for all $i,j$. In this case, it follows that $\rho_{12}=0$ since $\lambda_1\lambda_2 = \M =2+\sqrt{2}$ is indecomposable in $K$ and not a square, see Proposition~\ref{Prop:IndecM} and Corollary~\ref{Cor:NonsquareM}. In the same manner, we can see that $\rho_{14}=0$, which is a consequence of the indecomposability of $\S$ in $K$, see Proposition~\ref{Prop:IndecS}, and of the fact that $\S$ is not a square except for the cases when $s=3$ or $5$, see Lemma~\ref{Lemma:NonsquareS}. These cases will be resolved separately in Subsection~\ref{Subsec:Specific}. Lemma~\ref{Lemma:TrivDecomp} claims that possible values of $\rho_{13}$ belong to the set $\{0,\pm 1,\pm\sqrt{2}\}$ for $s\neq3,5$. The remaining coefficients $\rho_{ij}$ are examined in the following lemma.

\begin{lemma} \label{Lemma:2coeff}
Let $s$ be an odd square-free positive integer such that $\splitatcommas{s\neq3,5,7,11,13,15,17,19,21,29,33}$, and $K=\BQ{2}{s}$.
\begin{enumerate}
\item If $3(2+\sqrt{2})\succcurlyeq\omega^2$ where $\omega\in\OK$, then $\omega\in\{0,\pm 1,\pm(1+\sqrt{2})\}$. \label{Lemma:2coeff-1}
\item If $(2+\sqrt{2})\mathcal{S}\succcurlyeq\omega^2$ where $\omega\in\OK$, then $\omega=0$. \label{Lemma:2coeff-2}
\item If $3\mathcal{S}\succcurlyeq\omega^2$ where $\omega\in\OK$, then $\omega\in\{0,\pm 1,\pm\sqrt{2}\}$. \label{Lemma:2coeff-3}
\end{enumerate}

\end{lemma}

\begin{proof}
We will show only part (3) of this lemma, the proofs of the other statements are analogous. First of all, let $s\equiv 3 \pmod{4}$, which implies that $\S=\S_1= \cs+\sqrt{s}$. 
If 
\[
\omega=\frac{a}{2}+\frac{b}{2}\sqrt{2}+\frac{c}{2}\sqrt{s}+\frac{d}{2}\sqrt{2s}
\]
for some $a,b,c,d\in\Z$, then by comparing the traces of $3\S$ and $\omega^2$ we can conclude that 
\[
3\cs\geq\frac{a^2}{4}+\frac{b^2}{2}+\frac{c^2}{4}s+\frac{d^2}{2}s.
\]
If $\frac{s}{4}>3\cs$, then necessarily $c=d=0$ and $\omega$ belongs to the quadratic subfield $\Q(\sqrt{2})$; this occurs whenever $s\geq 168$. 
In such a case, it follows that $\omega=a'+b'\sqrt{2}$ where $a',b'\in\Z$, which means that
\[
3\cs+3\sqrt{s}\succcurlyeq a'^2+2b'^2+2a'b'\sqrt{2}.
\]
This inequality can be satisfied only if 
\[
3\cs-3\sqrt{s}\geq a'^2+2b'^2,
\]
from which follows that $a'^2+2b'^2<3$. Hence $\omega\in\{0,\pm 1,\pm\sqrt{2}\}$. It is easy to verify by a computer program that this is in fact also true for all $s<168$ except of the ones listed in the statement of this lemma. We can use similar argumentation for $s\equiv 1\pmod{4}$.
\end{proof}

Knowing all possible values of $\rho_{ij}$, we are able to construct all totally positive semidefinite matrices with the given diagonal -- so to finish our task, it suffices to check their determinants. A minor problem which has to be handled is that the value of $\S$ depends on the chosen field $K$, but this can easily be overcome:

\begin{lemma} \label{lemma:coef2}
There are no singular $4\times 4$ totally positive semidefinite matrices with diagonal entries $\lambda_1=1$, $\lambda_2=2+\sqrt{2}$, $\lambda_3=3$ and $\lambda_4=\S$ over $\OK$ for $K=\BQ{2}{s}$ where $\splitatcommas{s \neq3,5,7,11,13,15,17,19,21,29,33}$.
\end{lemma}
\begin{proof}
From Lemma~\ref{Lemma:2coeff} and the discussion above it we already know all the possible matrices $G$, so it suffices to show that none of them has determinant zero: Observe that $\S$ is the only element in the matrix which does not lie in $\Q(\sqrt{2})$. If we denote by $\Delta$ the determinant of the upper left $3\times 3$ matrix, then expansion of the determinant of $G$ along the last column gives an equality of the form $\Delta \S + \alpha = \det G$ for some $\alpha\in \Q(\sqrt{2})$.

Suppose that $\det G = 0$. By running a computer program we easily check that $\Delta \neq 0$ for all the candidate matrices, so the above equality is a nontrivial linear equation for $\S\notin \Q(\sqrt{2})$ with coefficients in $\Q(\sqrt{2})$; this is a contradiction. 
\end{proof}

By this we have proven the nonexistence of universal ternary forms over $\BQ{2}{s}$ except for eleven specific values of $s$. We postpone $s=3,5, 21, 33$ to Subsection~\ref{Subsec:Specific} and consider all the other remaining values: The only difference for $s=7,11,13,15,17,19,29$ is that we get more possibilities for the coefficient $\rho_{34}$ (i.e., the corresponding fields contain more elements $\omega$ such that $3\mathcal{S}\succcurlyeq\omega^2$), whereas the parts \eqref{Lemma:2coeff-1} and \eqref{Lemma:2coeff-2} of Lemma~\ref{Lemma:2coeff} still hold even for these values of $s$. For some possible choices of $\rho_{34}$ we cannot use the procedure introduced in the proof of Lemma~\ref{lemma:coef2}, as they do not belong to $\Q(\sqrt{2})$. However, now we work in concrete fields, and in each of them, there are only finitely many possibilities for $\rho_{34}$ (compare the traces of $3\S$ and $\rho_{34}^2$); thus, we can compute explicitly all the corresponding $4\times 4$ determinants to check that they are nonzero. This extends Lemma~\ref{lemma:coef2}, and thus also resolves the nonexistence of a universal ternary quadratic form, to all the values of $s$ except for  $\splitatcommas{s=3,5,21,33}$.

Thus, we have proven the following conclusion. 

\begin{conclusion}
In a biquadratic field $K=\BQ{2}{s}$ with $s\neq 3, 5, 21, 33$, no ternary form can simultaneously represent all the numbers $1$, $\M=2+\sqrt{2}$, $3$ and $\S$ at the same time.
\end{conclusion}

\subsection{\texorpdfstring{$\sqrt5 \in K$}{sqrt5 in K}} \label{Subsec:5}
Since we postpone the fields $\BQ25$ and $\BQ35$ (together with a few other fields) to Subsection~\ref{Subsec:Specific},  we shall require $m=5$. We use exactly the same strategy as for $m=2$, which was outlined at the beginning of Subsection~\ref{Subsec:2}. The only difference lies in the exact choice of the four elements to perform the escalation with; this time, we put
\[
\lambda_1=1,\quad \lambda_2=2, \quad \lambda_3=6+\sqrt{5} \quad \text{and} \quad \lambda_4=\mathcal{S}.
\]
Recall that we need to compute all values of $\rho_{ij}$ satisfying $\lambda_i \lambda_j\succcurlyeq \rho_{ij}^2$. Clearly, $\rho_{12}\in\{0,-1,1\}$ and $\rho_{14}=0$ except for $s=3,5$ (see Lemma~\ref{Lemma:NonsquareS}), which we have postponed to the next subsection. The sets for other coefficients are computed in the following lemma. 

\begin{lemma} \label{lemma:5coeff}
Let $s$ be a square-free positive integer such that $5\nmid s$ and $\splitatcommas{s\neq2,3,6,7,13,17,21,29,33,37,53}$, and let $K=\BQ{5}{s}$.
\begin{enumerate}
\item If $6+\sqrt{5}\succcurlyeq\omega^2$ where $\omega\in\OK$, then 
\[
\omega\in\bigg\{0,\pm1,\pm\Big(\frac{1+\sqrt{5}}{2}\bigg),\pm\bigg(\frac{1-\sqrt{5}}{2}\bigg),\pm\bigg(\frac{3+\sqrt{5}}{2}\bigg)\bigg\}.
\] 
\item If $2(6+\sqrt{5})\succcurlyeq\omega^2$ where $\omega\in\OK$, then
\[
\qquad\omega\in\bigg\{0,\pm 1,\pm 2,\pm\sqrt{5},\pm(1+\sqrt{5}),\pm\bigg(\frac{1+\sqrt{5}}{2}\bigg),\pm\bigg(\frac{1-\sqrt{5}}{2}\bigg), \pm\bigg(\frac{3+\sqrt{5}}{2}\bigg),\pm\bigg(\frac{5+\sqrt{5}}{2}\bigg)\bigg\}.
\] 
\item If $2\mathcal{S}\succcurlyeq\omega^2$ where $\omega\in\OK$, then $\omega\in\{0,-1,1\}$.
\item If $(6+\sqrt{5})\mathcal{S}\succcurlyeq\omega^2$ where $\omega\in\OK$, then
\[
\omega\in\bigg\{0,\pm1,\pm\bigg(\frac{1+\sqrt{5}}{2}\bigg),\pm\bigg(\frac{1-\sqrt{5}}{2}\bigg),\pm\bigg(\frac{3+\sqrt{5}}{2}\bigg)\bigg\}.
\]
\end{enumerate}
\end{lemma}

\begin{proof}
We can use analogous argumentation as for $m=2$: In all four cases, comparing traces yields an inequality which clearly allows only finitely many choices of $\omega$; moreover, for $s$ sufficiently large, $\omega \in \Q(\sqrt{m})$ can be proven. The rest is only a straightforward checking of inequalities, readily handled by a computer program.
\end{proof}

The previous lemma together with the knowledge of possible choices for $\rho_{12}$ and $\rho_{14}$ allows us to prove that none of the resulting Gram matrices has determinant zero by the same small trick as in Lemma~\ref{lemma:coef2}. Thus, under the assumptions of the previous lemma, we cannot find any ternary universal quadratic form over $\OK$.

If $s=7$, we obtain more choices of $\omega$ for 
$2(6+\sqrt{5})$, while the other sets of coefficients remain unchanged; for $s=29,37,53$, we get a larger set only in part (4), i.e., for $(6+\sqrt5)\S$. However, even after including these modifications, the regularity of all the resulting $4\times 4$ matrices over these concrete fields can be checked by a direct computation. Thus, we have proven the following statement.

\begin{conclusion}
In a biquadratic field $K$ where $m=5$ and $\splitatcommas{s\neq 6, 13, 17, 21, 33}$, no ternary form can simultaneously represent all the numbers $1$, $2$, $6+\sqrt{5}$ and $\S$ at the same time.
\end{conclusion}

It remains to solve the cases when $\splitatcommas{s=2,3,6,13,17,21,33}$. Of course, $s=2,3$ is absurd if $m=5$. Moreover, the cases with $s=6, 21, 33$ have already been solved in Subsection~\ref{Subsec:SqrtsWithUnits}. The only remaining fields $\BQ{5}{13}$ and $\BQ{5}{17}$ will be handled in Subsection~\ref{Subsec:Specific}.

\subsection{Remaining biquadratic fields} \label{Subsec:Specific}
For the remaining biquadratic fields, we use the procedure introduced in Subsection~\ref{Subsec:2}. We find elements $\lambda_1, \lambda_2, \lambda_3, \lambda_4 \in\OKPlus$ such that every symmetric totally positive semidefinite matrix over $\OK$ with them on the diagonal is regular, thereby showing that no ternary form can represent them all. In fact, we put $\lambda_1=1$ and conveniently choose the other $\lambda_i$'s; if possible, we prefer indecomposable integers in these fields. Consequently, by solving inequalities $\rho_{ij}^2 \preccurlyeq \lambda_i\lambda_j$ with the help of a computer (upper bounds on coefficients $\rho_{ij}^2$ are, as always, obtained by comparing traces, so there are only finitely many values to be checked), we obtain all possible non-diagonal entries of totally positive semidefinite matrices with the given diagonal. To complete the proof, it suffices to compute all the determinants. In Table~\ref{Tab:specific}, we exhibit the quadruples of $\lambda_i$'s which we used and for which all the resulting candidates for totally positive semidefinite Gram matrices are regular. It implies that even in these fields, there cannot exist a ternary universal quadratic form, and we even explicitly know four elements that are never represented by the same form. Note that $\BQ{2}{3}$ was already solved in \cite[Subsec.~5.1]{CLSTZ} by the same method; here we list the diagonal coefficients which were used in that article.  

\begin{table}[ht]
\begin{tabular}{|c|c|c|c|c|c|}
\hline 
$m$ & $s$ & $t$ & $\lambda_2$ & $\lambda_3$ & $\lambda_4$\\
\hline
\hline 
 $2$ & $3$ & $6$ & $4+\frac{5}{2}\sqrt{2}+2\sqrt{3}+\frac{3}{2}\sqrt{6}$ & $3+\sqrt{6}$ & $3-\frac{3}{2}\sqrt{2}-\sqrt{3}+\frac{1}{2}\sqrt{6}$\\
\hline
$2$ & $5$ & $10$ & $2+\sqrt{2}$ & $3$ & $\frac{5}{2}+\frac{1}{2}\sqrt{2}+\frac{1}{2}\sqrt{5}+\frac{1}{2}\sqrt{10}$\\
\hline
$2$ & $21$ & $42$ & $2+\sqrt{2}$ & $\frac{5}{2}+\frac{1}{2}\sqrt{21}$ & $5+\frac{5}{2}\sqrt{2}+\sqrt{21}+\frac{1}{2}\sqrt{42}$\\
\hline
$2$ & $33$ & $66$ & $2+\sqrt{2}$ & $\frac{7}{2}+\frac{1}{2}\sqrt{33}$ & $6+\sqrt{33}$\\
\hline
$3$ & $5$ & $15$ & $2+\sqrt{3}$ & $\frac{5}{2}+\frac{1}{2}\sqrt{3}+\frac{1}{2}\sqrt{5}+\frac{1}{2}\sqrt{15}$ & $3+\frac{1}{2}\sqrt{3}+\frac{1}{2}\sqrt{15}$\\
\hline
$5$ & $13$ & $65$ & $\frac{5}{2}+\frac{1}{2}\sqrt{13}$ & $4+\sqrt{13}$ & $\frac{13}{4}+\frac{3}{4}\sqrt{5}+\frac{3}{4}\sqrt{13}+\frac{1}{4}\sqrt{65}$\\
\hline
$5$ & $17$ & $85$ & $\frac{5}{2}+\frac{1}{2}\sqrt{17}$ & $\frac{29}{2}+\frac{7}{2}\sqrt{17}$ & $\frac{13}{4}+\frac{1}{4}\sqrt{5}+\frac{1}{4}\sqrt{17}+\frac{1}{4}\sqrt{85}$ \\
\hline

\end{tabular}
\caption{Chosen diagonal coefficients}\label{Tab:specific}
\end{table}
By this, the proof of Theorem~\ref{Thm:Main} is finished.

\section*{Open questions}
Although we achieved our goal and proved that no totally real biquadratic field admits a ternary universal quadratic form, there are many very natural open problems. One may want to generalize this result to non-classical forms; however, the corresponding problem is, to our best knowledge, still open even for the much simpler case of quadratic fields. More interesting question is, which (if any) biquadratic fields admit a quaternary universal form, and more generally, what is the lowest number of variables $n$ such that there is a biquadratic field which admits an $n$-ary universal form. And for which $n$ are there infinitely many such fields? Tools developed in this paper should prove useful in investigating these problems. A bolder generalization of our result, namely the nonexistence of universal ternary forms for all totally real multiquadratic fields, would provide strong evidence towards Kitaoka's conjecture.

Theorem~\ref{Thm:Inde} significantly extended our knowledge about indecomposable integers in biquadratic fields; however, there is a very natural open question: Is it possible for an element from the subfields $\Q(\sqrt{m})$ or $\Q(\sqrt{s})$ to decompose in the biquadratic extension, or can this happen only for the subfield $\Q(\sqrt{t})$, as evidence seems to suggest?

%------------------------------------------------------------------------ ----------------------

\section*{Acknowledgment}
The authors wish to express their thanks to Martin \v{C}ech, who participated in the early stages of the research.    
A special thanks belongs to Wai Kiu Chan for his quick answer regarding Appendix~\ref{App:sqrt10}.
The authors are also greatly indebted to V\'it\v{e}zslav Kala for suggesting the problem and for his comments, to Ond\v{r}ej Draganov for his help with programs in Mathematica, and to Alexander Sl\'avik for one brief but fruitful consultation. 

%----------------------------------------------------------------------------------------------

%-----------------------------------------------------------------------
\newpage
\appendix

\section{Ternary forms over the quadratic field \texorpdfstring{$\Q(\sqrt{10})$}{Q(sqrt10)}} \label{App:sqrt10}

The main source of inspiration for this article is the beautiful paper \cite{CKR}, which proves that real quadratic fields except for $\Q(\sqrt2)$, $\Q(\sqrt3)$ and $\Q(\sqrt5)$ admit no universal ternary forms, and fully describes universal ternary forms over those three fields. We have found two discrepancies in the proof of nonuniversality given there; this appendix is devoted to their corrections. Both of the problems arise when dealing with a nondiagonalizable form $\uqf{1} \perp Q_0$ (denoted there as $\uqf{1} \perp L_0$; the analogous part in this paper is Subsection~\ref{Subsec:CaseIIInondiag}) over a totally real quadratic field $F$. Just as in our proof, $\e\in\OF^2$ is a vector such that $Q_0(\e)=2$, and $\f\in F^2$ is a \uv{complementing vector} such that $\OF = \indx{\e}\e\oplus\indx{\f}\f$ and $\indx{\e}=\indx{\f}^{-1}$. 

\bigskip

The first problem was the claim that $\gamma=Q_0(\f) \in \OF$; that is not always the case, as we shall see in Example~\ref{Example:sqrt10}. However, for the rest of the proof, it is sufficient  that $2\gamma \in \OF$; this we have proven in Lemma~\ref{Lemma:Integrity}. (We also show there that $\beta = B_{Q_0}(\e,\f)$ belongs to $\OF$, which was stated but not proven in \cite{CKR}.)

\begin{example} \label{Example:sqrt10}
In $F=\Q(\sqrt{10})$, consider the binary form 
\begin{equation} \label{Eq:sqrt10QF}
Q_0(v,w) = 7v^2 - 2 \cdot 3\sqrt{10}vw + 13 w^2.
\end{equation}
Note that it is totally positive definite, since $7 \in \OFPlus$ and $\det Q_0 = 7 \cdot 13 - (3\sqrt{10})^2 = 1 \in \OFPlus$. We show that $Q_0$ is not diagonalizable: If it is, then by comparing determinants it follows that $Q_0 \cong \bqf{1}{1}$ (using the fact that, in this field, $\UFPlus = \UFctv$); however, $Q_0$ represents $7$ while it is easy to check that $\bqf{1}{1}$ does not.

For $\e = (\sqrt{10},2)^{\mathrm{T}}$ we see that $Q_0(\e)= 2$
and $\indx{\e} = (\sqrt{10}, 2)^{-1} = \bigl(\frac{\sqrt{10}}{2}, 1\bigr)$.
Set $\f = \bigl(1, 1 + \frac{\sqrt{10}}{2}\bigr)^{\mathrm{T}}$; clearly $\indx{\f} = (\sqrt{10},2)$, so indeed $\indx{\e}\indx{\f} = \OF$, and it is routine to check $\indx{\e}\e + \indx{\f}\f = \OF^2$. Further, we compute
\[ \gamma = Q_0(\f) = \frac{45}{2} + 7\sqrt{10}, 
\hspace{5mm}
\beta = B_{Q_0}(\e,\f) = -4 - \sqrt{10};\]
obviously, $\gamma\notin\OF$ (but $2\gamma \in \OF$) and $\beta\in\OF$.

On a side note, there are altogether four vectors which represent $2$, namely $\pm (\sqrt{10}, 2)^{\mathrm{T}}$ and $\pm (4,\sqrt{10})^{\mathrm{T}}$; the property $\indx{\e} = (\sqrt{10},2)^{-1}$ holds for all of them.
\end{example}

\bigskip

The second problem in the paper \cite{CKR} is that the proof is not completely correct in the case of $F=\Q(\sqrt{10})$. The problematic claim is: \emph{If the number $10$ is represented by the form $2Q_0$ as $2Q_0(y\e+z\f) = (2y+\beta z)^2 + (2\gamma-\beta^2)z^2$, then $(2y+\beta z)^2$ is necessarily a square of an integer} (compare with Lemma~\ref{Lemma:10and8}). This is not true; it can also happen that $(2y+\beta z)^2 = (\pm \sqrt{10})^2$.

Note that $2y+\beta z = \pm\sqrt{10}$ and $(2\gamma-\beta^2)z^2 = 0$ if and only if $z=0$ and $y = \pm \frac{\sqrt{10}}{2}$. Since $y \in \indx{\e}$, this implies $\indx{\e}\neq\OF$. Moreover, $(\sqrt{10}, 2)$ is a prime ideal satisfying $(\sqrt{10}, 2)^2 = (2)$; since the square of the ideal $\indx{\f} = \indx{\e}^{-1}$ divides $(2)$ and here we have $\indx{f}\neq\OF$, it follows that necessarily $\indx{\f} = (\sqrt{10},2)$. Then $\indx{\e} = (\sqrt{10},2)^{-1} = \bigl(\frac{\sqrt{10}}{2},1\bigr)$, hence indeed $\pm \frac{\sqrt{10}}{2}\in\indx{e}$, and the vector $\pm\frac{\sqrt{10}}{2}\e + 0\f$ belongs to $\OF^2$.

\begin{remark} \label{Rem:sqrt10Ie}
In the above paragraph, we have actually proven the following: If $Q_0$ is a (totally positive definite, classical) binary form over $\Q(\sqrt{10})$ and $\vct{e}\in\OF^2$ is such that $Q_0(\vct{e})=2$ and $\indx{e}\neq\OF$, then $\indx{e}=(\sqrt{10},2)^{-1}$.
\end{remark}

Now we proceed to illustrate the problem on the form $Q_0$ given in the example above.

\begin{example}
Let us continue with Example~\ref{Example:sqrt10}; we consider the form $Q_0$ given by \eqref{Eq:sqrt10QF}, $\vct{e}=(\sqrt{10},2)^{\mathrm{T}}$ and $\vct{f}=\bigl(1, 1+\frac{\sqrt{10}}{2}\bigr)^{\mathrm{T}}$. For $y\in\indx{e}$ and $z\in\indx{f}$, we compute
\[Q_0(y\e+z\f) = 2y^2 + 2\beta yz + \gamma z^2 = 2y^2 + 2 (-4-\sqrt{10})yz  + \Bigl(\frac{45}{2} + 7\sqrt{10}\Bigr)z^2;\]
thus,
\[2Q_0(y\e+z\f) = (2y+\beta z)^2 + (2\gamma-\beta^2)z^2 = \bigl(2y + (-4-\sqrt{10})z\bigr)^2 + (19+6\sqrt{10})z^2.\]
By plugging in $y = \pm \frac{\sqrt{10}}{2}$ and $z = 0$, we indeed get $10 = (\sqrt{10})^2 + 0$.
\end{example}

Actually, for an arbitrary form $Q_0$, if $2$ is represented by $Q_0$ by a vector $\e$ such that $\indx{\e} \neq \OF$, then $5$ is automatically represented by $Q_0$ as well, because we have just seen that $\frac{\sqrt{10}}{2}\e \in \OF^2$, and 
\[
Q_0\biggl(\frac{\sqrt{10}}{2}\e\biggr) = \biggl(\frac{\sqrt{10}}{2}\biggr)^2 \cdot Q_0(\e) = \frac{10}{4} \cdot 2 = 5.
\]
Therefore, if $\indx{\e} \neq \OF$, no new information can be gained by considering representations of $10$ by $2Q_0$. Instead of that, we have to consider another element of $\OFPlus$ to show the nonexistence of a universal form.

\begin{proposition}
No ternary form $\uqf{1} \perp Q_0$ over $F=\Q(\sqrt{10})$ such that $Q_0(\e) = 2$ for a vector $\e\in\OF^2$ with $\indx{e}\neq\OF$ is universal; in fact, such a form never represents $\M = 4 + \sqrt{10}$. 
\end{proposition}

\begin{proof}
Since $\indx{e}\neq\OF$, Remark~\ref{Rem:sqrt10Ie} implies $\indx{\e} = (\sqrt{10},2)^{-1}$.  Suppose that $\uqf{1} \perp Q_0$ represents $\M$. Just as in the proof of Theorem~\ref{Thm:Main}, since $\M$ is indecomposable and nonsquare, it has to be represented by the binary form $Q_0$, and therefore $2\M$ is represented by $2Q_0$; using the representation $2Q_0(y\e + z\f) = (2y+\beta z)^2 + (2\gamma-\beta^2)z^2$, we see that $2\M$ decomposes as a sum of a square and a totally nonnegative number. Lemma~\ref{Lemma:SquaresUnder2M} (or a simple direct computation) shows that the square is either $0$ or $1$. However, since $2 \in \indx{\e}^{-2}$, $y \in \indx{\e}$, $\beta\in\OF$ and $z \in \indx{\e}^{-1}$, we get $2y+\beta z \in \indx{\e}^{-1} = (\sqrt{10},2)$, which is a nontrivial ideal, so it cannot contain $\pm 1$.

Therefore, the only choice is $2y+\beta z = 0$ and $(2\gamma-\beta^2)z^2 = 2\M$; denote $\Delta = 2\gamma-\beta^2$. Observe that since $z\in (\sqrt{10},2)$, we get $\norm{F/\Q}{z}\in 2\Z$. It is also easy to compute the norm $\norm{F/\Q}{\M}=\norm{F/\Q}{4+\sqrt{10}}=6$. Now let us apply norms on the equation $\Delta z^2 = 2\M$: We obtain 
\[
\norm{F/\Q}{\Delta}\norm{F/\Q}{z}^2 = 4\cdot 6.
\]
Since this equals $2^3\cdot 3$, it means that $\norm{F/\Q}{z}$ can only be $\pm1$ or $\pm2$; since the former possibility is excluded by our observation that the norm of $z$ is even, it must hold that $\norm{F/\Q}{z} = \pm 2$. However, there is no element of norm $\pm 2$ in $F$: The equality $x^2-10y^2=\pm 2$ would mean $x^2 \equiv \pm 2 \pmod{5}$, which is a contradiction.
\end{proof}

%----------------------------------
\section{Units in biquadratic fields}
\label{Subsec:Units}

In this appendix, we present some results on units in totally real biquadratic fields. In particular, we are interested in the existence of a nonsquare totally positive unit. If such a unit exists, then  one of the simplest parts of the proof of Theorem~\ref{Thm:Main} contained in Subsections~\ref{Subsec:CaseI} and~\ref{Subsec:CaseIInonsquare} is applicable. Indeed, as we will see, this is often the case.

The units of a quadratic field are well understood: A quadratic field $\Q(\sqrt{n})$ contains a nonsquare totally positive unit if and only if its fundamental unit $\ve_n$ has (quadratic) norm $+1$ (in which case \emph{every} unit is either totally positive or totally negative). That happens for almost all fields -- a necessary condition for $\norm{\Q(\sqrt{n})/\Q}{\ve_n} = -1$ is that $n$ is divisible by no prime of the form $4k+3$, because clearly $-1$ must be a quadratic residue modulo every divisor of $n$. 

The situation in a biquadratic field is a bit more difficult (see, e.g., \cite{MU}). Note that a totally positive fundamental unit of a quadratic subfield may become a square in the biquadratic field; in such a case, its square root is not totally positive anymore (this follows from Lemma~\ref{Lemma:MartinSquares}(\ref{Lemma:MartinSquaresItem1}), since no element of zero trace can be totally positive).

We start by showing that in most biquadratic fields, at least one nonsquare totally positive unit exists. Remember that the meaning of $m,s,t$ and $n_1,n_2,n_3$ is fixed by Convention \ref{Conv}.

\begin{lemma}
\label{Lemma:RootsOfUnits} 
If $K=\BQ{m}{s}$, then:
 \begin{itemize}
     \item the fundamental unit of $\Q(\sqrt{m})$ is not a square in $K$;
     \item if $m\neq 2$, then the fundamental unit of $\Q(\sqrt{s})$ is not a square in $K$.
 \end{itemize}
 More generally, let $K$ be a biquadratic number field. If $\gcd(n_2,n_3)\geq 3$, then the fundamental unit of $\Q(\sqrt{n_1})$ is not a square in $K$. 
\end{lemma}
\begin{proof}
The last, most general statement follows directly from Corollary~\ref{Corollary:BecomingSquare}: Since $\gcd(n_2,n_3)\geq3$ is square-free, it has some nontrivial odd divisor. However, no unit can have a nontrivial rational integer divisor.

The claims about $\Q(\sqrt{m})$ and $\Q(\sqrt{s})$ are simple corollaries: If $\gcd(n_2,n_3)\leq 2$, then it holds that $n_1=\frac{n_2n_3}{\gcd(n_2,n_3)^2} > n_2,n_3$ (i.e., $n_1=t$) unless one of $n_2$, $n_3$ equals $2$.
\end{proof}
Note that although Lemma~\ref{Lemma:RootsOfUnits} shows that in most cases, $\ve_m$ as well as $\ve_s$ are nonsquare units (and we already know that in most cases they are totally positive), it does \emph{not} mean that then there are at least two independent nonsquare totally positive units in $\OK$. It is quite possible that $\ve_m\ve_s$ is a square.

Another important piece of information for the proof of Theorem~\ref{Thm:Main} is whether $2\ve$ is a square or not. Note that since in a quadratic field the fundamental unit differs from any other positive nonsquare unit only by multiplication by a square, the following lemma actually generalizes to any positive nonsquare unit:

\begin{lemma}
\label{Lemma:RootsOf2epsilon}
Let $F=\Q(\sqrt{n_1})$ and $\ve\in F$ be the fundamental unit.
\begin{itemize}
    \item If $n_1\neq 2$, then $2\ve$ is a square in $F$ if and only if $\OK[F]$ contains an element of (quadratic) norm $\pm2$ with integer coefficients, i.e., if and only if at least one of the Pell's equations $x^2-n_1y^2=\pm2$ is solvable.
    \item Suppose $2\ve$ is not a square in $F$, and the biquadratic field $K \supseteq F$ satisfies $\gcd(n_2,n_3)\geq 3$. Then $2\ve$ is not a square in $K$ either.
\end{itemize}
\end{lemma}
\begin{proof}
The second part is a simple consequence of Corollary~\ref{Corollary:BecomingSquare}
(clearly $2\ve$ cannot have any odd rational integer divisors). As for the first one, start by observing that for an element $\frac{a}{2}+\frac{b}{2}\sqrt{n_1}$ with $a,b$ odd integers, $n_1 \equiv 1 \pmod{4}$, its square is of the same form. Thus, since $2\ve$ has always integer coefficients, its potential square root must have the form $a'+b'\sqrt{n_1}$ for some $a',b'\in\Z$.

If $\sqrt{2\ve}$ exists in $\OK[F]$, then clearly $\norm{F/\Q}{\sqrt{2\ve}}=\pm2$, which proves one implication. As for the other, take $\alpha=x+\sqrt{n_1}y$ where $\norm{F/\Q}{\alpha}=x^2-n_1y^2=\pm2$. Then \[\alpha^2 = (x^2+n_1y^2)+2xy\sqrt{n_1} = (\pm2+2n_1y^2)+2xy\sqrt{n_1},\] 
thus $\frac{1}{2}\alpha^2 \in\OK[F]$. Moreover, $\frac{1}{2}\alpha^2$ is a unit, thus $\alpha^2 = 2\ve^k$ for some $k\in\Z$ (it cannot be $-2\ve^k$, since $\alpha^2$ is positive whereas $-2\ve^k$ is negative due to $\ve$ being fundamental). Clearly, if $k$ was even, then $2$ would be a square, i.e., $n_1=2$. Therefore $k=2l+1$ and we have $2\ve = (\alpha\ve^{-l})^2$.
\end{proof}

From the previous two lemmas we can make a simple but strong conclusion: Suppose that in $K=\BQ{m}{s}$, the square-free integer $m$ satisfies the following: The Pell's equation $x^2-my^2 = a$ has no solution for all three right-hand sides $a=-1$, $a=-2$ and $a=2$. Then the fundamental unit of $\Q(\sqrt{m})$, $\ve_m$, is totally positive and neither $\ve_m$ nor $2\ve_m$ is a square in $K$.

By considering quadratic residues, it is easy to see that a sufficient condition for non-solvability of all three equations is that $m$ is divisible by
\begin{itemize}
    \item at least one prime of the form $4k+3$ (because of $a=-1$), and
    \item at least one prime of the form $8k+3$ or $8k-3$ (due to $a=2$), and
    \item at least one prime of the form $8k-1$ or $8k-3$ (regarding $a=-2$).
\end{itemize}
All in all, the only potentially \uv{bad-behaving} $m$'s are those which contain in their prime decomposition only primes $2$, $8k+1$ and one more residue class modulo $8$. Clearly, the density of such integers is zero. 

Of course, it is possible to continue this examination by deriving some analogy of Lemmas~\ref{Lemma:RootsOfUnits} and~\ref{Lemma:RootsOf2epsilon} for fields containing $\sqrt2$; however, the already presented results are sufficient to illustrate that fields not containing any nonsquare totally positive unit are rare and that mostly this unit multiplied by $2$ is not a square either. We conclude this part by showing that for some integral bases, there is always a nonsquare totally positive unit in $K$.

\begin{corollary}
If $K$ contains no nonsquare totally positive unit, then either $m=2$, or the integral basis is one of (B2) and (B4a).
\end{corollary}
\begin{proof}
If $m\neq 2$, by Lemma~\ref{Lemma:RootsOfUnits} neither $\ve_m$ nor $\ve_s$ is a square. Thus, for a contradiction, it suffices to show that at least one of them is totally positive. It turns out that for all bases except (B2) and (B4a), at least one of the numbers $m$ and $s$ contains a prime divisor of the form $4k+3$, which shows that the corresponding field contains no unit of norm $-1$.
\end{proof}

%----------------------------------------------------------------------------------------------
\newpage
\section{Scheme of the main proof} \label{App}
\rotatebox{90}{
\begin{forest}for tree=draw, for tree={align=center, inner sep=3}
[\hyperlink{Sec:NonspecialCases}{{\Large \textbf{Main cases}}}, s sep=5mm, l sep=15mm, draw={black,thick}, inner sep=5, for tree={s sep=3.5mm}
    [{\hyperlink{Subsec:CaseI}{Case (I)} \\ $\abs{\sfrac{\UKPlus}{\UKctv}}>2$}
        [,phantom
            [,phantom]
            [,phantom]
        ]
        [,phantom
            [,phantom]
            [,phantom]
        ]
    ]
    [{\hyperlink{Subsec:CaseIInonsquare}{Case (II)} \\ $\abs{\sfrac{\UKPlus}{\UKctv}}=2$\\ \scriptsize{$Q\cong\tqf{1}{\ve}{\gamma}$}}
        [\hyperlink{Subsec:CaseIInonsquare}{$2\ve\neq\square$}
            [,phantom]
            [,phantom]
            %[\scriptsize{$\uqf{\gamma}$ repr. $1$}]
			%[\scriptsize{$\uqf{\gamma}$ repr. $2$}]
        ]
        [\hyperlink{Subsec:CaseIIsquare}{$2\ve=\square$}
			[\hyperlink{CaseIIsquare-not35}{\scriptsize{$m\neq3,5$}}]
			[\hyperlink{CaseIIsquare-35}{\scriptsize{$m=3,5$}}]
        ]
    ]
    [{\hyperlink{Subsec:CaseIIIdiag}{Case (III)} \\ $\abs{\sfrac{\UKPlus}{\UKctv}}=1$\\ \scriptsize{$Q\cong \uqf1\bot Q_0$}}
		[{\hyperlink{Subsec:CaseIIIdiag}{$Q_0$ diagonalizable}\\ \scriptsize{$Q\cong\tqf{1}{\beta}{\gamma}$}}
				[\hyperlink{CaseIIIdiag-1}{\scriptsize{$\bqf{\beta}{\gamma}$ repr. $1$}}]
				[\hyperlink{CaseIIIdiag-2}{\scriptsize{$\bqf{\beta}{\gamma}$ repr. $2$}}
				]
		]
		[{\hyperlink{Subsec:CaseIIInondiag}{$Q_0$ nondiagonalizable}\\ \scriptsize{$Q_0(\vct{e})=2$}}
			[{\hyperlink{CaseIIInondiag-a}{(a) $\indx{e}=\OK$} \\ \scriptsize{$2Q_0$ repr. $8$ or $10$}}
				    [\hyperlink{CaseIIInondiag-a-1}{\scriptsize{(ii): $1+9$, }}\vspace{-2mm}\\ 
				    \hyperlink{CaseIIInondiag-a-1}{\scriptsize{(iv): $9+1$,}}\vspace{-2mm}\\
				    \hyperlink{CaseIIInondiag-a-1}{\scriptsize{(vii): $4+4$}}
				    ]
					[\hyperlink{CaseIIInondiag-a-2}{\scriptsize{(vi): $1+7$}}
					]
					[\hyperlink{CaseIIInondiag-a-1}{\scriptsize{(i): $0+10$, }}\vspace{-2mm}\\ 
				    \hyperlink{CaseIIInondiag-a-1}{\scriptsize{(iii): $4+6$,}}\vspace{-2mm}\\
				    \hyperlink{CaseIIInondiag-a-1}{\scriptsize{(v): $0+8$}}
				    ]
			]
			[{\hyperlink{CaseIIInondiag-b}{(b) $\indx{e}\neq\OK$}\\ \scriptsize{$2Q_0$ repr. $8$ or $10$}}
					[\hyperlink{CaseIIInondiag-a-1}{\scriptsize{(ii): 1+9, }}\vspace{-2mm}\\ 
				    \hyperlink{CaseIIInondiag-a-1}{\scriptsize{(iv): 9+1,}}\vspace{-2mm}\\
				    \hyperlink{CaseIIInondiag-a-1}{\scriptsize{(vii): 4+4}}
				    ]
				    [\hyperlink{CaseIIInondiag-a-2}{\scriptsize{(vi): 1+7}}
					]
					[\hyperlink{CaseIIInondiag-a-1}{\scriptsize{(i): 0+10, }}\vspace{-2mm}\\ 
				    \hyperlink{CaseIIInondiag-a-1}{\scriptsize{(iii): 4+6,}}\vspace{-2mm}\\
				    \hyperlink{CaseIIInondiag-a-1}{\scriptsize{(v): 0+8}}
				    ]
			]
		]
	]
]
\end{forest}
}

\end{document}